\documentclass[reqno, 10pt]{amsart}

\usepackage[parfill]{parskip}    
\usepackage{amsrefs, amsmath}
\usepackage{amsfonts, mathrsfs}
\usepackage{amssymb}
\usepackage{amscd, mathtools}
\usepackage{hyperref}
\usepackage{cleveref, multirow} 
\usepackage{fullpage, color}
\usepackage[labelformat=empty]{subfig}
\usepackage{booktabs}
\usepackage{xypic, dynkin-diagrams}
\usepackage{enumerate, longtable}
\usepackage{makecell}
\usepackage{IEEEtrantools}
\usepackage{bbm}

\usepackage{ulem}

 \usepackage{relsize}
\usepackage{tikz-cd}

\definecolor{colour1}{rgb}{0.79216,0.12156,0.7}
\definecolor{colour2}{rgb}{0.8,0.5,0.9}
\definecolor{colour3}{rgb}{0.6,0.8,0.7}

\newcommand{\tikzline}[2]{\tikz[baseline=-0.5ex]\draw[colour1] (0,0) -- (0.6, 0);}

\numberwithin{equation}{section}

\BibSpec{article}{%
+{}{\sc \PrintAuthors} {author}
+{,}{ \textrm} {title}
+{,}{ \textit} {journal}
+{,}{ } {volume}
+{}{ \IfEmptyBibField{journal}{}{\PrintDate{date}}} {transition}
+{,}{ } {pages}
+{,}{ } {status}
+{.}{} {transition}
+{}{ Preprint available at \tt} {eprint}
+{.}{} {transition}
}

\BibSpec{book}{%
+{}{\sc \PrintAuthors} {author}
+{,}{ \textit} {title}
+{.}{ } {series}
+{,}{ } {volume}
+{.}{ } {publisher}
+{,}{ } {place}
+{,}{ } {date}
+{.}{} {transition}
}

\BibSpec{collection.article}{
+{} {\sc \PrintAuthors} {author}
+{,} { \textrm } {title}
+{,} { in \it \PrintConference} {conference}
+{,} { pp.~} {pages}
+{.} {\PrintBook} {book}
+{,} { \PrintDateB} {date}
+{.} {} {transition}
}

\BibSpec{innerbook}{
+{.} { \emph} {title}
+{.} { } {part}
+{:} { \emph} {subtitle}
+{.} { } {series}
+{,} { } {volume}
+{.} { Edited by \PrintNameList} {editor}
+{.} { Translated by \PrintNameList}{translator}
+{.} { \PrintContributions} {contribution}
+{.} { } {publisher}
+{.} { } {organization}
+{,} { } {address}
+{,} { \PrintEdition} {edition}
+{,} { \PrintDateB} {date}
+{.} { } {note}
}

\allowdisplaybreaks[3]

\crefname{equation}{Eq.}{Eqs.}
\crefname{eqnarray}{Eq.}{Eqs.}
\crefname{algo}{Algorithm}{Algorithms}
\crefname{conj}{Conjecture}{Conjectures}
\crefname{lem}{Lemma}{Lemmas}
\crefname{thm}{Theorem}{Theorems}
\crefname{claim}{Claim}{Claims}
\crefname{rmk}{Remark}{Remarks}
\crefname{prop}{Proposition}{Propositions}
\crefname{section}{Section}{Sections}
\crefname{appendix}{Appendix}{Appendices}
\crefname{cor}{Corollary}{Corollaries}
\crefname{figure}{Figure}{Figures}
\crefname{table}{Table}{Tables}
\crefname{example}{Example}{Examples}
\crefname{prob}{Problem}{Problems}
\crefname{assm}{Assumption}{Assumptions}
\crefname{defn}{Definition}{Definitions}

\DeclareMathOperator{\tr}{Tr}

\newcommand{\de}{{\partial}}

\newcommand{\rd}{\mathrm{d}}

\newcommand{\re}{\mathrm{e}}

\newcommand{\bbV}{\mathbb{V}}

\newcommand{\bbA}{\mathbb{A}}

\newcommand{\bbN}{\mathbb{N}}

\newcommand{\bbZ}{\mathbb{Z}}
\newcommand{\bbR}{\mathbb{R}}
\newcommand{\bbC}{\mathbb{C}}
\newcommand{\bbP}{\mathbb{P}}

\newcommand{\bbQ}{\mathbb{Q}}

\def\bary{\begin{array}} 
\def\eary{\end{array}} 
\def\ben{\begin{enumerate}} 
\def\een{\end{enumerate}}
\def\bit{\begin{itemize}} 
\def\eit{\end{itemize}}
\def\nn{\nonumber} 

\newcommand{\cO}{\mathcal{O}}
\newcommand{\cT}{\mathcal{T}}

\newcommand{\cP}{\mathcal{P}}
\newcommand{\cC}{\mathcal{C}}

\newcommand{\LL}{\mathcal{L}}

\newcommand{\cK}{\mathcal{K}}
\newcommand{\cN}{\mathcal{N}}
\newcommand{\cW}{\mathcal{W}}
\newcommand{\cG}{\mathcal{G}}
\newcommand{\cA}{\mathcal{A}}
\newcommand{\HH}{\mathcal{H}}

\newcommand{\cI}{\mathcal{I}}

\newcommand{\RR}{\mathcal{R}}

\newcommand{\cV}{\mathcal{V}}

\newcommand{\cM}{\mathcal M}
\newcommand{\cQ}{\mathcal Q}

\def\beq{\begin{equation}}                     %
\def\eeq{\end{equation}}                       %
\def\bea{\begin{eqnarray}}                     
\def\eea{\end{eqnarray}}
\def\bary{\begin{array}} 
\def\eary{\end{array}} 
\def\ben{\begin{enumerate}} 
\def\een{\end{enumerate}}
\def\bit{\begin{itemize}} 
\def\eit{\end{itemize}}
\def\nn{\nonumber} 
\def\de {\partial}

\def\a{\alpha}
\def\b{\beta}

\def\d{\delta}

\def\Res{\mathrm{Res}}

\theoremstyle{plain}

\newtheorem{thm}{Theorem}[section]
\newtheorem{lem}[thm]{Lemma}

\newtheorem{prop}[thm]{Proposition}

\newtheorem*{conj*}{Conjecture}
\newtheorem{cor}[thm]{Corollary}
\newtheorem*{cor*}{Corollary}

\theoremstyle{definition}

\newtheorem*{rem*}{Remark}

\newtheorem*{rems*}{Remarks}
\newtheorem{rmk}[thm]{Remark}
\newtheorem{example}{Example}[section]

\newcommand{\GITl}[1]{\backslash \!\! \backslash _{\kern-.2em #1 \kern0.1em}}
\newcommand{\GIT}[1]{/\!\!/_{\kern-.2em #1 \kern0.1em}}

\renewcommand{\l}{\left}
\renewcommand{\r}{\right}
\newcommand{\bra}{\left\langle}
\newcommand{\ket}{\right\rangle}

\def\bred{\begin{color}{red}}
\def\ered{\end{color}}

\def\bes{\begin{subequations}}
\def\ees{\end{subequations}}

\usepackage{tikz}
\usetikzlibrary{decorations.pathreplacing,angles,quotes}

\newcommand\Tr{\operatorname{Tr}}

\newtheorem{theorem}{Theorem}[section]

\newtheorem{lemma-definition}[theorem]{Lemma-Definition}

\theoremstyle{definition}

\newtheorem{defn}[theorem]{Definition}

\theoremstyle{remark}

\numberwithin{equation}{section}
\numberwithin{figure}{section}

\renewcommand {\max} {{\operatorname{max}}}

\newcommand {\ord}  {\operatorname{ord}}

\newcommand {\PGL}  {\operatorname{PGL}}

\def\mydate{\ifcase\month \or January\or February\or March\or
April\or May\or June\or July\or August\or September\or October\or 
November\or December\fi \space\number\day,\space\number\year}

\usetikzlibrary{chains}

\tikzset{node distance=2em, ch/.style={circle,draw,on chain,inner sep=2pt},chj/.style={ch,join},every path/.style={shorten >=4pt,shorten <=4pt},line width=1pt,baseline=-1ex}

\usepackage{calc}
\newsavebox\CBox
\newcommand\hcancel[2][1pt]{%
  \ifmmode\sbox\CBox{$#2$}\else\sbox\CBox{#2}\fi%
  \makebox[0pt][l]{\usebox\CBox}%
  \rule[0.5\ht\CBox-#1/2]{\wd\CBox}{#1}}

\begin{document}

\title{
  Mirror symmetry for extended affine Weyl groups
  }

\author{Andrea Brini}
\address{\tiny School of Mathematics and Statistics,
University of Sheffield, S11 9DW, Sheffield, United Kingdom. \\ on leave from CNRS, DR 13, Montpellier, France.
}
\email{a.brini@sheffield.ac.uk}
\author{Karoline van Gemst}
\address{\tiny School of Mathematics and Statistics,
University of Sheffield, S11 9DW, Sheffield, United Kingdom
}
\email{kvangemst1@sheffield.ac.uk}

\thanks{This project has been supported by the Engineering and Physical Sciences
  Research Council under grant agreement ref.~EP/S003657/2.}

\begin{abstract}
We give a uniform, Lie-theoretic mirror symmetry construction for the Frobenius manifolds defined by Dubrovin--Zhang in \cite{MR1606165} on the orbit spaces of extended affine Weyl groups, including exceptional Dynkin types. The B-model mirror is given by a one-dimensional Landau--Ginzburg superpotential constructed from a suitable degeneration of the family of spectral curves of the affine relativistic Toda chain for the corresponding affine Poisson--Lie group. As applications of our mirror theorem we
give closed-form expressions for the flat coordinates of the Saito metric and the Frobenius prepotentials in all Dynkin types, compute the topological degree of the Lyashko--Looijenga mapping for certain higher genus Hurwitz space strata, and construct hydrodynamic bihamiltonian hierarchies (in both Lax--Sato and Hamiltonian form) that are root-theoretic generalisations of the long-wave limit of the extended Toda hierarchy. 

\end{abstract}

\maketitle
\setcounter{tocdepth}{1}
\tableofcontents

\section{Introduction}
\label{Section:intro}

Frobenius manifolds, introduced by B.~Dubrovin in \cite{Dubrovin:1994hc} as a coordinate free formulation of the Witten--Dijkgraaf--Verlinde--Verlinde (WDVV) equations of 2D topological field theory, have sat for a good quarter of a century at a key point of confluence of algebraic geometry, singularity theory, quantum field theory, and the theory of integrable systems. In algebraic geometry, they serve as a model for the quantum co-homology (genus zero Gromov--Witten theory) of smooth projective varieties; in singularity theory, they encode the existence of pencils of flat pairings on the base of the mini-versal deformations of hypersurface singularities; in physics, they codify the associativity of the chiral ring of topologically twisted $\cN=(2,2)$ supersymmetric field theories in two dimensions; and in the theory of integrable hierarchies, they provide a loop-space formulation of hydrodynamic bihamiltonian integrable hierarchies in 1+1 dimensions. 

On top of the physics-inspired examples of Frobenius manifolds coming from Witten's topological A- and B-twists of 2D-theories with four supercharges, an interesting source of Frobenius manifolds is well-known to arise in Lie theory. Let $\mathfrak{g}_\RR$ be a simple complex Lie algebra associated to an irreducible root system  $\RR$, and write $\mathfrak{h}_\RR$ and $\mathfrak{g}^{(1)}_\RR$ for, respectively, its Cartan subalgebra and the associated untwisted affine Lie algebra. In a celebrated result \cite{Dubrovin:1993nt}, Dubrovin constructed a class of semi-simple polynomial Frobenius manifolds on the space of regular orbits of the reflection representation of $\mathrm{Weyl}(\mathfrak{g}_\RR)$ (and in fact, on the orbit spaces of the defining representation of any Coxeter group). A remarkable extension of this was provided by Dubrovin and Zhang \cite{MR1606165}, who defined a Frobenius manifold structure $\cM^{\rm DZ}_{\RR}$ on quotients of $\mathfrak{h}_\RR\times \bbC$ by a suitable semi-direct product $\mathrm{Weyl}(\mathfrak{g}_\RR^{(1)}){\ltimes}  \bbZ$. They furthermore provided a mirror symmetry construction for Dynkin type~A, $\mathfrak{g}_{A_{N-1}}=\mathrm{sl}_N(\bbC)$, in terms of Laurent-polynomial one-dimensional Landau--Ginzburg models, which was later generalised to classical Lie algebras in \cite{Dubrovin:2015wdx}.  A question raised by \cite{MR1606165,Dubrovin:2015wdx} was whether a similar uniform mirror symmetry construction for all Dynkin types could be established, including exceptional Lie algebras.

This paper gives a constructive Lie-theoretic answer to this question, which is furthermore entirely explicit, and provides closed-form expressions for the flat coordinates of the analogue of the Saito--Sekiguchi--Yano metric and for the Frobenius prepotential. Our mirror theorem has simultaneous implications for singularity theory, integrable systems, the Gromov--Witten theory of Fano orbicurves, and Seiberg--Witten theory, some of which are explored here.

\subsection{Main results}

\subsubsection{Mirror symmetry for Dubrovin--Zhang Frobenius manifolds}

Our main result is the following general mirror theorem for Dubrovin--Zhang Frobenius manifolds (see \cref{thm:mirror} for the complete statement, and \cref{tab:notation} for details of the notation employed). Let $H_{g,\mathsf{n}}$ be the Hurwitz space of isomorphism classes $[\lambda: C_g \to \bbP^1]$ of covers of the complex line by a genus $g$ curve $C_g$ with ramification profile at infinity described by $\mathsf{n} \in (\bbZ_{\geq 0})^{\ell(\mathsf{n})}$, $\ell(\mathsf{n}) \geq 1$. Fixing a suitable meromorphic function $\mu$ on $H_{g,\mathsf{n}}$ 
induces, as a particular case of a classical construction of Dubrovin \cite{Dubrovin:1992eu,Dubrovin:1994hc}, a semi-simple Frobenius manifold structure  $\HH^{[\mu]}_{g,\mathsf{n}}$ on $H_{g,\mathsf{n}}$. 

\begin{thm}[=\cref{thm:mirror}]
For any simple Dynkin type $\RR$ there exists a highest weight $\omega$ for the corresponding simple Lie algebra $\mathfrak{g}$, pairs of integers $(g_{\omega}, \mathsf{n}_{\omega})$, and an explicit embedding $\iota_{\omega} : \cM^{\rm DZ}_{\RR} \hookrightarrow \HH^{[\mu]}_{g_{\omega},\mathsf{n}_{\omega}}$, 
such that $\iota_\omega$ is a Frobenius manifold isomorphism onto its image $\cM_\omega^{\rm LG} \coloneqq \iota_\omega(\cM^{\rm DZ}_\RR)$.
\label{thm:main_0}
\end{thm}

In other words, $\iota_{\omega}$ identifies the Frobenius manifold $\cM^{\rm DZ}_{\RR}$ with a distinguished stratum $\cM^{\rm LG}_\omega$ of a Hurwitz space, which is an affine-linear subspace of the Frobenius manifold $\HH^{[\mu]}_{g_{\omega},\mathsf{n}_{\omega}}$ in its natural set of flat coordinates. The datum of the covering map on $H_{g_{\omega},\mathsf{n}_{\omega}}$ defines a one-dimensional B-model superpotential for $\cM^{\rm DZ}_{\RR}$ in terms of a family of (trigonometric) meromorphic functions $\cM^{\rm LG}_\omega$, whose Landau--Ginzburg residue formulas determine the Dubrovin--Zhang flat pencil of metrics and the Frobenius product structure on $T\cM^{\rm DZ}_{\RR}$.\\

\cref{thm:main_0} is proved in two main steps. Fixing $\omega$ a dominant weight in a minimal non-trivial Weyl orbit, we first associate to $\cM^{\rm DZ}_{\RR}$ a family of spectral curves specified by the vanishing of the characteristic polynomial in the representation $\rho_\omega$ for a pencil of group elements $\mathsf{g}(\lambda) \in \cG:=
\exp\mathfrak{g}$. The construction of the family hinges on determining all character relations of the form $\chi_{\wedge^k \rho_{\omega}} =\mathfrak{p}^\omega_k(\chi_{\rho_{1}},\dots, \chi_{\rho_{l_{\RR}}})$ in the Weyl character ring of $\cG$, where $\rho_i$ is the $i^{\rm th}$ fundamental representation of $\cG$, and $l_\RR$ is the rank of $\RR$.
The resulting family sweeps a submanifold of a Hurwitz space $H_{g_\omega, \mathsf{n}_\omega}$: the second step consists in establishing that this is a Frobenius submanifold of the Frobenius manifold $\HH^{[\mu]}_{g_\omega, \mathsf{n}_\omega}$ satisfying the defining properties of $\cM_\RR^{\rm DZ}$.\\
 
 Our construction is motivated by a conjectural relation of the almost-dual Frobenius manifold \cite{MR2070050} for $\RR$ of type ADE with the orbifold quantum co-homology of the associated simple surface singularity, as proposed in \cite{Brini:2013zsa, Brini:2017gfi}, which is in turn described by a degeneration of a family of spectral curves for the relativistic Toda chain associated to (a co-extension of) the corresponding affine Poisson--Lie group of type ADE \cite{Fock:2014ifa,Williams:2012fz}. The one-parameter family of group elements $\mathsf{g}(\lambda)$ in our construction is given by the Lax operator for the chain, where $\lambda$ is the spectral parameter: the relation to the associated Dubrovin--Zhang Frobenius manifold of type ADE is suggested by analogous results for the simple Lie algebra case due to Lerche--Warner, Ito--Yang, and Dubrovin \cite{Lerche:1996an,Ito:1997ur,MR2070050}.
 
  For the minimal choices, the target Hurwitz space $H_{g_\omega,\mathsf{n}_\omega}$ is a space of rational functions ($g_\omega=0$) for type $\RR=A_l$, $B_l$, $C_l$, $D_l$ and $G_2$, and it is a space of meromorphic functions on higher genus curves in the other exceptional types. Note that different choices of $\omega$ induce different families of spectral curves, and therefore different embeddings $\iota_{\omega}:\cM^{\rm DZ}_{\RR}\hookrightarrow \HH^{[\mu]}_{g_{\omega}, \mathrm{m}_\omega}$ inside a parent Hurwitz space. Whilst we prove \cref{thm:main_0} for dominant weights $\omega$ in a minimal non-trivial Weyl orbit,  we  also provide verifications that different non-minimal choices of $\omega$ indeed give rise to isomorphic Frobenius manifolds.

\subsubsection{Application I: Frobenius prepotentials}

The original Dubrovin--Zhang construction establishes the existence of a Frobenius manifold structure on $\cM^{\rm DZ}_{\RR}$ by abstractly constructing a flat pencil of metrics $\gamma^*+\lambda \eta^*$ on $T^*\cM^{\rm DZ}_{\RR}$, where $\gamma^*$ arises from an extension of the Killing pairing to $\mathfrak{h}_\RR\oplus \bbC$, without reference to an actual system of flat coordinates for $\eta^*$ (the analogue of the Saito--Sekiguchi--Yano metric for finite reflection groups).  From \cref{thm:main_0}, the metric $\eta^*$ and Frobenius product on the base of the family of spectral curves can then be computed using Landau--Ginzburg residue formulas for the superpotential: the associativity of the Frobenius product reduces the analysis of the pole structure of the Landau--Ginzburg residues to the sole poles of the superpotential, giving closed-form expressions for the flat coordinates of $\eta^*$ and its prepotential. We then obtain the following 
\begin{thm}[=\cref{Thm:5.1} and Examples in \cref{sec:examples}]
For all $\RR$, we provide flat coordinates for the Saito metric of the Dubrovin--Zhang pencil and closed-form prepotentials for $\cM^{\rm DZ}_{\RR}$.
\label{thm:prep_0}
\end{thm}
Our expressions recover results of \cite{MR1606165,Dubrovin:2015wdx} for classical Lie algebras; the statements for exceptional Dynkin types are new. \cref{thm:main_0} is key to the determination of the prepotential: the Landau--Ginzburg calculation reduces the computation of flat coordinates for $\eta^*$ and a distinguished subset of structure constants to straightforward residue calculations on the spectral curves, from which the entire product structure on the Frobenius manifold can be recovered using WDVV equations.

\subsubsection{Application II: Lyaskho--Looijenga multiplicities of meromorphic functions}
\label{sec:ll_0}

The enumeration of isomorphism classes of covers of $S^2$ with prescribed ramification over a point is a classical problem in topology and enumerative combinatorics, going back to Hurwitz' formula for the case in which the covering surface is also a Riemann sphere. The result of the enumeration for a cover of arbitrary geometric genus $g$ and branching profile $\mathsf{n}=(n_0,\dots, n_m)$ is the Hurwitz number $h_{g,\mathsf{n}}$, whose significance straddles several domains in enumerative combinatorics \cite{MR1649966,MR1763948}, representation theory of the symmetric group \cite{MR1181077}, moduli of curves \cite{Ekedahl:LL}, and mathematical physics \cite{Crescimanno:1994eg,Bouchard:2007hi,Caporaso:2006gk}. It was first noticed by Arnold \cite{MR1387484} that when the branching profile has maximal degeneration (i.e., for polynomial maps $f:\bbP^1\to \bbP^1$) this problem is intimately related to considering the topology of the complement of the discriminant for the base of the type $A_l$ mini-versal deformation, and in particular to the degree of the  Lyashko--Looijenga mapping \cite{MR542251,MR0422675}
$\mathrm{LL}: \bbC[\mu] \to \bbC[\mu]$, which assigns to a polynomial $\lambda(\mu)$ the unordered set of its critical values $\mathrm{LL}(\lambda)(\mu)=\prod_{\lambda'(\tilde z)=0} (\mu-f(\tilde z))$. This is a finite polynomial map \cite{MR1387484,MR0422675}, inducing a stratification of $\bbC[\mu]$ according to the degeneracy of the critical values of $\lambda$. The computation of the topological degree of this mapping on a given stratum, enumerating the number of polynomials sharing the same critical values counted with multiplicity, can usually be translated into a combinatorial problem enumerating some class of embedded graphs. This connection was used by Looijenga \cite{MR0422675} to reprove Cayley's formula for the enumeration of marked trees (corresponding to the co-dimension zero stratum), and by Arnold \cite{MR1387484} to encompass the case of Laurent polynomials (see also \cite{MR1724267,MR2324558} for generalisations to rational functions and discriminant strata). The extension of this combinatorial approach to arbitrary strata at higher genus, involving enumerations of suitable coloured oriented graphs ($k$-constellations), appears unwieldy \cite{MR1944086}. However, when $\lambda(\mu)$ is the Landau--Ginzburg superpotential of a semi-simple, conformal Frobenius manifold, the graded structure of the latter can be used to determine the Lyashko--Looijenga multiplicity of $\lambda(\mu)$ by a direct application of the quasi-homogeneous Bezout theorem \cite{MR777682}, with no combinatorics involved. In particular \cref{thm:main_0} has the following immediate consequence.

\begin{thm}[=\cref{cor:LLdeg}]
For all $\RR$ we compute the Lyashko--Looijenga multiplicity of the stratum $\iota_\omega(\cM^{\rm DZ}_\RR) =\cM^{\rm LG}_\omega \subset  \HH^{[\mu]}_{g_\omega,\mathsf{n}_\omega}$.
\end{thm}
This includes, in particular, the higher genus Hurwitz spaces appearing for types $\RR=E_n$ and $F_4$ (see \cref{Tab:topdataclass}).

\subsubsection{Application III: the dispersionless extended type-$\RR$ Toda hierarchy}

The datum of a semi-simple conformal Frobenius manifold is equivalent to the existence of a $\tau$-symmetric quasi-linear integrable hierarchy, which is bihamiltonian with respect to a Dubrovin--Novikov hydrodynamic Poisson pencil. Having a description of the Frobenius manifold in terms of a closed-form prepotential allows to give an explicit presentation of the hierarchy in terms of an infinite set of commuting 1+1 PDEs in normal coordinates. The loop-space version of \cref{thm:prep_0} is then the following
\begin{thm}[=\cref{prop:mirrorham}]
For all $\RR$, we construct a bihamiltonian dispersionless hierarchy on the loop space $\LL \cM_\RR^{\rm DZ}$ in Hamiltonian form for the canonical Poisson pencil associated to $\cM_\RR^{\rm DZ}$. 
\label{thm:is_0}
\end{thm}
For type~$A_n$ this integrable hierarchy is the zero-dispersion limit of Carlet's extended bigraded Toda hierarchy \cite{MR2246697}, and for type~$D_n$ it is the long-wave limit of the Cheng--Milanov extended $D$-type hierarchy \cite{Cheng:2019qjo}. For simply-laced $\RR$, we expect that the principal hierarchies of \cref{thm:is_0} should coincide with the dispersionless limit of the Hirota integrable hierarchies constructed by Milanov--Shen--Tseng in \cite{Milanov:2014pma}. The non-simply-laced cases are, to the best of our knowledge, new examples of hydrodynamic integrable hierarchies: our construction of the Landau--Ginzburg superpotential is highly suggestive that these should be obtained as symmetry reductions of the hierarchies in \cite{Milanov:2014pma} by the usual folding procedure of the Dynkin diagram. Aside from laying the foundation for determining the prepotential of $\cM^{\rm DZ}_\RR$, \cref{thm:main_0}  also provides a dispersionless Lax formulation for the hierarchy as an explicit reduction of Krichever's genus-$g_\omega$, $\ell(\mathsf{n}_\omega)$-pointed universal Whitham hierarchy.

\subsection{Further applications}

We also highlight three further applications of \cref{thm:main_0}, which are the subject of ongoing investigation and whose details will be provided in three separate publications. 

\subsubsection{The orbifold Norbury--Scott conjecture}

When $\RR=A_1$, the Frobenius manifold $\cM_\RR^{\rm DZ}$ famously coincides with the quantum cohomology $\mathrm{QH}(\bbP^1, \bbC)$ of the complex projective line.  In \cite{MR3268770}, the authors propose a higher genus version of this statmement and  conjecture that the Chekhov--Eynard--Orantin topological recursion applied to the Landau--Ginzburg superpotential of $\bbP^1$
computes  the $n$-point, genus-$g$ Gromov--Witten invariants  of $\bbP^1$ with descendant insertions of the K\"ahler class  (the ``stationary" invariants) in terms of explicit residues on the associated spectral curve (see \cite{DuninBarkowski:2012bw} for a proof).  It was shown in \cite{MR2672302} that for type $\RR=A_l$, $D_l$ and $E_l$ the Dubrovin--Zhang Frobenius manifolds $\cM_\RR^{\rm DZ}$ are isomorphic to the orbifold quantum cohomology of the Fano orbicurves $\mathscr{C}_\RR = [\bbC^\star \backslash \bbC^2/\Gamma_\RR]$, where $\Gamma_\RR < SU(2)$, $|\Gamma_\RR|<\infty$ is the McKay group of type $\RR$. In particular,
\beq
\mathscr{C}_{\RR} \simeq 
\l\{
\bary{cc}
\bbP(1,l), & \RR=A_l,\\
\bbP_{2,2,l-2}, & \RR=D_l,\\
\bbP_{2,3,l-3} & \RR=E_l.
\eary
\r.
\eeq
The construction of the LG superpotentials of \cref{thm:main_0} now associates a family of mirror spectral curves to the quantum cohomology of these orbifolds. As anticipated in \cite{Brini:2017gfi}, it is natural to conjecture that the Norbury--Scott theorem receives an orbifold generalisation through \cref{thm:main_0}, whereby higher genus stationary Gromov--Witten invariants of $\mathscr{C}_\RR$ can be computed by residue calculus on the respective type $\RR$ spectral curve mirrors. The investigation of the correct phrasing for the topological recursion is ongoing.

\subsubsection{Seiberg--Witten theory}

For the case of polynomial Frobenius manifolds with $\RR=A_l$, $D_l$ or $E_l$, it was noted by a number of authors \cite{Lerche:1996an,Ito:1997ur,MR2070050} that the {\it odd periods} of the Frobenius manifold (in the language of \cite{MR2070050}) give the quantum periods of the Seiberg--Witten family of curves dual to $\cN=2$ pure super Yang--Mills theory on $\bbR^4$ with gauge group given by the compact real form of $\exp(\mathfrak{g})$. In \cite{Nekrasov:1996cz}, Nekrasov reformulated the Seiberg--Witten study of $\cN=2$, $d=4$ gauge theories in the context of five-dimensional $\cN=1$ gauge theories compactified on a circle, by viewing the five-dimensional theory on $\bbR^4 \times S^1$ with gauge group $G$ as, effectively, a four-dimensional theory with gauge group the extended loop group $\widehat{G}$. In this context the classical Coulomb vacua are parametrised by orbits of the associated extended affine Weyl group. It is only natural to conjecture that the construction of odd periods and Picard--Fuchs system for four dimensional Seiberg--Witten theory from the polynomial Frobenius manifolds can be lifted to provide solutions of five-dimensional Seiberg--Witten theory using their Dubrovin--Zhang, extended affine counterpart; this can indeed be explicitly checked for simply-laced cases \cite{Brini:2021wrm}. Considerations about folding in five dimensions also allow to treat non-simply laced Lie groups, which points to the existence of a new class of Frobenius manifolds having as monodromy group an extension of the {\it twisted} affine Weyl group.

\subsubsection{Saito determinants on discriminant strata}

In \cite{antoniou2020saito}, the authors consider semi-simple Frobenius manifolds embedded as discriminant strata on the Dubrovin--Hertling polynomial Frobenius structures on the orbits of the reflection representation of Coxeter groups. In particular, they use the Landau--Ginzburg mirror superpotentials to establish structural results on the determinant of the restriction of the Saito metrics to arbitrary strata. A specific question asked by \cite{antoniou2020saito} is how much of that story can be lifted to the study of the Dubrovin--Zhang Frobenius manifolds on extended affine Weyl group orbits. The Landau--Ginzburg presentation of \cref{thm:main_0} unlocks the power to employ the same successful methodology in the affine setting as well.

\begin{table}[!h]
\begin{tabular}{|c|m{9cm}|}
\hline
$\RR$ & {An irreducible root system}  \\
\hline
$l_\RR$ & {The rank of $\RR$}  \\
\hline
$\mathfrak{g}_\RR$ & {The complex simple Lie algebra with root system $\RR$} \\
\hline
$\cG_\RR$ & {The simply connected complex simple Lie group $\exp(\mathfrak{g}_\RR)$} \\
\hline
$\mathfrak{h}_\RR$ (resp. $\cT_\RR$) & The Cartan subalgebra of $\mathfrak{g}_\RR$ (resp. Cartan torus of $\cG_\RR$)   \\
\hline
$\mathsf{g}$ (resp. $\mathsf{h}$) & A regular element of $\cG_\RR$ (resp. $\mathfrak{h}_{\RR}$) \\
\hline
$\cW_\RR$/$\widehat{\cW}_\RR$/$\widetilde{\cW}_\RR$  & The Weyl/affine Weyl/extended affine  Weyl group \\ & of Dynkin type $\RR$ \\
\hline
$\{\a_1, \dots, \a_{l_\RR}\}$ & The set of simple roots of $\RR$\\
\hline
$\{\omega_1, \dots, \omega_{l_\RR}\}$ & The set of fundamental weights of $\RR$  \\
\hline
$\Lambda_r(\RR)$& The lattice of roots of $\RR$\\
 (resp. $\Lambda_r(\RR)^\pm$) &   (resp. the semi-group of positive/negative roots)  \\ 
\hline
$\Lambda_w(\RR)$ & The lattice of all weights of $\RR$\\
 (resp. $\Lambda_w(\RR)^\pm$) &    (resp. the monoid of non-negative/non-positive weights) \\ 
\hline
$\rho_\omega$ & The irreducible representation of $\cG_\RR$ with highest weight $\omega$ \\
\hline
$\rho_i$ & The $i^{\rm th}$ fundamental representation of $\cG_\RR$, $i=1, \dots, l_\RR$ \\
\hline
$\Gamma(\rho)$  & The weight system of the representation $\rho$ \\
\hline
$\mathrm{Mult}_\rho(\omega)$  & The dimension of the weight space of $\omega$ in the \\ &   representation $\rho$ \\
\hline
$\chi_\omega$ (resp. $\chi_i$) & The formal character of  $\rho_\omega$ (resp. $\rho_{i}$)  \\
\hline
$\mathsf{w}$ & The Weyl vector of $\mathfrak{g}_\RR$  \\
\hline
$[i_1\dots i_{l_\RR}]_\RR$  & Components of a weight in the $\omega$-basis of $\RR$ \\
\hline
$(x_1, \dots, x_{l_\RR})$ & Linear 
coordinates on $\mathfrak{h}_\RR$  w.r.t. the
co-root \\
 & basis $\{\a_1^\vee, \dots, \a_{l_\RR}^\vee\}$
\\
\hline
$(Q_1, \dots, Q_{l_\RR})$ & $(\exp (x_1), \dots, \exp (x_{l_\RR}))$\\
\hline
$C_{\RR}$ (resp. $\cK_{\RR}$) & The Cartan (resp. symmetrised Cartan) matrix of $\RR$ \\
\hline
$\cM^{\rm DZ}_{\RR}/\cM^{\rm LG}_{\omega}/\HH^{[\mu]}_{g, \mathsf{n}}$ & Dubrovin--Zhang/Landau--Ginzburg/Hurwitz  Frobenius\\ 
 ($M^{\rm DZ}_{\RR}/M^{\rm LG}_{\omega}/H_{g, \mathsf{n}}$) & manifolds  (resp. their underlying complex manifolds)\\
\hline
$\mathsf{n} \vdash d$ & A padded partition of $d\in \bbN$ \\
 & (i.e. parts are allowed to be zero)\\
\hline
$|\mathsf{n}|$ (resp. $\ell(\mathsf{n})$) & The length of a partition $\mathsf{n}$ \\
&  (resp. the number of parts in $\mathsf{n}$) \\
\hline
\end{tabular}
\caption{Notation employed throughout the text.}
\label{tab:notation}
\end{table}

\subsection{Organisation of the paper}

The paper is organised as follows. In \cref{Section:DZ} we recall the definition of 
the affine Lie theoretic Frobenius manifolds of Dubrovin--Zhang (DZ).
In \cref{Section:LG} we state how to construct  Frobenius manifolds in terms of Landau--Ginzburg (LG) superpotentials defined on suitable strata of a Hurwitz space. We also recall the construction in \cite{Brini:2017gfi}, describing how to find LG-superpotentials for DZ-manifolds from the characteristic equation of a suitable degeneration of the Lax operator for the type $\RR$ periodic relativistic Toda chain. This boils down to finding relations in the representation ring of $\cG$, which we determine for all Dynkin types giving explicit algebraic expressions for the corresponding superpotentials. In \cref{Section:MS} we prove the mirror theorem, and determine in turn closed-form prepotentials for the corresponding Dubrovin--Zhang Frobenius manifolds, including the hitherto unknown exceptional cases in type $E_6$, $E_7$ and $F_4$. Finally, in \cref{Section:LL}, we discuss the applications to the extended $\RR$-type Toda hierarchies and the calculation of the multiplicities of the Lyashko--Looijenga map of $\cM_\omega^{\rm LG}$. Our notation\footnote{To declutter the polynomial expressions of the prepotentials of $\cM_\RR^{\rm DZ}$, and in a slight departure from the conventions in the Frobenius manifolds literature, components of a chart of $\cM_\RR^{\rm DZ}$ will consistently be written with lower indices; in particular the Einstein summation convention is {\it never} assumed.}
 is described in \cref{tab:notation}.

\subsection*{Acknowledgements} We are grateful to I.~A.~B.~Strachan for his comments on a previous version of the manuscript. Special thanks are offered to A.~Takahashi for pointing out a calculational error in \cref{cor:LLdeg,Tab:topdataclass}, of which we became aware after the appearance of \cite{otani2023number}: this is now rectified in the current version.  We are also indebted to the anonymous referee for numerous helpful suggestions of improvement.

\section{Dubrovin--Zhang Frobenius manifolds}
\label{Section:DZ}

\subsection{Generalities on Frobenius manifolds}
\label{Section:DZ.1}

We start by recalling the basic definitions from the theory of Frobenius manifolds.

\begin{defn}
A (complex, holomorphic) Frobenius manifold is a 5-tuple $\cM=(M,  \cdot, \eta , e, E)$, where $M$ is a finite dimensional complex manifold such that at each point $p \in M$, the fibre $T_pM$ of the holomorphic tangent bundle at $p$ has the structure of a unital associative commutative algebra with multiplication $\cdot$ and identity element $e$, varying holomorphically. Additionally, $\eta$ is a flat, holomorphic, non-degenerate symmetric $(0,2)$-tensor such that the Frobenius property holds:
\begin{equation}
    \eta(X \cdot Y, Z) = \eta(X, Y \cdot Z), \qquad \forall \, X, Y, Z \in \Gamma(M,TM).
\end{equation}

Moreover, the following properties are satisfied:

\ben
\item the unit vector field is horizontal,
    $\nabla e = 0$,
w.r.t. the Levi-Civita connection $\nabla$ associated to $\eta$;
\item
there exists a $(0,3)$-tensor $c \in \Gamma(M,\mathrm{Sym}^3 T^*M)$ such that 
$\nabla_{W}c(X, Y, Z)$
is totally symmetric $\forall  \, W, X, Y, Z \in \Gamma(M,TM)$;

\item there exists $E \in \Gamma(M,TM)$ such that $\nabla E$ is covariantly constant, and the corresponding 1-parameter group of diffeomorphisms acts by conformal transformations of the metric and  the product tensor. 
\een
\label{def:frob}
\end{defn}

\smallskip

A complex Frobenius manifold is {\it semisimple} if the set $$\mathrm{Discr}(M) \coloneqq \{p\in M~|~\exists v \in T_p M~{\rm with}~ v \cdot v=0\}$$ has positive complex codimension. Whenever $E$ is in the group of units of $(T_p M, \cdot)$,  one may define a second flat metric, $\gamma \in \Gamma(M, \mathrm{Sym}^2 T^*M)$, by 
\begin{equation}
    \gamma(E \cdot X, Y) = \eta(X, Y). 
\end{equation}
A key consequence of \cref{def:frob} is the existence of a one-parameter affine family of flat metrics on $T^*M$. The non-degenerate pairings $\eta$ and $\gamma$ on $\Gamma(M, TM)$ define dual cotangent metrics $\eta^*$ and $\gamma^*$, the Gram matrices of which are inverses of those of $\eta$ and $\gamma$. Then \cref{def:frob} implies that 
 $\gamma^*$ and $\eta^*$ form a flat pencil of metrics:
i.e.   $\gamma^*+\lambda \eta^*$
is a flat metric for any $\lambda \in \mathbb{C}$, and its Christoffel symbols satisfy $\Gamma(\gamma^*+\lambda \eta^*)=\Gamma(\gamma^*)+\lambda \Gamma(\eta^*)$. \\

Since the metric $\eta$ is flat, a Frobenius manifold carries a canonical affine equivalence class of charts given by flat frames for $\eta$. Spelling out the Frobenius manifold axioms in one such chart $(t_1, \dots, t_{\dim M})$ on some sufficiently small open $U \subset M$ amounts to the local existence of a holomorphic function $F$ (the \emph{prepotential}) with the following properties (we denote $\partial_i$ as a shorthand for $\frac{\partial}{\partial t_i}$):

\ben[(i)]
    \item the Gram matrix $\eta_{ij} \equiv \eta(\partial_i, \partial_j) = \partial_{ijk}^3 F$, is constant and non-degenerate;
    \item $e=\de_{t_1}$;
    \item $c_{ijk} \equiv c(\partial_i, \partial_j, \partial_k) = \eta(\partial_i \cdot \partial_j, \partial_k) = \dfrac{\partial^3 F}{\partial t_i \partial t_j \partial t_k}$;
    \item $E = \sum_{i} d_i t_i \partial_i + \sum_{i} r_i \partial_i $;
    \item $\gamma_{ij} = \sum_k E^k c_{kij} $; 
    \item $\partial_i \cdot \partial_j = \sum_k c^k_{ij} \partial_k$, where $c^k_{ij} \coloneqq \sum_{m} \eta^{k m}c_{mij}$, and $\eta^{ij} \coloneqq (\eta)^{-1}_{ij}$.
    \item $F$ satisfies the Witten--Dijkgraaf--Verlinde--Verlinde (WDVV) equations,
\begin{equation}
\sum_{k,l}    \dfrac{\partial^3 F}{\partial t_i \partial t_j \partial t_k} \eta^{kl}\dfrac{\partial^3 F}{\partial t_l \partial t_m \partial t_n} = j \longleftrightarrow m.
\end{equation}
\een 



\subsection{Frobenius manifolds from extended affine Weyl groups}
We will here give a condensed description of the Dubrovin--Zhang construction of semi-simple Frobenius manifold structures on the space of regular orbits of extended affine Weyl groups, which follows closely the account given in \cite{Brini:2017gfi}. \\

Let $\mathfrak{g}_\RR$ be a rank-$l_\RR$ complex simple Lie algebra associated to a root system $\mathcal{R}$, $\mathfrak{h}_\RR$ the associated Cartan subalgebra, $\dim \mathfrak{h}_\RR=l_{\mathcal{R}}$, and  $\mathcal{W}_\RR$ the Weyl group. The construction of Dubrovin--Zhang Frobenius manifolds depends on a canonical choice
of a marked node in the Dynkin diagram of $\RR$, which will be labelled $\bar k \in \{1, \dots, l_\RR\}$, and we let $\alpha_{\bar k}$ and $\omega_{\bar k}$ denote the corresponding simple root and fundamental weight, respectively.  This node is an ``attaching" vertex for the external nodes in the diagram, that is, the one which if removed splits the Dynkin diagram into disconnected A-type pieces: we depict this choice of marking on the Dynkin diagrams with one vertex added, corresponding to the affine root, in  \cref{fig:dynkin}. Except for $\RR=A_l$, where any node can be chosen, $\bar k$ marks the fundamental representation $\rho_{\bar k}$ of $\mathfrak{g}_\RR$ of highest dimension. \\

The action of $\mathcal{W}_{\RR}$ on $\mathfrak{h}_\RR$ may be lifted to an action of the affine Weyl group $\widehat{\mathcal{W}}_{\RR} \cong \mathcal{W}_{\RR} \ltimes  \Lambda^\vee_r(\RR)$, with  $\Lambda_r^\vee(\RR)$ being the lattice of co-roots:
 \begin{align}
    \widehat{\mathcal{W}}_{\RR} \times \mathfrak{h}_\RR& \mapsto \mathfrak{h}_\RR,\\
    ((w, \alpha^\vee), h) & \mapsto w(h) + \alpha^\vee.
\end{align}
Then the \textit{extended} affine Weyl group $\widetilde{\mathcal{W}}_{\RR}$
 is defined as the semi-direct product $\widetilde{\mathcal{W}}_{\RR} \coloneqq \widehat{\mathcal{W}}_{\RR} \ltimes \mathbb{Z}$ acting on $\mathfrak{h}_\RR \oplus \mathbb{C}$ by 
\begin{align}
    \widetilde{\mathcal{W}}_{\RR} \times {\mathfrak{h}_{\mathcal{R}}} \oplus \mathbb{C} & \rightarrow \mathfrak{h}_\RR \oplus \mathbb{C},\\
    ((w, \alpha^\vee, {l}), (h, v)) & \mapsto (w(h) + \alpha^\vee + l\omega_{\bar k}, {v} - {l}){.} 
    \label{eq:extweyl}
\end{align}

\begin{figure}[t]
    \centering
\def\arraystretch{1.5}    
    \begin{tabular}{|c|c|c|}
    \hline
    $\RR$ & $\bar k$ & Canonically marked affine Dynkin diagram \\
    \hline \hline 
      $A_l$   & $1, \dots, l$ &  \begin{tikzpicture}[scale=3] 
  \dynkin[label, root radius=0.06cm, labels={\alpha_0,\alpha_1,\alpha_2,\alpha_3,\alpha_{\bar{k}}, \alpha_{l-2},\alpha_{l-1},\alpha_l},affine mark=*] A[1]{ooo.X.ooo}  
\end{tikzpicture} \\ \hline   
        $B_l$ &  $l-1$ &  \begin{tikzpicture}[scale=3]
  \dynkin[label,affine mark=*, root radius=0.06cm, labels={\alpha_0~,\alpha_1~,\alpha_2,\alpha_3,\alpha_j, \alpha_{l-2},\alpha_{l-1},\alpha_l}] B[1]{ooo.o.oXo}
\end{tikzpicture} \\ \hline 
$C_l$ &  $l$ &   \begin{tikzpicture}[scale=3]
  \dynkin[label,affine mark=*, root radius=0.06cm, labels={\alpha_0,\alpha_1,\alpha_2,\alpha_3,\alpha_j, \alpha_{l-2},\alpha_{l-1},\alpha_l}] C[1]{ooo.o.ooX}
\end{tikzpicture} \\ \hline 

$D_l$ &  $l-2$ &   \begin{tikzpicture}[scale=3]
  \dynkin[label,affine mark=*, root radius=0.06cm, labels={\alpha_0~,\alpha_1~,\alpha_2,\alpha_3,\alpha_j, \alpha_{l-2},~\alpha_{l-1},~\alpha_l}] D[1]{ooo.o.Xoo}
\end{tikzpicture} \\ \hline 

$E_6$ &   $3$   &   \begin{tikzpicture}[scale=3]
  \dynkin[label,affine mark=*, root radius=0.06cm, labels={~\alpha_0,\alpha_1,~\alpha_6,\alpha_2,\alpha_3, \alpha_{4},\alpha_{5}}] E[1]{oooXoo}
\end{tikzpicture} \\ \hline 

$E_7$
&   $3$ &      \begin{tikzpicture}[scale=3]
  \dynkin[label,affine mark=*, root radius=0.06cm, labels={\alpha_0,\alpha_1,~\alpha_7,\alpha_2,\alpha_3, \alpha_{4},\alpha_{5},\alpha_6}] E[1]{oooXooo}
\end{tikzpicture} \\ \hline 

$E_8$

&  $3$ &       \begin{tikzpicture}[scale=3]
  \dynkin[label,affine mark=*, root radius=0.06cm, labels={\alpha_0,\alpha_1,~\alpha_8,\alpha_2,\alpha_3, \alpha_{4},\alpha_{5},\alpha_6,\alpha_7}] E[1]{oooXoooo}
\end{tikzpicture} \\ \hline 

$F_4$

&   $2$ &  \begin{tikzpicture}[scale=3]
  \dynkin[label,affine mark=*, root radius=0.06cm, labels={\alpha_0,\alpha_1,\alpha_2,\alpha_3, \alpha_{4}}] F[1]{oXoo}
\end{tikzpicture} \\ \hline 

$G_2$
&  $2$ &   \begin{tikzpicture}[scale=3]
  \dynkin[label,affine mark=*, root radius=0.06cm, labels={\alpha_0,\alpha_1,\alpha_2}] G[1]{oX}
\end{tikzpicture} \\ \hline 

    \end{tabular}

\caption{Affine Dynkin diagrams with canonical markings, as in \cite[Table~1]{MR1606165}. The node corresponding to the affine root is marked in black, and the canonical marked node is indicated with a $\times$.}

\label{fig:dynkin}
\end{figure}

Let $\Sigma_\RR$ denote the hyperplane arrangement associated to the root system $\RR$, and $\mathfrak{h}^{\text{reg}}_\RR \coloneqq \mathfrak{h}_\RR \backslash  \Sigma_\RR $ be the set of regular elements in $\mathfrak{h}_\RR$. The restriction of \eqref{eq:extweyl} to $\mathfrak{h}^{\rm reg}_\RR \oplus \bbC$ is then a free affine action, whose quotient defines the regular orbit space of the extended affine Weyl group of $\RR$ with marked node $\bar k$ as
\begin{equation}
M^{\rm DZ}_{\RR} \coloneqq (\mathfrak{h}_\RR^{\text{reg}} \times \mathbb{C})/ \widetilde{\mathcal{W}}_{\RR}  
\cong \mathcal{T}_\RR^{\text{reg}}/\cW_{\RR} \times \mathbb{C}^*, 
   \label{Def:DZmaniLG}
\end{equation}
where $\mathcal{T}^{\text{reg}}_\RR = \text{exp}( \mathfrak{h}_\RR^{\text{reg}})$ is the image of the set of regular elements of $\mathfrak{h}_\RR^{\rm reg}$ under the exponential map to the maximal torus $\mathcal{T}_\RR$. 

\begin{rmk}
Let $(x_1, \dots, x_{l_\RR})$ be linear coordinates on $\mathfrak{h}_\RR$ w.r.t. the co-root basis $\{\a_1^\vee, \dots, \a_{l_R}^\vee\}$, and extend these to linear coordinates $(x_1, \dots, x_{l_\RR}; x_{l_\RR+1})$ on $\mathfrak{h}_\RR \oplus \bbC$. Writing $Q_i=\re^{x_i}$, we denote $\cI_\RR \coloneqq \bbC[Q_1^\pm, \dots, Q_{l_\RR+1}^\pm]$ the ring of regular functions on the algebraic torus $ \mathcal{T}_\RR \times \bbC^\star \simeq (\bbC^\star)^{l_\RR+1}$. By its definition in \eqref{Def:DZmaniLG}, $M^{\rm DZ}_{\RR}$ is a smooth complex manifold homeomorphic to a Zariski open subset of the affine GIT quotient $\mathrm{Spec}~ \cI_{\RR}^{\widetilde{\mathcal{W}}_{\RR}}$. Note that, in \cite[Sec.~1, Definition and Main Lemma]{MR1606165}, the authors consider a partial compactification of the latter to an affine scheme $\mathrm{Spec}~\cA_\RR$, where $\cA_\RR$ is a polynomial subring, satisfying suitable boundedness conditions at infinity, of the Laurent polynomial ring $\cI_\RR$.  Then $M^{\rm DZ}_{\RR}$ also sits in the underlying affine variety as an open submanifold.\\
\end{rmk}

The linear coordinates $(x_1, \dots, x_{l_\RR}; x_{l_\RR+1})$ on $\mathfrak{h}_\RR^{\text{reg}} \oplus \bbC$ can serve as local coordinates on the regular orbit space. By orthogonal extension of minus the Cartan--Killing form on $\mathfrak{h}_\RR$, we define a non-degenerate pairing $\xi$ on  $\mathfrak{h}_\RR \times \mathbb{C}$ by
\begin{equation}
      \xi(\partial_{x_i}, \partial_{x_j}) \coloneqq
  \begin{cases}
                                   -(\cK_\RR)_{ij} & \text{if } i, j< l_{\mathcal{R}}+1, \\
                                   d_{\bar k} & \text{if } i=j=l_{\mathcal{R}}+1, \\
                                   0 & \text{otherwise,}
  \end{cases}
\end{equation}
with $x_{l_{\mathcal{R}}+1}$ parametrising linearly the right summand in $\mathfrak{h}_\RR \oplus \mathbb{C}$. Here, $d_i \coloneqq \bra \omega_i, \omega_{\bar k}\ket$, with $\bra \a, \b \ket$ being the pairing on $\mathfrak{h}_\RR^*$ induced by the restriction of the Killing form on the Cartan subalgebra. 
The quotient map $\aleph: \mathcal{T}_\RR^{\text{reg}} \times \mathbb{C}^* \rightarrow M_\RR^{\rm DZ}$ from \eqref{Def:DZmaniLG} defines a principal {$\mathcal{W}_{\mathcal{R}}$}-bundle on $M_\RR^{\rm DZ}$: a section $\tilde{\sigma}_i$, lifts a (sufficiently small) open $U \subset M_\RR^{\rm DZ}$ to the $i^\text{th}$ sheet of the cover $V_i \in \widetilde{\sigma}_i^{-1}(U) \equiv V_1 \sqcup \cdots \sqcup V_{|\mathcal{W}_{\mathcal{R}}|} $.  The following reconstruction theorem holds \cite[Thm~2.1]{MR1606165}:

\begin{thm}
There exists a unique (up to isomorphism) semisimple Frobenius structure $\cM_\RR^{\rm DZ}=(M_\RR^{\rm DZ}, e, E, \eta, \cdot)$  satisfying the following properties in flat coordinates $(t_1, \dots, t_{l_\RR+1})$ for $\eta$:
\begin{description}
    \item[DZ-I] $e = \partial_{t_{\bar k}}$;
    \item[DZ-II] $E = \dfrac{1}{d_{\bar k}} \partial_{x_{l_\RR+1}} = \sum_{j=1}^{l_{\RR}} \dfrac{d_j}{d_{\bar k}}t_j \partial_{{t}_j} + \dfrac{1}{d_{\bar k}} \partial_{{t}_{l_{\RR}+1}}$;
    \item[DZ-III] the intersection form is $\gamma = \tilde{\sigma}_i^* \xi$;
    \item[DZ-IV] the prepotential is polynomial in $t_1, \cdots, t_{l_{\RR}+1}$, $e^{t_{l_{\RR}+1}}$.
\end{description}
\label{Thm:DZ}
\end{thm}

Such Frobenius manifolds will always be of charge one, or equivalently, the prepotential will be a degree 2 quasi-homogeneous function of its arguments.

\section{Landau--Ginzburg superpotentials from Lie theory}
\label{Section:LG}

\subsection{One-dimensional LG mirror symmetry} 

\label{sec:hurgen}

Hurwitz spaces are moduli spaces parametrising ramified covers of the Riemann sphere. A point in such a space is {an equivalence} class $[\lambda: C_g \mapsto \mathbb{P}^1]$, where $C_g$ is a genus-$g$ smooth  complex projective curve (compact Riemann surface) and $\lambda$ is a morphism to the complex projective line realising $C_g$ as a branched cover of $\bbP^1$; the equivalence relation is here given by automorphisms of the cover. \\
We consider Hurwitz spaces with fixed ramification over infinity. Let the preimage of $\infty$ consist of $m+1$ distinct points, denoted by $\infty_i \in C_g$ for $i=0, \cdots, m$, with $\lambda$ having degree $n_i + 1$ near $\infty_i$. The corresponding Hurwitz space will be denoted ${H}_{g,\mathsf{n}}$, where $\mathsf{n} \coloneqq (n_0, \cdots, n_m)$.
This is a connected complex manifold (and in fact, an irreducible quasi-projective complex algebraic variety \cite{MR260752}) of dimension
%
   $d_{g,\mathsf{n}}  \coloneqq \text{dim}(H_{g,\mathsf{n}}) = 2g + 2m + \sum_{i=0}^m n_i.$
We will write $\pi$, $\lambda$ and $\Sigma_i$ for, respectively, the universal
family, the universal map, and the sections marking the $\infty_i$, as per the following commutative diagram:
\beq
\xymatrix{ C_g \ar[d]  \ar@{^{(}->}[r]& \cC_{g,\mathsf{n}}\ar[d]^\pi  \ar[r]^{\lambda}  &  \bbP^1 \\
                     [\lambda]  \ar@{^{(}->}[r]^{\rm pt}   \ar@/^1pc/[u]^{P_i}&                                             H_{g,\mathsf{n}}  \ar@/^1pc/[u]^{\Sigma_i}& 
}
\label{eq:hurdiag}
\eeq
We furthermore denote by $\rd=\rd_\pi$  the relative differential with respect to the
universal family and $p_i^{\rm cr}\in C_g\simeq \pi^{-1}([\lambda])$ the critical points $\rd
\lambda =0$ of the universal map. By the Riemann existence theorem, the critical values of $\lambda$, $\{u_i\}_{i = 1, \cdots ,d_{g;\mathsf{n}}}$, serve as local coordinates away from the closed subsets in $H_{g, \mathsf{n}}$ in which 
    $u_i = u_j$ for $i \neq j$, 
whose union is called the discriminant. Additionally, there is an action on a Hurwitz space given by the affine subgroup of the $\PGL_2(\bbC)$-action on the target,
\begin{equation}
    (C,\lambda) \mapsto (C, a\lambda+b), \qquad u_i \mapsto au_i + b,
    \label{Eq:Hurwitzaffine}
\end{equation}
for $a, b \in \mathbb{C}$, and $\forall i=1,\cdots, d_{g;\mathsf{n}}$. 

%
%
On the complement of the discriminant, we can associate a family of semi-simple, commutative, unital $\bbC$-algebra structures on the the tangent fibres at $(u_1, \dots, u_{d_{g,\mathsf{n}}})$ by positing that the coordinate vector fields in the $u$-chart are the idempotents of the algebra,
\begin{equation}
    \partial_{u_i} \cdot \partial_{u_j} = \delta_{ij}\partial_{u_i}.
\label{eq:prodss}
\end{equation}
The unit and Euler vector field,
\begin{equation}
 e =    \sum_{i=1}^{d_{g;\mathsf{n}}}\partial_{u_i},
\quad     E = \sum_{i=1}^{d_{g;\mathsf{n}}}u_i \partial_{u_i}, 
\label{eq:eEhur}
\end{equation}
arise here as the generators of the affine action in \eqref{Eq:Hurwitzaffine} by, respectively, translations and rescalings.

What remains to be constructed to define a full-fledged Frobenius manifold structure on $H_{g,\mathsf{n}}$ is a flat non-degenerate symmetric pairing playing the role of $\eta$, such that the vector fields $e$ and $E$ are, respectively, horizontal and linear under its Levi-Civita connection. This will depend on additional data \cite[Lecture~5]{Dubrovin:1994hc}, as we now explain in the generality we will require. 

\begin{defn}
  A meromorphic function $\mu : \cC_{g,\mathsf{n}} \to \bbP^1$ on the universal family is \emph{$\lambda$-admissible} if it satisfies the following properties:
\ben[(i)] 
\item $\mu$ does not factor through $\lambda$, i.e. $\nexists~ g : \bbP^1 \rightarrow \bbP^1$ s.t. $\mu = g \circ \lambda$;
\item $0 \neq \rd \mu \in \Omega^1_{\cC_{g,\mathsf{n}}/H_{g,\mathsf{n}}}$; 
\item $\mathrm{div}(\mu)=\sum_{i=0}^l a_i \Sigma_i(H_{g,\mathsf{n}})$ for $a_i\in \bbZ$.
\een 
\label{def:ladm}
\end{defn}
The datum of a $\lambda$-admissible projection $\mu$ allows to define extra structure on the universal curve, as follows. We can first of all canonically associate to it a relative one-form given by $\rd \log \mu \in \Omega^1_{\cC_{g,\mathsf{n}}/H_{g,\mathsf{n}}}(\infty_0+\dots+\infty_m)$: this is an exact third-kind differential on the fibres of the universal curve, which has simple poles at $\infty_i$ with residues $\Res_{\infty_i} \rd\log\mu =a_i$.
%
%
The second ingredient that the $\mu$-projection provides is a notion of a meromorphic Ehresmann connection on $T\cC_{g,\mathsf{n}}$, defined in terms of its singular foliation by level sets of $\mu$. Let  $p \in \cC_{g,\mathsf{n}}$ with $\rd \mu(p)\neq 0$, $\mathsf{m} \coloneqq \mu(p)$, so that the leaf $\cC_{g,\mathsf{n}}^{(\mathsf{m})} \coloneqq \{ p' \in \cC_{g,\mathsf{n}} | \mu(p') = \mathsf{m}\}$ is locally transverse to the fibres of the universal curve.
Let $U$ be a small neighbourhood of $\pi(p) \in H_{g, \mathsf{n}}$; by transverality, there is a canonical local holomorphic section $\Sigma_\mathsf{m} : U \to \pi^{-1}(U) $ of $\cC_{g,\mathsf{n}}$, lifting $U$ to $\pi^{-1}(U) \cap \cC_{g,\mathsf{n}}^{(\mathsf{m})}$. Accordingly, holomorphic vector fields $X\in \Gamma(U, TH_{g,\mathsf{n}})$ are lifted to local holomorphic sections $(\Sigma_\mathsf{m})_* X$ of $T\cC_{g,\mathsf{n}}$ which are tangent to the leaves of the foliation. This defines locally around $p$ a holomorphic derivation as
$$\delta^{(\mu)}_X f \coloneqq L_{(\Sigma_\mathsf{m})_* X} f.$$ Globally, however, the leaves of the $\mu$-foliation will fail to be transverse to the fibres of the universal curve at the critical locus of $\mu$: the derivation $\delta^{(\mu)}_X$ will therefore take values in the ring of meromorphic functions on $\cC_{g,\mathsf{n}}$, with poles on the ramification divisor of $\mu$:
\beq 
\bary{rccl}
\delta^{(\mu)}_X ~ : ~ & H^0(\cC_{g,\mathsf{n}}, \cO_{\cC_{g,\mathsf{n}}})  & \xrightarrow{\hspace*{.75cm}}   & H^0(\cC_{g,\mathsf{n}}, \cK_{\cC_{g,\mathsf{n}}}), \\
\\
 & f & \xmapsto{\hspace*{.75cm}} & (\delta^{(\mu)}_X f)(p) \coloneqq (L_{(\Sigma_{\mu(p)})_* X}) f(p) \, .
 \eary 
\label{eq:horder}
\eeq 
In more low-brow terms, and in local coordinates $p=(u_1, \dots, u_{d_{g;\mathsf{n}}}; \mu)$ on the complement of $\rd \mu=0$, the derivation $\delta^{(\mu)}_{\de_{u_i}} f$ is simply the partial derivative taken with respect to $u_i$ whilst keeping $\mu$ constant. The meromorphicity of the derivation near an order-$r$ ramification point $q^{\rm cr}\in \cC_{g,\mathsf{n}}$ with $\mu(q^{\rm cr})=\mathsf{m}(u_1, \dots, u_{d_{g;\mathsf{n}}})$, $\rd \mu(q^{\rm cr})=0$, is then just expressing that $$\delta^{(\mu)}_{\de_{u_i}} (\mu(p)-\mathsf{m})^{1/r} = -r^{-1} (\mu(p)-\mathsf{m})^{(1-r)/r}\de_{u_i} \mathsf{m}$$ has a pole of order $r-1$ at $p=q^{\rm cr}$ as soon as $\de_{u_i} \mathsf{m}(u) \neq 0$.\\

With these definitions at hand, we can define a Frobenius manifold structure $\HH_{g,\mathsf{n}}^{[\mu]} \coloneqq (H_{g,\mathsf{n}}, \cdot, \eta, e, E)$ on the Hurwitz space $H_{g,\mathsf{n}}$. For later convenience, we will introduce an additional parametric dependence of the Frobenius structure on a normalisation factor $\cN\in \bbC^\star$: this will be immaterial {\it per se} in the comparison with the Frobenius manifolds of \cref{Thm:DZ}, as such factor can be scaled away by a Frobenius manifold isomorphism given by a time-$1/(2\cN)$ flow along the Euler vector field\footnote{The Frobenius manifolds of \cref{Thm:DZ} have charge $d=1$, hence their prepotentials are quasi-homogeneous of degree $3-d=2$. A time-$s$ Euler flow hence scales them by $2s$.}, but it will be helpful in simplifying the notation of the proof of \cref{thm:mirror}. In terms of the datum of $(\lambda, \mu, \cN)$, the metric $\eta$ is defined by the residue formula
\begin{equation}
    \eta(X, Y) \coloneqq -\cN \sum_{i} \underset{p_i^{\text{cr}}}{\text{Res}}\dfrac{\delta^{(\mu)}_X~\lambda~ \delta^{(\mu)}_Y\lambda}{\text{d}\lambda} \l(\frac{\rd \mu}{\mu}\r)^2, 
    \label{eq:etares}
\end{equation}
for  $X, Y \in \Gamma(H_{g,\mathsf{n}},TH_{g,\mathsf{n}})$.  Furthermore, combining \eqref{eq:prodss} and \eqref{eq:etares} the $3$-tensor $c(X,Y,Z)$ is defined by the LG formula
\beq
    c(X, Y, Z) \coloneqq  \eta(X, Y \cdot Z) \coloneqq -\cN\sum_{i} \underset{p_i^{\text{cr}}}{\text{Res}}\dfrac{\delta^{(\mu)}_X\lambda~\delta^{(\mu)}_Y\lambda~\delta^{(\mu)}_Z\lambda}{ \text{d}\lambda}\l(\frac{\rd \mu}{\mu}\r)^2,
    \label{eq:cres}
\end{equation}
which clearly satisfies the Frobenius property. Moreover, the second flat pairing is obtained upon replacing $\lambda$ by $\log \lambda$ in \eqref{eq:etares}, 
\begin{equation}
    \gamma(X, Y)  \coloneqq -\cN \sum_{i} \underset{p_i^{\text{cr}}}{\text{Res}}\dfrac{\delta^{(\mu)}_X \log \lambda~ \delta^{(\mu)}_Y \log \lambda}{\rd \log \lambda} \l(\frac{\rd \mu}{\mu}\r)^2.
    \label{eq:gres}
\end{equation}


\begin{prop}
The residue formulas \eqref{eq:etares}--\eqref{eq:gres} define a Frobenius manifold structure
$\HH_{g,\mathsf{n}}^{[\mu]}\coloneqq (H_{g,\mathsf{n}}, \cdot, \eta, e, E)$, which is semi-simple outside the discrimimant of $H_{g,\mathsf{n}}$. 
\end{prop}

The statement of the Proposition is a direct specialisation of the Main Lemma and the proof of Thm~5.1 in \cite[Lecture~5]{Dubrovin:1994hc}, where Frobenius structures are constructed on (suitable covers) of Hurwitz spaces $H_{g, \mathsf{n}}$. These depend on a choice of a {\it primary differential} $\phi$: a meromorphic relative 1-form on $\pi:\cC_{g,\mathsf{n}}\to H_{g,\mathsf{n}}$ satisfying suitable admissibility conditions, classified in five Types I--V in \cite[Lecture~5]{Dubrovin:1994hc}. If $\mu$ is $\lambda$-admissible, the case $\phi \coloneqq \sqrt{\cN} \rd \log \mu$ considered here (an exact third kind differential having at most simple poles at the poles of $\lambda$) is readily seen to satisfy the criteria of Type III given in  \cite[Lecture~5]{Dubrovin:1994hc}, and therefore leads to an honest Frobenius manifold. We will call the marked meromorphic function $\lambda$ a {\it Landau-Ginzburg (LG) superpotential} for the Frobenius manifold, with primary differential $\sqrt{\cN} (\rd\mu)/\mu$.

\subsection{Superpotentials for extended affine Weyl groups} 
\label{sec:supconstr}

We give here a general method for the construction of spectral curves associated to affine relativistic Toda chains, as anticipated in \cite{Brini:2017gfi}, for arbitrary Dynkin types. \\

Let $\omega\in \Lambda_w^+(\RR)$ be the highest weight of a non-trivial irreducible representation $\rho_\omega \in \mathrm{Rep}(\cG_\RR)$ of minimal dimension. In particular, $\rho_\omega$ is quasi-minuscule, 
i.e. all non-zero weights in the weight system $\Gamma(\rho_\omega)$ are in the same irreducible orbit under the action of the Weyl group, and it is minuscule (quasi-minuscule with no zero weights) for all  $\RR \neq B_l$, $E_8$, $F_4$ and $G_2$. Consider the characteristic polynomial of $\mathsf{g}\in \cG$ in the representation $\rho_\omega$,
\bea
\cQ_{\omega}(\chi_1, \dots, \chi_{l_\RR};\mu) &=& \det_{\rho_\omega}{\l(\mathsf{g}-\mu
\mathbf{1}\r)} = \sum_{k=0}^{\dim \rho_\omega} (-\mu)^{(\dim \rho_\omega-k)}
\chi_{\wedge^k \rho_{{\omega}}}(\mathsf{g}),
\label{eq:charpol}
\eea
where the second equality is the co-factor expansion of the determinant. Recall that the representation ring of a simple Lie group is an integral polynomial ring generated by the fundamental representations,
\beq
\chi_{\wedge^k \rho_{\omega}}(\mathsf{g}) = \mathfrak{p}^\omega_k(\chi_1, \dots, \chi_{l_\RR})\in \bbZ[\chi_1, \dots, \chi_{l_\RR}],
\label{eq:chardec}
\eeq
where $\chi_i(\mathsf{g}) \coloneqq \tr_{\rho_{i}}(\mathsf{g})$ is the $i^{\rm th}$ fundamental character. Since $\rho_\omega$ is quasi-minuscule, $\cQ_{\omega}$ factorises as
\beq
\cQ_{\omega}= (1-\mu)^{\mathsf{z}_0} \cQ^{\rm red}_{\omega} = (1-\mu)^{\mathsf{z}_0} \prod_{0\neq \omega' \in \Gamma(\rho_\omega)}\l(\re^{\omega' \cdot h}-\mu\r),
\label{eq:weylrel}
\eeq
where $\mathsf{z}_0$ is the dimension of the zero weight space of $\rho_{\omega}$, and $\re^{\mathsf{h}}$ with $[\re^{\mathsf{h}}]=[\mathsf{g}]$ is a choice of Cartan torus element conjugate to $\mathsf{g}$: in particular, $\mathsf{z}_0=0$ and $\cQ^{\rm red}_{\omega}=\cQ_{\omega}$ for $\RR\neq B_l$, $E_8$, $F_4$, or $G_2$.  

Define now
\beq 
\cP_{\omega}(w_0, \dots, w_{l_\RR+1};\lambda,\mu) \coloneqq
 \cQ_{\omega}^{\mathrm{red}}\l(\chi_i=w_i-\delta_{i \bar k}\frac{\lambda}{w_0};\mu \r),
 \label{eq:shiftfun}
\eeq 
and consider, 
as $w\coloneqq (w_0 ; w_1, \dots, w_{l_\RR}) \in {\, \bbC^*} \times \bbC^{l_\RR}$ varies, the family of plane algebraic curves in $\mathrm{Spec}\bbC[\lambda,\mu]$ with fibre at $w$ given by
$C^{(\omega)}_w \coloneqq \bbV\l(\cP_{\omega} \r)$.
%
We compactify and desingularise the fibres over $w$ by taking the normalisation $\overline{C_w^{(\omega)}}$ of their closure in $\bbP^2$. Marking the $\lambda$-projection $(\lambda, \mu)\to \lambda \in \bbP^1$ and varying $w$ defines a subvariety
 $M_\omega^{\rm LG}$ of the Hurwitz space $H_{g_{\omega},\mathsf{n}_{\omega}}$, where $g_\omega=h^{1,0}\big(\overline{C_w^{(\omega)}}\big)$, and $\mathsf{n}_{\omega}$ records the ramification profile at infinity of the $\lambda$-projection.
 \begin{rmk}
 It is not obvious that the above defines an immersion of $M_\omega^{\rm LG}$ into $H_{g_\omega, \mathsf{n}_\omega}$.
 This will follow from the ``rectification statement'' in \cref{Thm:5.1}, according to which $M_\omega^{\rm LG}$ embeds as a dimension-($l_{\RR}+1$) affine hyperplane into $H_{g_\omega, \mathsf{n}_\omega}$ in the respective flat coordinate systems.
 \end{rmk}
 %
In the next two Sections we will calculate the character relations \eqref{eq:chardec}, and therefore determine explicitly the polynomials $\cP_\omega$.  The following Proposition is an anticipated consequence of this direct calculation.
 
 \begin{prop}
 Let 
 $\widetilde{\cP}^{(j)}_\omega(w_0, \dots, w_{l_\RR}; \mu) \coloneqq [\lambda^j] \cP_{\omega}$ be the $j^{\rm th}$ coefficient of $\cP_{\omega}$ in the variable $\lambda$,  and let $ j_\omega^{\max}\coloneqq \deg_\lambda \cP_{\omega}$. Then $\widetilde{\cP}_\omega^{(j_\omega^{\max})}$ is a product of a monomial $\mu^{a_\omega}$ and cylotomic polynomials $\Phi_{k_i}(\mu)$
\beq
\label{eq:polmodind}
\widetilde{\cP}_\omega^{(j_\omega^{\max})}(w_0, \dots, w_{l_\RR}; \mu) = \mu^{a_\omega} \prod_{i=1}^{b_\omega} \Phi_{k^\omega_i}(\mu),
\eeq 
with $a_\omega \in \bbZ_{>0}$ and $b_\omega, k^\omega_i \in \bbZ_{\geq 0}$. Moreover,
\bea
\label{eq:muzeroes}
\widetilde{\cP}^{(j)}_\omega(w_0, \dots, w_{l_\RR}; 0) &=& \delta_{j0},\\
\lim_{\mu'\to 0}(-\mu')^{\dim\rho_\omega}\widetilde{\cP}^{(j)}_\omega(w_0, \dots, w_{l_\RR}; 1/\mu') &=& \delta_{j0}.
\label{eq:mupoles}
\eea
\label{prop:mu}
\end{prop}
\begin{rmk}
The projection $\lambda: \overline{C_w^{(\omega)}} \to \bbP^1$ can only possibly have poles at $\mu=\infty$ or at the zeroes of $\widetilde{\cP}_\omega^{(j_\omega^{\max})}$. The first equality in the Proposition, \eqref{eq:polmodind}, entails then that  $\delta^{(\mu)}_{\de_{w_i}}\widetilde{\cP}_\omega^{(j_\omega^{\max})} =0$ since the r.h.s. is constant in $w$ at fixed $\mu$. In particular, the ramification profile $\mathsf{n}_{\omega}$ is independent of $w$. \\
The second part of the Proposition implies that the zeroes of $\mu$ must occur at points that are poles of $\lambda$, since by \eqref{eq:muzeroes} the equation $\cP_\omega(\lambda, \mu)|_{\mu=0} = 0$ has no solutions for finite $\lambda$. Likewise, replacing $\mu\to \mu=1/\mu'$ reveals that poles of $\mu$ are also poles of $\lambda$ by \eqref{eq:mupoles}. 
Therefore, $\rd\log \mu$ has at most simple poles at the poles of $\lambda$,  showing in particular that $\mu$ satisfies Property~(iii) in \cref{def:ladm}, with Properties~(i)-(ii) being obvious, and is therefore $\lambda$-admissible. 
\label{rmk:muadm}
\end{rmk}

Following the discussion of \cref{sec:hurgen}, and fixing $\cN_\omega \in \bbC^\star$, we can then define a family of semi-simple, commutative, unital Frobenius algebras on $TM^{\rm LG}_\omega$ via \eqref{eq:etares}-\eqref{eq:cres}:
%
\bea
\label{eq:todaeta}
\eta(\de_{w_i}, \de_{w_j}) &=&  - \cN_\omega \sum_{l}\underset{p_l^{\text{cr}}}{\text{Res}}\frac{\delta^{(\mu)}_{\de_{w_i}}\lambda~
  \delta^{(\mu)}_{\de_{w_j}}\lambda}{\rd \lambda}\l( \frac{\rd\mu}{\mu} \r)^2,  \\
\label{eq:todac}
\eta(\de_{w_i}, \de_{w_j}\cdot \de_{w_k}) &=&  - \cN_\omega
\sum_{l}\underset{p_l^{\text{cr}}}{\text{Res}}\frac{\delta^{(\mu)}_{\de_{w_i}}\lambda~
  \delta^{(\mu)}_{\de_{w_j}}\lambda \delta^{(\mu)}_{\de_{w_k}}\lambda}{\rd \lambda}\l( \frac{\rd\mu}{\mu} \r)^2,\\
\label{eq:todag}
\gamma(\de_{w_i}, \de_{w_j}) &=&  - \cN_\omega \sum_{l}\underset{p_l^{\text{cr}}}{\text{Res}}\frac{\delta^{(\mu)}_{\de_{w_i}}\log \lambda~
  \delta^{(\mu)}_{\de_{w_j}} \log \lambda}{\rd \log \lambda}\l( \frac{\rd\mu}{\mu} \r)^2, 
\eea
where $\{p_l^{\rm cr}\}_l$ are the ramification points of $\lambda :
\overline{C_w^{(\omega)}} \to \bbP^1$. This doesn't yet give a Frobenius manifold, or indeed a Frobenius submanifold of $\HH_{g_\omega,\mathsf{n}_\omega}^{[\mu]}$, as $\eta$ is not guaranteed to be non-degenerate or flat at this stage. The following statement establishes that this is the case. 

\begin{thm}[Mirror symmetry for DZ Frobenius manifolds]
The Landau--Ginzburg formulas \eqref{eq:todaeta}--\eqref{eq:todag} define a
semi-simple conformal Frobenius submanifold $\iota_{\omega}: \cM_{\omega}^{\rm
  LG}=(M_{\omega}^{\rm
  LG}, \eta, e, E, \cdot) \hookrightarrow \HH^{[\mu]}_{g_\omega,\mathsf{n}_\omega}$. In particular, \eqref{eq:todaeta} and \eqref{eq:todag} give flat, non-degenerate metrics on $TM_{\omega}^{\rm LG}$, and the identity and Euler vector
fields read
\beq
e= w_0^{-1} \de_{w_{\bar k}}, \qquad E= w_0 \de_{w_0}.
\label{eq:eE}
\eeq
Furthermore,
\beq
\cM_{\omega}^{\rm LG} \simeq \cM^{\rm DZ}_{\RR}.
\label{eq:geniso}
\eeq
\label{thm:mirror}
\end{thm}

The explicit embedding $\iota_{\omega} : M_{\omega}^{\rm
  LG} \hookrightarrow  H_{g_\omega,\mathsf{n}_\omega}$ is described by the relations \eqref{eq:weylrel} in the character ring of $\cG$, setting the coefficients associated to the interior of the Newton polytope of $\cP_{\omega}^{\rm red}(\lambda,\mu)$ to be the polynomials $\mathfrak{p}^\omega_k(w_1, \dots, w_r)$. \\

The proof of \cref{thm:mirror} requires two key  steps:

\ben
\item
computing the exterior relations \eqref{eq:weylrel} in the Weyl character ring of $\cG$: this was solved for simply-laced cases in \cite{Borot:2015fxa,Brini:2019agj}, and we complete the solution here in full generality;
\item proving that the LG formulas \eqref{eq:todaeta}--\eqref{eq:todag}, combined with the reconstruction theorem \cref{Thm:DZ}, establish the mirror statement of \cref{thm:mirror}.
\een

In the remainder of this Section we perform the first step and construct explicitly the family of LG mirror duals to type-$\RR$ Dubrovin--Zhang Frobenius manifolds. We will then devote \cref{Section:MS} to show how the second step leads to a proof of \cref{thm:mirror}.

\subsection{Superpotentials for classical Lie groups}
In the following we present the construction of the spectral curve for the classical root systems $\RR=A_l$, $B_l$, $C_l$, $D_l$ independently, and show how our construction for a weight $\omega$ corresponding to a minimal-dimensional representation $\rho_\omega$ recovers the mirror results of \cite{Dubrovin:2015wdx} for these cases. We will use the shorthand notation  $\varepsilon_i \coloneqq \chi_{\wedge^i\rho_{\omega}}$ for the exterior characters of  $\rho_\omega$.


\subsubsection{$\RR=A_l$}
The Dynkin diagram for affine $A_l$ is shown in \cref{fig:dynkin}. In this case we can choose any (non-affine) node to be the marked one, since the removal of any node from the corresponding finite Dynkin diagram results in two disconnected $A$-type pieces, with ranks adding up to $l-1$. 

A choice of minimal, nontrivial, irreducible representation $\rho_\omega\coloneqq \rho_1=(\mathbf{l+1})$ for $\mathrm{SL}_\bbC(l+1)$ is the defining $(l+1)$-dimensional representation, the other choice corresponding to its dual representation, $\rho_l =\wedge^l \square$.  We then have that $\varepsilon_i = \chi_i$ for $i =1, \dots, l$, and $\varepsilon_0 = \varepsilon_{l+1} = 1$.
Thus \eqref{eq:charpol} becomes
\begin{equation}
\mathcal{P}_{[10\dots 0]_{A_l}} = \dfrac{(-1)^{\bar{k}} \lambda \, \mu^{\bar{k}}}{w_0} + 1 + (-1)^{l+1}\mu^{l+1} + \sum_{i=1}^{l} (-1)^i w_i \mu^i, 
\label{Eq:Alcurve}
\end{equation}
which defines a family of genus 0 curves. 
Setting  (\ref{Eq:Alcurve}) equal to zero and solving for $\lambda$ gives
\begin{equation}
    \lambda = \dfrac{(-1)^{\bar{k}} w_0(1+(-1)^{l+1}\mu^{l+1} + \sum_{i=1}^{l} (-1)^{i} w_i \mu^i)}{\mu^{\bar{k}}},
    \label{Eq:AlSuppot}
\end{equation}
which is, for every point in the moduli space, a meromorphic function of $\mu$ with poles at $0$ and $\infty$ of orders $\bar k, l+1-\bar k$, respectively. We see that we have $l+1$ parameters $w_0, \cdots, w_l$, and so the resulting Frobenius manifold is $l+1$ dimensional\footnote{This is, in fact,  the case for all DZ-manifolds associated to a simple Lie algebra of rank $l$.}. In particular, it is an $l+1$ dimensional submanifold of the Hurwitz space $H_{0,\mathsf{n}_\omega}$, with ramification profile $\mathsf{n}_\omega = (\bar k-1, l-\bar k)$. This Hurwitz space, however, is of dimension $
2+\bar k-1+l-\bar k = l+1$,
and so the DZ-Frobenius manifold associated to $A_l$ is isomorphic to (a full-dimensional ball inside) its associated Hurwitz space. This, as we will see, will not be the case for the other Dynkin types.

\subsubsection{$\RR=B_l$}

For $l>2$, the minimal, nontrivial, irreducible representation of $\mathrm{Spin}(2l+1)$ is the defining representation $\rho_1=(\mathbf{2l+1})$ of the special orthogonal group in $(2l+1)$-dimensions\footnote{For $l=2$, the 4-dimensional spin representation $\rho_2$ is both minimal and minuscule. It is also isomorphic to the vector representation $\rho_1$ of $C_2$, which is included in the discussion of the next section.}, which is the irreducible representation with highest weight $\omega_1$. In this case the marked node is $\bar k =l-1$, as depicted in \cref{fig:dynkin}.

For $i<l$, the $i^{\rm th}$ fundamental representation $\rho_i$ of $B_l$ is the $i^{\rm th}$ exterior power of $(\mathbf{2l+1})$. For $i=l$, the decomposition of the tensor square of $\rho_l$ leads to
\begin{equation}
  \mathfrak{p}^{[10\dots 0]_{B_l}}_i =
  \begin{cases}
                                   \chi_i & \text{if $i<l$}, \\
                                   \chi_l^2 - \sum_{j=0}^{l-1}\chi_j & \text{if $i = l$}.\\
  \end{cases}
\end{equation}

Together with the self-duality relation $\mathfrak{p}^{[10\dots 0]_{B_l}}_i = \mathfrak{p}^{[10\dots 0]_{B_l}}_{2l+1-i}$, we get that the curve is the zero locus of 
\begin{equation}
  \mathcal{P}_{[10\dots 0]_{B_l}} = \dfrac{(-1)^l (\mu-1)(\mu+1)^2 \mu^{l-1} \lambda}{w_0} + \sum_{i=0}^{l} (-1)^i\mu^i(1-\mu^{2(l-i)+1})\varepsilon_i,
  \label{Eq:Blcurve}
\end{equation}
with 
\beq 
\varepsilon_i =
\begin{cases}
1 & \text{if } i=0, \\
w_i & \text{if } 1<i<l, \\
w_l^2 - \sum_{j=0}^{l-1}w_j & \text{if } i=l.
\end{cases}
\eeq 
Note that (\ref{Eq:Blcurve}) has a factor of $(\mu-1)$, since $(\mathbf{2l+1})$ has a one-dimensional zero weight space. Setting to zero the reduced characteristic polynomial $\mathcal{P}^{\rm red}_{[10\dots 0]_{B_l}}=\mathcal{P}_{[10\dots 0]_{B_l}}/(\mu-1)$ gives
\begin{equation}
    \lambda = \dfrac{(-1)^lw_0}{\mu^{l-1}(\mu+1)^2}\sum_{j=0}^{2l}\mu^j \left( \sum_{i=0}^{\text{min}(j, 2l-j)} (-1)^i \varepsilon_i\right).
    \label{Eq:BlSuppot}
\end{equation}
 For each point in the moduli space, this is a rational function in $\mu$ with three poles at $0, -1$, and $ \infty$ of orders $l-1, 2, l-1$, respectively. Hence, $M_{[10\dots 0]_{B_l}}^{\rm LG}$ is a sublocus in the $(2l+1)$-dimensional Hurwitz space $H_{0,\mathsf{n}_\omega}$, with $\mathsf{n}_\omega = (l-2, 1, l-2)$. The latter carries an involution given by sending $\mu \to 1/\mu$, and $M_{[10\dots 0]_{B_l}}^{\rm LG}$ is characterised as the $(l+1)$-dimensional stratum that is fixed by the involution.

\subsubsection{$\RR=C_l$}
The minimal, nontrivial, irreducible representation for $\mathrm{Sp}(2l)$ is the defining representation $\rho_1=(\mathbf{2l})$ of the rank $2l$ symplectic group. Again, this representation corresponds to the one in which $\omega_1$ is highest weight. The canonical node is the $l^{\text{th}}$ node, as shown in \cref{fig:dynkin}.

The exterior powers $\wedge^i\rho_1$  are reducible, with only fundamental representations appearing as direct summands in their decomposition, giving the character relations

\begin{equation}
     \mathfrak{p}^{[10\dots 0]_{C_l}}_i =
  \begin{cases}
                                  \sum_{j=0}^{\frac{i}{2}} \chi_{2j} & \text{for } i~{\rm even}, \\
                                    \sum_{j=0}^{\frac{i-1}{2}}\chi_{2j+1} &  \text{for } i~{\rm odd}.	 \\
  \end{cases}
\end{equation}

From this, and the fact that $\bar k = l$ for $C_l$, we see that the characteristic polynomial \eqref{eq:charpol} is
\begin{equation}
    \mathcal{P}_{[10\dots 0]_{C_l}} = \dfrac{(-1)^l \mu^l \lambda}{w_0} + \sum_{i=0}^{l-1}(-1)^i \varepsilon_i \mu^i(1+\mu^{2(l-i)})+ (-1)^l \varepsilon_l \mu^l,
    \label{Eq:Clcurve}
\end{equation}
with $\varepsilon_{2i}=\sum_{j=0}^{i} \chi_{2j}$, $\varepsilon_{2i+1}=\sum_{j=0}^{i} \chi_{2j+1}$ (and $\varepsilon_0 = 1$ as usual). Setting equal to zero and solving for $\lambda$ gives
\begin{equation}
    \lambda = \dfrac{(-1)^{l-1} w_0 \left(\sum_{i=0}^{l-1}(-1)^i \varepsilon_i \mu^i(1+\mu^{2(l-i)})+ (-1)^l \varepsilon_l \mu^l \right)}{\mu^l}, 
    \label{Eq:ClSuppot}
\end{equation}
which is a rational function in $\mu$ with two poles at $0$ and $\infty$ both of order $l$. Hence, the associated covering Hurwitz space is $H_{0,\mathsf{n}_\omega}$, with $\mathsf{n}_\omega = (l-1, l-1)$, which has dimension $2l$. As before, there is an involution on this Hurwitz space  sending $\mu \to 1/\mu$, with $M_{[10\dots 0]_{C_l}}^{\rm LG}$ being the $(l+1)$-dimensional stratum that is fixed by it.

\subsubsection{$\RR=D_l$}
For $l\geq 4$, the minimal, nontrivial, irreducible representation of $\mathrm{Spin}(2l)$ is the defining vector\footnote{For $l=4$, this can be any of the irreducible 8-dimensional representations, related by triality.} representation $\rho_1=(\mathbf{2l})_v$ of $\mathrm{SO}(2l)$, which corresponds to the irreducible representation with highest weight $\omega_1$. The canonical node is the one with label $l-2$, as shown in Figure \ref{fig:dynkin}.

The character relations for $D_l$ were found in \cite{Borot:2015fxa} to be
\bea
  \mathfrak{p}^{[10\dots 0]_{D_l}}_{i}  &=& \chi_i, \quad i<l-1, \\
  \mathfrak{p}^{[10\dots 0]_{D_l}}_{l-1}  &=& \chi_{l-1}\chi_{l} -
  \begin{cases}
                                  \sum_{j=0}^{\frac{l}{2}-2} \chi_{2j+1} & \text{if $l$ is even,} \\
                                 \sum_{j=0}^{\frac{l-3}{2}} \chi_{2j} & \text{if $l$ is odd,} \\
  \end{cases} \\
  \mathfrak{p}^{[10\dots 0]_{D_l}}_{l} &=&  \chi_{l-1}^2+ \chi_{l}^2 - 2
        \begin{cases}
                                   \sum_{j=0}^{\frac{l}{2}-1} \chi_{2j} & \text{if $l$ is even,} \\
                                  \sum_{j=0}^{\frac{l-3}{2}}\chi_{2j+1} & \text{if $l$ is odd,} \\
  \end{cases}
\eea
and $\mathfrak{p}^{[10\dots 0]_{D_l}}_{i} =  \mathfrak{p}^{[10\dots 0]_{D_l}}_{2l-i}$, so that
\begin{equation}
    \mathcal{P}_{[10\dots 0]_{D_l}} =  \dfrac{(-1)^l\mu^{l-2} (\mu^2-1)^2 \lambda}{w_0} + \sum_{i = 0}^{l-1} (-1)^i \varepsilon_i \mu^i (1+\mu^{2(l-i)}) + (-1)^l \mu^l \varepsilon_l,
    \label{Eq:Dlcurve}
\end{equation}
where as before we denote $\varepsilon_i(w_1, \dots, w_l)=  \mathfrak{p}^{[10\dots 0]_{D_l}}_{i}(\chi_j=w_j)$, with $\varepsilon_0 = 1$. Setting (\ref{Eq:Dlcurve}) equal to zero and solving for $\lambda$ gives
\begin{equation}
    \lambda = (-1)^{l-1}\dfrac{w_0\left(\sum_{i = 0}^{l-1} (-1)^i \varepsilon_i \mu^i (1+\mu^{2(l-i)}) + (-1)^l \mu^l \varepsilon_l \right)}{\mu^{l-2}(\mu^2-1)^2}, 
    \label{Eq:DlSuppot}
\end{equation}
which, for every point $w$, is a rational function in $\mu$ with four poles at $0$, $\infty$, $1$, $-1$ of orders $l-2$, $l-2$, $2$, $2$, respectively. Hence, the parent Hurwitz space is $H_{0;\mathsf{n}_\omega}$, where $\mathsf{n}_\omega = (l-3, l-3, 1, 1)$, which has dimension $2l+2$. Once more this hosts an involution obtained by sending $\mu \to 1/\mu$, identifying $M_{[10\dots 0]_{D_l}}^{\rm LG}$ as its fixed locus.

\subsection{Comparison with the Dubrovin--Strachan--Zhang--Zuo construction}
\label{sec:DSZZ}

For the case of $\mathcal{R} = A_{l}$, an LG-superpotential was already found in the original paper \cite{Dubrovin:1994hc}, with\footnote{To relate this to the expression in \cite{Dubrovin:1994hc} let $\mu = e^{i \phi}$.}
\begin{equation}
\lambda = \sum_{j=0}^{k+m}b_j \mu^{m-j} ,
    \label{Eq:AlDZsupPot}
\end{equation}  where $b_j \in \mathbb{C}$ and $b_0 b_{k+m}\neq 0$. Moreover, in \cite{Dubrovin:2015wdx}, the authors construct a three-integer parameter family of superpotentials of the form\footnote{To relate this to the expression in \cite{Dubrovin:2015wdx} let $\mu = e^{2 i \phi}$.}

\begin{equation}
 \lambda^{\rm DSZZ}(l, k, m) =    \dfrac{4^m \mu^m \sum\limits_{j=0}^{l}a_j 2^{-2(-j+k+m)} \left(\frac{\mu+1}{\sqrt{\mu}}\right)^{2(-j+ k+m)}}{(\mu-1)^m}.
 \label{Def:DSZZSuppot}
\end{equation}

The key result of \cite{Dubrovin:2015wdx} is an identification of \eqref{Def:DSZZSuppot} with a superpotential for a Dubrovin--Zhang Frobenius manifold  of type $B_l$, $C_l$, $D_l$, possibly with a non-canonical choice of marked node in the Dynkin diagram, for suitable choices of $(l,k,m)$. In particular, the mirror theorem for the canonical label $\bar k$ is obtained by setting $(l, k, m)$ equal to $(l, l-1, 1), (l, l, 0),$ and $(l, l-2, 1)$, respectively. We shall now show that the results of  \cite{Dubrovin:1994hc} and \cite{Dubrovin:2015wdx} coincide with our construction in the previous section.

\subsubsection{$\RR=A_l$}

By using the fact that $k+m=l+1$,  (\ref{Eq:AlDZsupPot}) becomes
\begin{equation}
  \lambda =  \dfrac{\sum_{j=0}^{l+1} b_j \mu^{l+1-j}}{\mu^{k}}, 
\end{equation}
which is the same as (\ref{Eq:AlSuppot}) by
\begin{equation}
        b_i = (-1)^k w_0\begin{cases}
                                  1 & \text{if $i=0$}, \\
                                   (-1)^{l+1}  & \text{if $i = l+1$}, \\
                                   w_i & \text{otherwise}.\\
  \end{cases}
\end{equation}

\subsubsection{$\RR=B_l$}

In the case of $B_l$ we consider (\ref{Def:DSZZSuppot}) with $k = l-1$, $m = 1$, which is
\begin{equation}
    \lambda^{\rm DSZZ}(l, l-1, 1) = \dfrac{4\mu \sum_{j=0}^{l}a_j 2^{-2(l-j)}\left(\sqrt{\mu}  + \dfrac{1}{\sqrt{\mu}} \right)^{2(l-j)}}{(\mu-1)^2}.
    \label{Eq:DSZZB}
\end{equation}
Simplifying (\ref{Eq:DSZZB}) gives: 
\begin{equation}
     \dfrac{w_0}{(\mu-1)^2}\sum_{j=0}^{l} \dfrac{a_j 2^{-2(l-j-1)} (\mu+1)^{2(l-j)}}{w_0 \,  \mu^{l-j-1}} = \dfrac{(-1)^l w_0 }{(\mu+1)^2 \mu^{l-1}} \sum_{\beta = 0}^{2l} (-1)^\beta C_{\beta} \mu^{\beta} ,
\end{equation}
where we have used the binomial theorem and let $\mu \mapsto -\mu$,  with 
\begin{equation}
    C_\beta = \dfrac{(-1)^l}{w_0} \sum_{j, \alpha | j+ \alpha = \beta }a_j 2^{-2(l-j-1)} \binom{2(l-j)}{\alpha} =  \dfrac{(-1)^l}{w_0}\sum_{j=0}^{\beta} a_j 2^{-2(l-j-1)}\binom{2(l-j)}{\beta-j}.
\end{equation}
 
 On the other hand, the superpotential constructed from the spectral curve, (\ref{Eq:BlSuppot}),  is given by  
\begin{equation}
    \lambda_{B_l} = \dfrac{(-1)^lw_0}{\mu^{l-1}(\mu+1)^2}\sum_{j=0}^{2l}\mu^j \left( \sum_{i=0}^{\text{min}(j, 2l-j)} (-1)^i \varepsilon_i\right), 
\end{equation}

which we can write as 
 \begin{equation}
      \lambda_{B_l} = \dfrac{(-1)^lw_0}{\mu^{l-1}(\mu+1)^2}\sum_{j=0}^{2l}b_j\mu^j,
 \end{equation}
with $b_{j}  = \sum_{i=0}^{\text{min}(j, 2l-j)} (-1)^i \varepsilon_i$; note that $b_j = b_{2l-j}$. This means that we want to match up
 $b_i =  (-1)^i C_i$, hence 
\beq
  \sum_{j=0}^{\text{min}(i, 2l-i)} (-1)^j \varepsilon_j = \dfrac{(-1)^{l+i}}{w_0}\sum_{j=0}^{i} a_j 2^{-2(l-j-1)} \binom{2(l-j)}{i-j}. 
\label{Eq:Blmatchup}
\eeq

We claim that 
\begin{equation}
    \varepsilon_i = \dfrac{(-1)^{l}}{w_0}\sum_{j=0}^{l}a_j 2^{-2(l-i-1)}\binom{2l-2j+1}{i-j}.
    \label{Eq:Blepsilon}
\end{equation}

 \begin{proof}
  The $i=0$ case is clear giving $\varepsilon_0  = \dfrac{(-1)^l}{w_0} 2^{-2(l-1)}a_0$ obtained by taking $j = 0$.  
  
So suppose $0<i \leq l$. Then from (\ref{Eq:Blmatchup}) we get
\bea
 \varepsilon_i & = & \dfrac{(-1)^l}{w_0}\sum_{j=0}^i a_j 2^{-2(l-j-1)}\binom{2(l-j)}{i-j} + \dfrac{(-1)^l}{w_0}\sum_{j=0}^{i-1}a_j 2^{-2(l-j-1)} \binom{2(l-j)}{i-1-j} \nn \\
     &=& \dfrac{(-1)^l}{w_0}\left(a_i2^{-2(l-i-1)} + \sum_{j=0}^{i-1} a_j 2^{-2(l-j-1)}\left(\binom{2(l-j)}{i-j} + \binom{2(l-j)}{i-1-j}  \right) \right).\nn \\
\eea
 Furthermore, 
\begin{align*}
    \binom{2(l-j)}{i-j}+\binom{2(l-j)}{i-1-j}  & \, = \dfrac{(2(l-j))!}{(i-j)!(2l-i-j)!} + \dfrac{(2(l-j))!}{(i-j-1)!(2l-i-j+1)!} \\
    & \, = \binom{2l-2j+1}{i-j},
\end{align*}
which gives the result for $i \leq l$.  Hence, since $\varepsilon_i = \varepsilon_{2l+1-i}$, we have (\ref{Eq:Blepsilon}) for all $i$.

 \end{proof}


\subsubsection{$\RR=C_l$}

In the case of $C_l$, we consider (\ref{Def:DSZZSuppot}) with $k = l, m = 0$ which is
\begin{equation}
    \lambda^{\rm DSZZ}(l, l, 0) =  \sum_{j=0}^{l}a_j 2^{-2(l-j)}\left(\sqrt{\mu}  + \dfrac{1}{\sqrt{\mu}} \right)^{2(l-j)}. 
    \label{Eq:ClDZSuppot}
\end{equation}
Simplifying (\ref{Eq:ClDZSuppot}) gives: 
 \begin{equation}
   \dfrac{(-1)^{l-1} w_0}{\mu^l} \sum_{j=0}^{l}\dfrac{(-1)^{l-1}a_j 2^{-2}(l-j) (\mu+1)^{2(l-j)}}{w_0 \mu^{-j}} =   \dfrac{(-1)^{l-1} w_0}{\mu^l} \sum_{\beta=0}^{2l} C_{\beta} \mu^\beta  ,
 \end{equation}
where we again have used the binomial theorem, and with
\begin{equation}
    C_\beta = \dfrac{(-1)^{l-1}}{w_0}\sum_{j=0}^{\beta} a_j 2^{-2(l-j)} \binom{2(l-j)}{\beta-j}. 
\end{equation}
Thus, the equivalence is obtained in the case of $C_l$ by letting 

\begin{equation}
    \varepsilon_i \mapsto \dfrac{(-1)^{l+i-1}}{w_0} \sum_{j=0}^{i} a_j 2^{-2(l-j)} \binom{2(l-j)}{i-j}.
\end{equation} 

\subsubsection{$\RR=D_l$}

For $D_l$, we want to consider (\ref{Def:DSZZSuppot}) with $k = l-2, m=1$, which is of the form

\begin{equation}
    \lambda^{\rm DSZZ}(l, l-2, 1) = \dfrac{4\mu \sum_{j=0}^{l} a_j 2^{-2(l-j-1)}\left( \sqrt{\mu} + \dfrac{1}{\sqrt{\mu}}\right)^{2(l-j-1)}}{(\mu-1)^2}. 
\end{equation}

This is equivalent to
\begin{equation}
\dfrac{(-1)^{l-1} w_0}{(\mu-1)^2 \mu^{l-2} (\mu+1)^2} \sum_{j=0}^{l} \dfrac{(-1)^{l-1}a_j(\mu+1)^{2(l-j)}}{ 2^{2(l-j-2)}w_0 \mu^{-j}} = \dfrac{(-1)^{l-1} w_0}{\mu^{l-2}(\mu^2-1)^2}\sum_{\beta=0}^{2l} C_\beta \mu^{\beta},
\end{equation}
with 
\begin{equation}
    C_{\beta} = \dfrac{(-1)^{l-1}}{w_0} \sum_{}^{} a_j 2^{-2(l-j-2)} \binom{2(l-j)}{\beta-j}, 
\end{equation}
where, again, the binomial theorem has been used. Hence, the map 
\begin{equation}
    \varepsilon_i \mapsto \dfrac{(-1)^{l+i-1}}{w_0} \sum_{j=0}^{i} a_j 2^{-2(l-j-2)} \binom{2(l-j)}{\beta-j}
\end{equation}
gives the equivalence.

\subsection{Superpotentials for exceptional Lie groups}
We present here the construction of the spectral curve for the exceptional types $E_6$, $E_7$, $F_4$, and  $G_2$. The $E_8$ case was treated extensively in  \cite{Brini:2017gfi} and \cite{Brini:2019agj}, and we only give a very brief presentation here.   

As for the classical cases, the construction of the superpotential hinges on determining the character relations \eqref{eq:chardec} for all $k$. Explicitly, for all dominant weights $\varpi=\sum_i \varpi_i \omega_i\in \Lambda^+_w(\RR)$ we should determine $N^{(\omega,k)}_\varpi \in \bbZ$ such that
\beq
\mathfrak{p}^\omega_k = \sum_{\varpi \in \Lambda^+_w(\RR)} N^{(\omega,k)}_\varpi \prod_{i=1}^{l_\RR} \chi_i^{\varpi_i}.
\label{eq:decnum}
\eeq
\begin{defn}
A set of dominant weights $\Pi_{\omega} \subset  \Gamma(\wedge^\bullet \rho_\omega)$ is called \emph{pivotal} for $\omega$ if, $\forall \varpi' \in \Gamma(\wedge^\bullet \rho_\omega)$,  $\exists \varpi \in \Pi_{\omega}$ such that $\varpi' \preceq \varpi$.
\label{def:pivotset}
\end{defn}

Here $\varpi' \preceq \varpi$ denotes the canonical partial ordering of weights, i.e. $\varpi' \preceq \varpi \Leftrightarrow \varpi-\varpi' = \sum_{i=1}^{l_\RR} n_i \a_i$ with $n_i \geq 0$.  \cref{def:pivotset} then states that a set $\Pi_\omega$ of dominant weights is pivotal for a representation $\rho_\omega$ if it is contained in the weight system of the exterior algebra of $\rho_\omega$, and all the other weights in $\Gamma(\wedge^\bullet \rho_\omega)$ are lower, in the partial order, than some of element of $\Pi_\omega$.\\
It will be useful, in the following, to consider pivotal sets that are as small as possible. As for the classical cases, we take $\omega$ to sit in a minimal non-trivial orbit of $\cW_\RR$, as described in \cref{tab:minex}.

\begin{example}
Consider the decomposition of the exterior algebra of $\rho_\omega$ into irreducible representations,
\beq
\Gamma(\wedge^\bullet \rho_\omega) = \bigoplus_{\omega'} \mathrm{Mult}_{\wedge^\bullet \rho_\omega}(\omega') \rho_{\omega'},
\label{eq:irrepext}
\eeq 
and let $\Pi_\omega$ denote the finite set of dominant weights appearing on the r.h.s. with non-zero multiplicity. Then $\Pi_\omega$ is pivotal for $\omega$, although not necessarily of minimal cardinality. Suppose e.g. $\omega=\alpha_0$ is the highest root, so that $\rho_\omega= \mathfrak{g}$ is the adjoint representation. Then the dominant weight given by twice the Weyl vector, $2\mathsf{w}=2 \sum_i \omega_i=\sum_{\a>0} \a$, appears with multiplicity $\mathrm{Mult}_{\wedge^\bullet \mathfrak{g}}(\mathsf{w})=1$ in \eqref{eq:irrepext}, and it is higher than any other highest weight in the decomposition of $\wedge^\bullet \rho_\omega$ into irreducibles. In this case, $\Pi_{\alpha_0}=\{\mathsf{w}\}$ is pivotal and of minimal cardinality. 
\end{example}

\begin{table}[!h]
\begin{tabular}{|c|c|c|c|}
\hline
$\RR$ & $\omega$ & $\rho_\omega$ & $|\mathfrak{I}_\omega|$ \\
\hline
\hline
$E_6$ & $[100000]$ (resp. $[000010]$) & $\mathbf{27}_{E_6}$ (resp. $\overline{\mathbf{27}}_{E_6})$ & $111$ \\
\hline
$E_7$ & $[0000010]$ & $\mathbf{56}_{E_7}$ & $907$\\
\hline
$E_8$ & $[00000010]$ & $\mathbf{248}_{E_8}$ & $950077$ \\
\hline
$F_4$ & $[0001]$ & $\mathbf{26}_{F_4}$ & $74$ \\
\hline
$G_2$ & $[10]$ & $\mathbf{7}_{G_2}$ & $5$ \\
\hline
\end{tabular}
\caption{Highest weights of minimal representations for exceptional root systems. The last column indicates the cardinality of their sets of admissible exponents (\cref{def:admexp}).}
\label{tab:minex}
\end{table}

\begin{lem}
The sets of dominant weights
\beq
\Pi_{\omega} \coloneqq
\begin{cases}
\big\{ [010120], [120010], [200200], [110110], [001030],  & \omega=[100000]_{E_6}, \\
[030000] , [020020], [000041], [000050]\big\}, & \\
\big\{[0002022], [0001113], [0100132], [0101041], [1001051], & \\
[1001061], [0011031], [0010070], [0000204], [0110050],  & \omega=[0000010]_{E_7}, \\
[0010070], [1100070], [1000090], [0003011], [0020040], & \\
[0004000], [0000105], [00000100], [0000006]\big\}, & \\
\{ [22222222]\}, & \omega=[00000010]_{E_8},  \\
\{ [0022]\}, & \omega=[0001]_{F_4},  \\
\{ [20]\}, & \omega=[10]_{G_2}, 
\end{cases}
\label{eq:pivot}
\eeq
are pivotal and of minimal cardinality for $\omega$.
\label{lem:pivot}
\end{lem}
For $\RR = E_8$, $F_4$ and $G_2$ the weight system of $\rho$ is the set of short roots of $\RR$, and the single element of its minimal pivotal set is then the sum of the positive short roots. For $\RR=E_6$, $E_7$ the pivotal sets of minimal cardinality in \eqref{eq:pivot} can be constructed by direct inspection of the weight system.
\begin{defn}
Let $\Pi_\omega$ be as in \cref{lem:pivot}.  We call the finite set
  \beq
  \mathfrak{I}_\omega:= \l\{\iota \in (\bbZ_{\geq 0})^{l_\RR} \bigg| \exists~\varpi=\sum_k \varpi_k \omega_k \in \Pi_\omega \mathrm{~s.t.~}\sum_k (C_\RR)^{-1}_{jk}(\iota_k -\varpi_k)\in \bbZ_{\leq 0} \r\}
  \label{eq:admexp}
  \eeq
the  set of \emph{admissible exponents} of the exterior algebra $\wedge^\bullet \rho_\omega$. 
\label{def:admexp}
\end{defn}
In other words, $\iota$ is admissible if and only if the corresponding dominant weight $\varpi_\iota \coloneqq \sum_i \iota_i \omega_i \preceq \varpi'$ for some weight $\varpi'$ in the minimal pivotal sets of \cref{lem:pivot}. We will use the short-hand notation $\varpi_\iota \preceq \Pi_\omega$ when this happens. The terminology is justified by the following

%
%
\begin{lem}
Let $\iota \in (\bbZ_{\geq 0})^{l_\RR}$, 
 $\iota \notin \mathfrak{I}_\omega$. Then 
$N^{(\omega,k)}_ {\varpi_\iota}=0$ for all $k$.
\label{lem:vanish}
\end{lem}
\begin{proof}
Consider the representation space version of \eqref{eq:decnum},
\beq
\wedge^k \rho_\omega = \bigoplus_{\iota} N^{(\omega,k)}_{ \varpi_\iota} \bigotimes_{i=1}^{l_\RR} \rho_{i}^{\iota_i}.
\label{eq:decnumrep}
\eeq
 By \cref{def:pivotset}, we have $\varpi \preceq \Pi_\omega$ for all $\varpi \in \Gamma(\wedge^k \rho_\omega)$, and furthermore, by \cref{def:admexp}, the $l_\RR$-tuple of its coefficients in the basis of fundamental weights is admissible. 
 Equivalently, if $\iota' \in (\bbZ_{\geq 0})^{l_\RR} \setminus \mathfrak{I}_\omega$ is not admissible, then the corresponding dominant weight $\varpi_{\iota'} \coloneqq \sum_i \iota'_i \omega_i \notin \Gamma(\wedge^k \rho_\omega)$ is not in the exterior algebra: taking multiplicities of \eqref{eq:decnumrep} at  $\varpi_{\iota'}$ then gives 
 \beq 
0= \mathrm{Mult}_{\wedge^k \rho_\omega} (\varpi_{\iota'}) = \sum_\iota N^{(\omega,k)}_{ {\varpi_\iota}} \mathrm{Mult}_{\otimes_i \rho_{i}^{\iota_i}} (\varpi_{\iota'}).
\label{eq:multzero}
 \eeq 
Further notice that
\beq 
\mathrm{Mult}_{\otimes_i \rho_{i}^{\iota_i}} (\varpi_{\iota'}) = 
\begin{cases}
1\,, & \iota'=\iota, \\
0\,, & \iota' \neq \iota \hbox{ and } \iota'-\iota \notin (\bbZ_{\geq 0})^{l_\RR}.
\end{cases}
\label{eq:multtens}
\eeq 
Both equalities in \eqref{eq:multtens} are consequences of the (elementary) fact that the weights of a tensor product representation are given by the sums of the weights of its factors, $\Gamma(V \otimes W) = \{v+w| v \in \Gamma(V), w \in \Gamma(W)\}$. In particular, the first equality  states that the weight space of $\varpi_{\iota}$ in $\otimes_i \rho_{i}^{\iota_i}$ is 1-dimensional, corresponding to the unique decomposition of $\varpi_\iota = \sum_i \iota_i \omega_i$ as a sum of highest weights for each factor; and the second is the assertion that $\varpi_{\iota'} \notin \Gamma(\otimes_i \rho_{i}^{\iota_i})$ when $\iota'_j>\iota_j$ for some $j$. As a consequence, \eqref{eq:multzero} is a linear homogeneous system in the unknowns $N_ {\varpi_\iota}^{(\omega, k)}$ with trivial kernel: the integral matrix $\mathsf{M}$ with coefficients $(\mathsf{M})_{\iota',\iota} \coloneqq \mathrm{Mult}_{\otimes_i \rho_{i}^{\iota_i}} (\varpi_{\iota'})$, in a choice of basis where the column $(\mathsf{M})_{\iota'}$ is to the left of the column $(\mathsf{M})_{\iota}$ whenever $\iota'-\iota \in (\bbZ_{\geq 0})^{l_\RR}$, is upper-triangular and with ones on the diagonal. The claim then follows.

\end{proof}

By \cref{lem:vanish}, the sum in the polynomial character decomposition \eqref{eq:decnum} localises on the set of admissible exponents, whose cardinality $|\mathfrak{I}_\omega|$ is shown in \cref{tab:minex}.

\begin{cor}  Fix a bijection $\sigma: \{1, \dots, |\mathfrak{I}_\omega|\} \to \mathfrak{I}_\omega$ inducing a total order on the set of admissible exponents, and let 
$\mathsf{N}_\omega \in \mathrm{Mat}_{|\mathfrak{I}_\omega| \times \dim{\rho_\omega}}(\bbZ)$ with 
$(\mathsf{N}_\omega)_{lk} \coloneqq N^{(\omega,k)}_{\varpi_{\sigma(l)}}$. Then, there exist explicit rational matrices $\mathsf{A}_\omega \in \mathrm{GL}_{|\mathfrak{I}_\omega|}(\bbQ)$, $\mathsf{B}_\omega \in \mathrm{Mat}_{|\mathfrak{I}_\omega| \times \dim{\rho_\omega}}(\bbQ)$ such that $\mathsf{N}_\omega=\mathsf{A}_\omega \mathsf{B}_\omega$. In particular, $\mathsf{N}_\omega$ is explicitly computable for all $\omega$.
\label{cor:nomega}
\end{cor}

\begin{proof}
 Fix 
$\cQ \coloneqq \{(Q^{(\kappa)}_1, \dots, Q^{(\kappa)}_{l_\RR}) \in \bbQ^{l_\RR}\}_{\kappa =1}^{|\mathfrak{I}_\omega|}$ rational points in $\cT_\RR$ in general position, and define
\bea
(\mathsf{D_\omega})_{\iota,\kappa} & \coloneqq & \prod_{i=1}^{l_\RR} \chi_i^{\iota_i}(Q^{(\kappa)}_1, \dots, Q^{(\kappa)}_{l_\RR}), 
\label{eq:domega}
\\
(\mathsf{B_\omega})_{\kappa,l} & \coloneqq & \chi_{\wedge^l \rho_{\omega}}(Q^{(\kappa)}_1, \dots, Q^{(\kappa)}_{l_\RR}).
\label{eq:bomega}
\eea  
Evaluating \eqref{eq:decnum} at $\cQ$ amounts to saying that
\beq 
\mathsf{D_\omega} \mathsf{N_\omega}=\mathsf{B_\omega}.
\label{eq:linsys}
\eeq 
The identification of the Weyl character ring as an integral polynomial ring generated by the fundamental characters, \eqref{eq:chardec}, assures that $\det\mathsf{D_\omega} \neq 0$  for general $\cQ$, and \eqref{eq:linsys} is a rank-$|\mathfrak{I}_\omega|$ linear problem over the rationals. We then have $\mathsf{A}_\omega = \mathsf{D}_\omega^{-1}$, and furthermore, the integrality of the coefficients in \eqref{eq:chardec} assures that $\mathsf{N}_\omega \in \bbZ$. Furthermore, given any such $\cQ$, the rational matrices $\mathsf{A}_\omega$, $\mathsf{B}_\omega$, and $\mathsf{N}_\omega$ are computable in an entirely explicit manner, as follows: the fundamental characters on the r.h.s. of \eqref{eq:domega} are given as
\beq 
\chi_i(Q^{(\kappa)}_1, \dots, Q^{(\kappa)}_{l_\RR}) = \sum_{\omega' \in \Gamma(\rho_i)} \prod_{j=1}^{l_\RR} \l(Q^{(\kappa)}_j\r)^{\omega'_j}, 
\eeq 
and their value can be computed from the known expressions of the elements of the fundamental weight systems $\Gamma(\rho_i)$, $i=1,\dots, l_\RR$. The exterior characters in \eqref{eq:bomega} can similarly be computed from the knowledge of $\Gamma(\rho_\omega)$ alone to evaluate the power sum virtual characters of $\rho_\omega$, 
\beq 
\chi_{\rho_\omega}((Q^{(\kappa)}_1)^n, \dots, (Q^{(\kappa)})^n_{l_\RR}) = \sum_{\omega' \in \Gamma(\rho_\omega)} \prod_{j=1}^{l_\RR} \l(Q^{(\kappa)}_i\r)^{n \omega'_j},
\eeq 
from which the exterior characters $\chi_{\wedge^k \rho_\omega}$ can be recovered using the Girard--Newton identities,
\beq 
k \chi_{\wedge^k \rho_\omega}= \sum_{n=1}^k(-1)^{n-1} \chi_{\wedge^{k-n} \rho_\omega}\chi_{\rho_\omega}.
\eeq 
The integral linear system \eqref{eq:linsys} can then be efficiently solved explicitly for $\mathsf{N}_\omega$  using Dixon's $p$-adic lifting algorithm.
\end{proof}
\medskip

Using \eqref{eq:charpol}--\eqref{eq:shiftfun}, \eqref{eq:decnum} and the complete calculation of the coefficients $N_ {\varpi_\iota}^{(\omega,k)}$ from \cref{lem:vanish,cor:nomega} we can construct Landau--Ginzburg superpotentials for all $\RR$, as we now describe for the remaining exceptional cases.

\subsubsection{$\RR=E_6$}

The Dynkin diagram for the affine $E_6$ root system is given in \cref{fig:dynkin}, for which the canonical label is $\bar k=3$. In this case there are two nontrivial minimal irreducible representations, the 27-dimensional fundamental representation $\rho_{1}=(\mathbf{27})$ with highest weight $\omega_1$ and its dual representation $\rho_{5}=(\overline{\mathbf{27}})$ with highest weight $\omega_5$, related by complex conjugation. \\

\begin{figure}[!h]
\includegraphics[scale=.6]{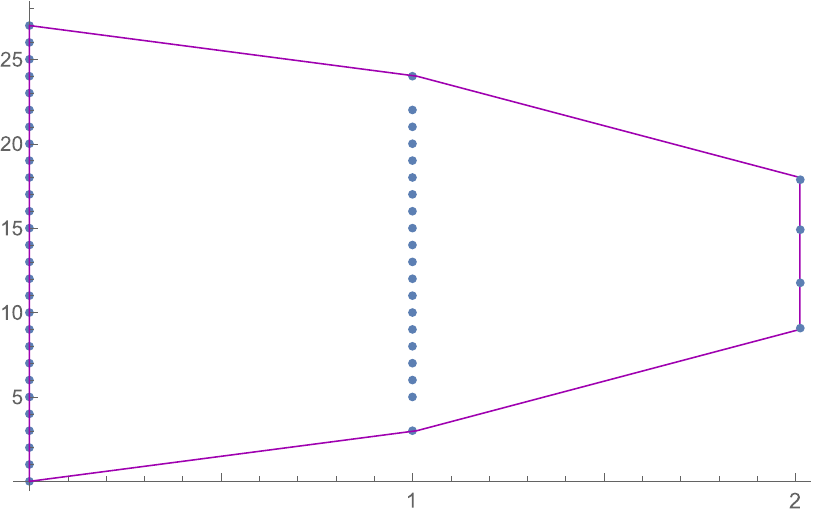}
\caption{Newton polygon for the $E_6$-spectral curve.}
\label{Fig:E6Newt}
\end{figure}

The character relations \eqref{eq:decnum} are given explicitly for $\omega=\omega_1$ in the appendix.  The resulting family of spectral curves has fibres which are hyperelliptic curves of genus 5, with Newton polygon as shown in \cref{Fig:E6Newt}, and ramification profile over $\infty$ 
  
\begin{equation}
 \left( \, \overbrace{3,6}^{\mu = 0}, \, \overbrace{6,3}^{\mu = \infty}, \, \overbrace{3}^{\mu = \varepsilon_3^j}   \, \right), 
 \label{Eq:E6ram}
\end{equation}
 where $\varepsilon_3$ is a primitive third root of unity. This realises $M_{[100000]_{E_6}}^{\rm LG}$ as a 7-dimensional subvariety of the $42$-dimensional Hurwitz space $H_{g_\omega, \mathsf{n}_\omega}$ with $g_\omega=5$ and $\mathsf{n}_\omega = (5,5,2,2,2,2,2)$, and explicit embedding described by \eqref{eq:e6char}.

\subsubsection{$\RR=E_7$}
The Dynkin diagram for the affine $E_7$ root system is given in Figure \ref{fig:dynkin}, in which we see that the canonical label is $\bar k=3$. For this case there is a unique choice of minimal representation, corresponding to the  56-dimensional fundamental representation having highest weight $\omega_6$. Choosing this representation gives character relations which we include in the appendix for $k=1,\dots, 11$. The resulting family of spectral curves has fibres of genus 33, with a degree 3 morphism to $\bbP^1$  inducing a $3:1$ branched cover of the Riemann sphere with ramification profile over $\infty$

\begin{equation}
 \left( \, \overbrace{12, 6, 4}^{\mu = 0}, \, \overbrace{12, 6, 4}^{\mu = \infty}, \, \overbrace{2}^{\mu = \pm 1}   \, \overbrace{4}^{\mu = \pm i} \right). 
 \label{Eq:E7ram}
\end{equation}

Hence, $M_\omega^{\rm LG}$ is an $8$-dimensional submanifold in the 130-dimensional Hurwitz space $H_{g_\omega,\mathsf{n}_\omega}$, with $g_\omega=33$ and $\mathsf{n}_\omega = (11,5,3,11,5,3,1,1,3,3)$. The associated Newton polygon is shown in \cref{Fig:E7Newt}. 

\begin{figure}[!h]
\includegraphics[scale=0.5]{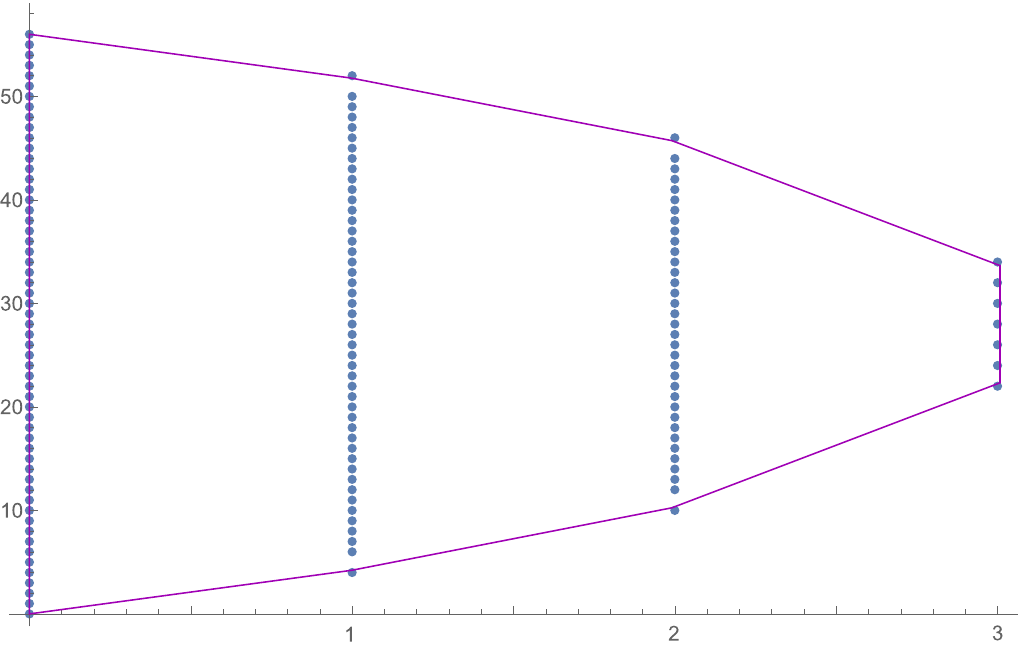}
\caption{Newton polygon for the $E_7$-spectral curve.}
\label{Fig:E7Newt}
\end{figure}

\subsubsection{$\RR=E_8$}

As mentioned, this case was thoroughly treated in \cite{Brini:2017gfi}, \cite{Brini:2019agj}, and we will provide only a brief presentation here. The Dynkin diagram for the affine $E_8$ root system is given in Figure \ref{fig:dynkin}, and the canonical label is, as for all the $E$-types, $\bar k=3$. In this case, the minimal, nontrivial, irreducible representation is the 248-dimensional adjoint representation. Explicit character relations were given in \cite{Brini:2019agj}, where their derivation is explained in detail.  It is shown in \cite{Brini:2017gfi} that the resulting curve is of genus 128 and induces a cover of the Riemann sphere with ramification over $\infty$ at $\mu = 0, \infty$, in addition to second, third and fifth roots of unity, with ramification profile given in \cite[Eq.~(5.34)]{Brini:2017gfi}.  
Explicit flat coordinates and prepotential can also be found in \cite{Brini:2017gfi}. The resulting parent Hurwitz space is of dimension 518.

\subsubsection{$\RR={F_4}$}

The Dynkin diagram for the affine root system of type $F_4$ is shown in \cref{fig:dynkin}. In this case the canonical node is the one corresponding to the fundamental weight $\omega_2$, $\omega=[0001]_{F_4}$, and $\rho_\omega$ will be the 26-dimensional fundamental representation, i.e. the irreducible representation of highest weight $\omega_4$. The character relations \eqref{eq:chardec} are then given as follows:

\bea
\mathfrak{p}^{[0001]_{F_4}}_1 &=& \chi_4, \nn \\
\mathfrak{p}^{[0001]_{F_4}}_2 &=& \chi_1+\chi_3, \nn \\
\mathfrak{p}^{[0001]_{F_4}}_3 &=& \chi_2+\chi_1 \chi_4-\chi_4, \nn \\
\mathfrak{p}^{[0001]_{F_4}}_4 &=& \chi_1^2+\chi_3 \chi_1-\chi_4^2-\chi_2, \nn \\
\mathfrak{p}^{[0001]_{F_4}}_5 &=& \chi_1^2 \chi_4-\chi_4^3 -\chi_1 \chi_4-2 \chi_2 \chi_4+\chi_3 \chi_4+\chi_4+\chi_3^2-\chi_2+\chi_3, \nn \\
\mathfrak{p}^{[0001]_{F_4}}_6 &= & \chi_1^3-\chi_1^2-\chi_4^2 \chi_1-3 \chi_2 \chi_1+\chi_3 \chi_4 \chi_1+\chi_4 \chi_1-\chi_1-\chi_4^3+\chi_4^2-\chi_2 \nn \\ &-& \chi_2 \chi_4    +\chi_3 \chi_4+\chi_4,  \nn \\
\mathfrak{p}^{[0001]_{F_4}}_7 &= & \chi_1 \chi_4^3-\chi_4^3-\chi_1 \chi_4^2+\chi_2 \chi_4^2-2 \chi_1 \chi_4-2 \chi_2 \chi_4-3 \chi_1 \chi_3 \chi_4+2 \chi_3 \chi_4 \nn \\ & + &\chi_4 +2 \chi_4^2 +\chi_1^2-\chi_1+\chi_1 \chi_2-\chi_2+\chi_1^2 \chi_3-\chi_1 \chi_3-2 \chi_2 \chi_3,  \nn \\
\mathfrak{p}^{[0001]_{F_4}}_8 &= & \chi_4^2 \chi_1^2-\chi_1^3-2 \chi_3 \chi_1^2+\chi_2 \chi_1-\chi_3 \chi_1+\chi_2 \chi_4 \chi_1-\chi_3 \chi_4 \chi_1-\chi_4 \chi_1\nn + \chi_4^3\\ & + &  \chi_3\chi_4^3  -2 \chi_3^2-\chi_4^2-\chi_2+3 \chi_2 \chi_3+\chi_3-3 \chi_3^2 \chi_4  + 2 \chi_2 \chi_4   -2 \chi_3 \chi_4 +\chi_4,  \nn \\
\mathfrak{p}^{[0001]_{F_4}}_9 &= & \chi_4^5-\chi_1 \chi_4^3-4 \chi_3 \chi_4^3-2 \chi_4^3+2 \chi_1 \chi_4^2+4 \chi_2 \chi_4^2+\chi_1 \chi_3 \chi_4^2+2 \chi_3^2 \chi_4-\chi_2  \nn \\ & + & \chi_1 \chi_2 \chi_4+\chi_2^2+2 \chi_1 \chi_3 \chi_4+3 \chi_3 \chi_4-2 \chi_1^2+\chi_2\chi_4-2 \chi_1 \chi_3^2 - 2 \chi_1 \chi_2 \nn \\ &+& 2 \chi_1 \chi_4-2 \chi_1^2 \chi_3-2 \chi_1 \chi_3-\chi_2 \chi_3,   \nn \\
\mathfrak{p}^{[0001]_{F_4}}_{10} &= &  \chi_1 \chi_4^4-5 \chi_3 \chi_4^3-2 \chi_4^3-2 \chi_1^2 \chi_4^2+3 \chi_2 \chi_4^2-3 \chi_1 \chi_3 \chi_4^2-\chi_3 \chi_4^2+5 \chi_3^2 \chi_4 \nn \\ & + & \chi_4^5-\chi_1 \chi_4^3+3
\chi_1 \chi_4+\chi_1 \chi_2 \chi_4+\chi_2 \chi_4+3 \chi_1 \chi_3 \chi_4+\chi_2 \chi_3 \chi_4+4 \chi_3 \chi_4 \nn \\ & -& \chi_4+\chi_1^3-\chi_2^2+\chi_1 \chi_3^2+3 \chi_3^2-\chi_1 \chi_2+\chi_2+2 \chi_1^2 \chi_3+ 3
\chi_1 \chi_3-3 \chi_2 \chi_3,  \nn \\
\mathfrak{p}^{[0001]_{F_4}}_{11} &= & \chi_1 \chi_4^4+\chi_2 \chi_4^3-\chi_3 \chi_4^3+\chi_4^3-\chi_1^2 \chi_4^2-2 \chi_1 \chi_4^2-2 \chi_2 \chi_4^2-5 \chi_1 \chi_3 \chi_4^2 \nn \\ &-& 2 \chi_4^2+ 2 \chi_3^2 \chi_4-\chi_1
\chi_4 \chi_2 \chi_4+\chi_1 \chi_3 \chi_4-3 \chi_2 \chi_3 \chi_4+\chi_3^3+3 \chi_1^2 \nn \\ & + &\chi_2^2 +4 \chi_1 \chi_3^2+2 \chi_3^2+\chi_1+3 \chi_1 \chi_2+3 \chi_1^2 \chi_3+5 \chi_1 \chi_3 + 2 \chi_2 \chi_3+\chi_3,  \nn \\
\mathfrak{p}^{[0001]_{F_4}}_{12} &= & \chi_4^4-\chi_4^5-\chi_1 \chi_4^4+3 \chi_1 \chi_4^3+2 \chi_3 \chi_4^3+2 \chi_1^2 \chi_4^2+\chi_3^2 \chi_4^2-\chi_1 \chi_4^2-\chi_2 \chi_4^2 \nn \\ & + & \chi_3
\chi_4^2-5 \chi_1 \chi_4-4 \chi_2 \chi_4-5 \chi_1 \chi_3 \chi_4-3 \chi_2 \chi_3 \chi_4 +\chi_3 \chi_4+\chi_4-\chi_1^3 \nn \\ & + &2 \chi_2^2-\chi_1 \chi_3^2+3 \chi_1 \chi_2-3 \chi_1 \chi_3-2 \chi_2 \chi_3+\chi_1 \chi_3 \chi_4^2 ,  \nn \\
\mathfrak{p}^{[0001]_{F_4}}_{13} &= & 2 \chi_4^4 -2 \chi_4^5-2 \chi_1 \chi_4^4+4 \chi_1 \chi_4^3-2 \chi_2 \chi_4^3+6 \chi_3 \chi_4^3+4 \chi_1^2 \chi_4^2+2 \chi_3^2 \chi_4^2 \nn \\  &-& 2 \chi_1 \chi_4^2-2
\chi_2 \chi_4^2+4 \chi_1 \chi_3 \chi_4^2-2 \chi_4^2-2 \chi_1^2 \chi_4-4 \chi_3^2 \chi_4-2 \chi_1 \chi_4\nn \\ & + &2 \chi_1 \chi_2 \chi_4-4 \chi_1 \chi_3 \chi_4+2 \chi_2 \chi_3 \chi_4-4 \chi_3 \chi_4-2
\chi_3^3-4 \chi_1^2 +2 \chi_4^3 \nn \\ & -& 2 \chi_2^2-4 \chi_1 \chi_3^2-4 \chi_3^2 -4 \chi_1 \chi_2-4 \chi_1^2 \chi_3-4 \chi_1 \chi_3-2 \chi_2 \chi_3+ 2, 
\label{eq:f4char}
\eea
with $\mathfrak{p}^{[0001]_{F_4}}_{26-i} = \mathfrak{p}^{[0001]_{F_4}}_i$. Note that the above relations in the character ring follow from those for $\RR=E_6$ and $\rho=(\mathbf{27})$ or $\rho=(\overline{\mathbf{27}})$ by folding; in particular $M_{F_4,[0001]_{F_4}}^{\rm LG}$ sits inside $M_{[100000]_{E_6}}^{\rm LG}$ as the fixed locus of the involution $w_1 \leftrightarrow w_5$, $w_2 \leftrightarrow w_4$. The generic fibre $\overline{C_w^{\omega_4}}$ is a genus 4 hyperelliptic curve with ramification over $\infty$
\begin{equation}
 \left( \, \overbrace{3,6}^{\mu = 0}, \, \overbrace{3,6}^{\mu = \infty}, \, \overbrace{3}^{\mu = \varepsilon_3}, \, \overbrace{3}^{\mu = \varepsilon_3^2}   \, \right), 
 \label{Eq:F4ram}
\end{equation}
 where $\varepsilon_3$ is a primitive third root of unity, and the associated Newton polygon is shown in \cref{Fig:F4Newt}. The extended affine $F_4$-Frobenius manifold is thus realised as a  5-dimensional submanifold of the 36-dimensional Hurwitz space $H_{g_\omega, \mathsf{n}_\omega}$, with $g_\omega=4$ and $\mathsf{n}_\omega = (5,5,2,2,2,2)$.

\begin{figure}[t]
\includegraphics[scale=0.5]{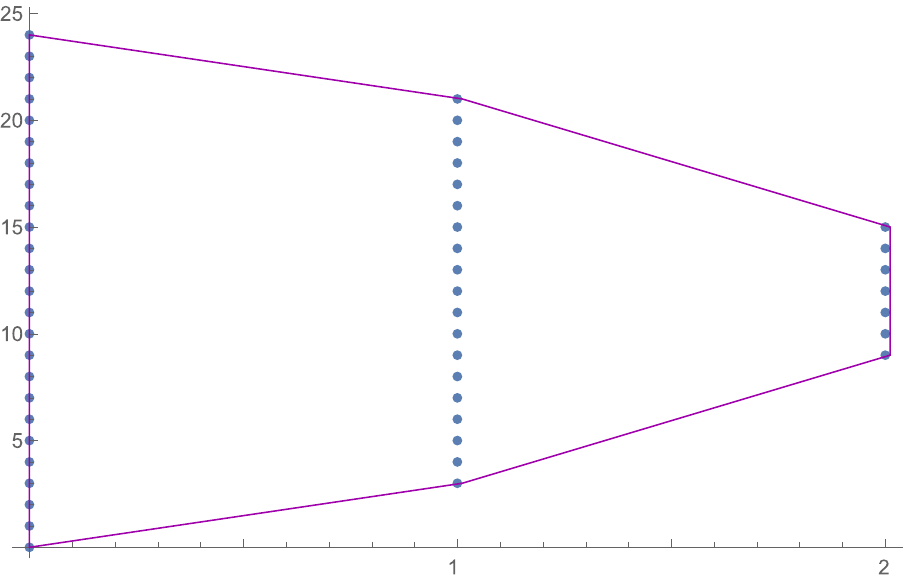}
\caption{Newton polygon for the $F_4$-spectral curve.}
\label{Fig:F4Newt}
\end{figure}

\subsubsection{$\RR={G_2}$}
\label{sec:G2LG}

The Dynkin diagram for the affine $G_2$ root system is given in Figure \ref{fig:dynkin}, and the canonical label is $\bar k=2$.  In this case,  $\rho_{\omega_1}=(\mathbf{7})$ is the 7-dimensional fundamental representation which is the irreducible representation with highest weight $\omega_1$. We obtain the character relations
\begin{equation}
     \mathfrak{p}^{[10]_{G_2}}_1 = \chi_1, \quad  \mathfrak{p}^{[10]_{G_2}}_2 = \chi_1 +  \chi_2, \quad \mathfrak{p}^{[10]_{G_2}}_3 = \chi_1^2 - \chi_2,
\label{eq:g2char}
\end{equation}
and $\mathfrak{p}^{[10]_{G_2}}_{7-i} =  \mathfrak{p}^{[10]_{G_2}}_i $, hence
\begin{equation*}
    \mathcal{P}_{[10]_{G_2}}(\lambda, \mu; w_0, w_1, w_2)  \equiv \sum_{i=0}^{3} (-1)^{i} \mathfrak{p}^{[10]_{G_2}}_i\left(w_1, w_2 - \frac{\lambda}{w_0}\right) \mu^i(1-\mu^{7-2i}),
    \label{Eq:G2curve}
\end{equation*}
and solving for $\lambda$
\bea
    \lambda &=&   \dfrac{w_0}{\mu^2 (\mu + 1)^2}\bigg(\mu^6 + (1-w_1)\mu^5 + (1+w_2)\mu^4 + (1-w_1^2+2w_2)\mu^3 \nn \\ &+& (1+w_2)\mu^2 + (1-w_1)\mu + 1 \bigg).
\label{eq:g2spot}
\eea

As was the case for $\RR=F_4$, the same superpotential could be obtained from the LG model of $\omega=[1000]_{D_4}$ by the order three folding of the $D_4$ Dynkin diagram, and $M_{[10]_{G_2}}^{\rm LG}$ sits inside $M_{[1000]_{D_4}}^{\rm LG}$ as the fixed locus of the triality action sending $(w_1, w_3,w_4) \to (w_{\epsilon(1)},w_{\epsilon(3)},w_{\epsilon(4)})$ with $\epsilon \in S_3$. The outcome is a family of rational functions in $\mu$, with three poles at $\mu = 0,  -1, \infty$, all of order two. This means that the resulting Frobenius manifold is a 3-dimensional sublocus in the 7-dimensional Hurwitz space $H_{0,\mathsf{n}_\omega }$, with $\mathsf{n}_\omega = (1,1,1)$.

\section{Mirror symmetry}
\label{Section:MS}

Having constructed $M_\omega^{\rm LG}$ for all Dynkin types, we now move on to proving \cref{thm:mirror}. 

As in \cite[Lecture~5]{Dubrovin:1994hc}, and with notation as in \cref{sec:hurgen}, define local coordinates $\kappa_i$ in a neighbourhood of $\infty_i$  by Lagrange inversion of $\lambda(\kappa)=\kappa_i^{n_i+1} +\cO(1)$. Consider then the following functions on $H_{g_\omega,\mathsf{n}_\omega}$:
\begin{subequations}
\begin{equation}	
    \tau_{i;\alpha} \coloneqq \,  \underset{\infty_i}{\mathrm{Res}} \,  \kappa_i^{-\alpha} \log\mu \, \rd\lambda, \quad \alpha = 1, \cdots, n_i,
    \label{Eq:5.1a}
\end{equation} 
\begin{equation}
    \tau_j^{\rm ext} \coloneqq \mathrm{p.v. } \,  \int_{\infty_0}^{\infty_j} \rd \log\mu, \quad j=1, \cdots, \ell(\mathsf{n}_\omega),
    \label{Eq:5.1b}
\end{equation}
\begin{equation}
    \tau_k^{\rm res} \coloneqq  \underset{\infty_k}{\mathrm{Res}} \, \lambda \, \rd\log\mu, \quad k = 0, \cdots, \ell(\mathsf{n}_\omega),
    \label{Eq:5.1c}
\end{equation}
\end{subequations}
where the principal value, {\rm p.v.}, indicates subtraction of the divergent part in $\kappa_i$.

\begin{lem}
There exist complex numbers $f_{i,\alpha}^{(r)}$, $q_j$, $r_k \in \bbC$  and holomorphic functions $\{t_i: M^{\rm LG}_\omega \to \bbC\}_{i=1,\dots, l_{\RR}+1}$ such that 
\bea 
\tau_j^{\rm ext}\big|_{M_\omega^{\rm LG}}
  &=&  q_j t_{l_\RR+1}   , \nn \\
    \tau_{i;\alpha}|_{M_\omega^{\rm LG}} &=& \sum_{r=1}^{l_\RR} f_{i,\alpha}^{(r)} t_r,  \nn \\
    \tau_k^{\rm res}|_{M_\omega^{\rm LG}} &=& r_k t_{\bar k}.
    \label{eq:taut}
\eea     
Moreover, $(t_1, \dots, t_{l_\RR+1})$ are flat coordinates for the metric \eqref{eq:reseta} on $\cM_\omega^{\rm LG}$.
\label{Thm:5.1}
\end{lem}
\begin{proof}
Suppose first that $\RR$ is any of the root systems $A_l$, $B_l$, $C_l$, $D_l$ or $G_2$. Then, from the discussion of \cref{sec:DSZZ,sec:G2LG}, we have $g_\omega=0$, and $M_\omega^{\rm LG} \subset H_{0,\mathsf{n}_\omega}$ is a space of rational functions. The statement of the Lemma is then a specialisation of \cite[Thm~5.1]{Dubrovin:1994hc}, which in particular asserts that $\tau_{i,\alpha}$, $\tau_j^{\rm ext}$ and $\tau_k^{\rm res}$ are a complete set of flat coordinates for the metric \eqref{eq:etares} on a genus zero Hurwitz space; and since $M_\omega^{\rm LG}$ is a fixed-locus of an involution acting linearly in these coordinates, it is specified by a linear condition of the form \eqref{eq:taut}. For the remaining four exceptional cases, the linear relations in \eqref{eq:taut} and the constancy of the Gram matrix of the metric \eqref{eq:todaeta} in the chart $(t_1, \dots, t_{l_\RR+1})$ follow from a direct residue calculation from \eqref{Eq:5.1a}--\eqref{Eq:5.1c}, and from making use of the explicit form of the superpotential from \eqref{eq:shiftfun} for each Dynkin type; we omit the details here\footnote{All calculations can be made available to the interested reader upon request.}. 
\end{proof}

\begin{lem}
Let $(t_1, \dots, t_{l_\RR+1})$ be flat coordinates for \eqref{eq:todaeta} as in \cref{Thm:5.1}. Then $t_{l_\RR+1} = \log w_0 /d_{\bar k}$ and, for all $i=1, \dots, l_{\RR}$,
\beq 
t_i(w_0, \dots, w_{l_\RR}) \in w_0^{d_i/d_{\bar k}} \bbZ[w_1, \dots, w_{l_\RR}].
\eeq 
Moreover, the change of variables $w \mapsto t(w)$ has a polynomial inverse 
\beq 
w_i(t_1, \dots, t_{l_\RR+1})\in \bbQ[t_1, \dots, t_{l_\RR+1},\re^{t_{l_\RR+1}}].
\eeq 
\label{lem:flatpol}
\end{lem}

\begin{proof}
A direct calculation from \eqref{Eq:5.1a}--\eqref{Eq:5.1c}, using as above the explicit form of \eqref{eq:shiftfun}, shows that the flat coordinates $(t_1, \dots, t_{l_\RR+1})$ are related (up to normalisation) to $(w_0, \dots, w_{l_\RR})$ as
\bea
t_i &=& w_0^{d_i/d_{\bar k}} \mathfrak{t}_i(w), \nn \\
t_{l_\RR+1} &=& \frac{\log w_0}{d_{\bar k}},
\eea  
where $\mathfrak{t}_i(w)$ are explicit integral polynomials in $(w_1, \dots, w_{l_\RR})$. Moreover, it can be verified directly that 
\bea
\de_{w_j}\mathfrak{t}_i = & 0 &\hbox{ if } d_j>d_i, \nn \\
\deg_{w_j}\mathfrak{t}_i = & 1 & \hbox{ if } d_j=d_i,
\eea 
implying that the inverse function $t \mapsto w(t)$ is a polynomial in $(t_1, \dots, t_{l_\RR}, \re^{t_{l_\RR+1}})$ with rational coefficients.
\end{proof}

The previous two Lemmas put us in a position to prove \cref{thm:mirror}.

\begin{proof}[Proof of \cref{thm:mirror}]
Consider the change-of-variables 
%
\beq
w_i(x_1, \dots, x_{l_\RR}) =  \chi_i(\re^{\mathsf{h}}) = \sum_{\omega' \in \Gamma(\rho_i)} \prod_{j=1}^{l_\RR}\re^{\omega'_j x_j}\,, \quad i=1,\dots l_\RR.
\label{eq:witox}
\eeq
Further identifying 
\beq 
w_0 = \re^{c_\omega x_{l_\RR+1}},
\label{eq:w0tox}
\eeq 
for $c_\omega \in \bbC^\star$, extends this to a local analytic isomorphism sending $$\cM_{\RR}^{\rm DZ} \ni (x_1, \dots, x_{l_\RR}; x_{l_\RR+1}) \mapsto (w_0; w_1, \dots, w_{l_\RR}) \in M_{\omega}^{\rm LG}.$$ We shall now prove that this is, in fact, an isomorphism of Frobenius structures upon checking the four defining properties, {\bf DZ-I} through {\bf DZ-IV}, of the reconstruction \cref{Thm:DZ}. \\

\begin{description}
\item[DZ-I and DZ-II] define holomorphic vector fields $e$, $E\in \Gamma(M_\omega^{\rm LG},TM_\omega^{\rm LG})$ as
\beq 
e \coloneqq \dfrac{\partial_{w_{\bar k}}}{w_0}, \quad E \coloneqq w_0\partial_{w_0}.
\label{eq:eEtoda}
\eeq 
Then, the one-parameter group of isomorphisms generated by the horizontal lift $\delta^{(\mu)}_e$ (resp. $\delta^{(\mu)}_E$) to the universal curve acts on the superpotential by translation (resp. conformal transformations) on the superpotential. To see this, note that, by \eqref{eq:shiftfun},
\bea 
w_0 \mapsto a w_0 & \leadsto & \lambda \mapsto a \lambda \nn, \\
w_{\bar k} \mapsto w_{\bar k} + b/w_0 & \leadsto & \lambda \mapsto \lambda+b. 
\label{eq:eEact}
\eea 
%
Since the unit and Euler vector field \eqref{eq:eEhur} of a Hurwitz--Frobenius manifold are characterised as the generators of the affine action \eqref{Eq:Hurwitzaffine} on the superpotential, \eqref{eq:eEact} identifies $e$ with the unit and $E$ with the Euler vector field of $\cM_\omega^{\rm LG}$. To verify {\bf [DZ-I]} and {\bf [DZ-II]}, it then remains to check that the expressions in \eqref{eq:eEtoda}  coincide with those for the respective vector fields of $\cM_\RR^{\rm DZ}$. A simple calculation from \cref{lem:flatpol} shows that
 \begin{gather}
      e =  \dfrac{1}{w_0}\sum_{i=1}^{l_{\mathcal{R}+1}}\dfrac{\partial t_i}{\partial w_{\bar k}} \partial_{t_i} = \partial_{t_{\bar k}},\label{Eq:LGe} \\  E = w_0\partial_{w_0} = w_0\sum_{i=1}^{l_{\mathcal{R}+1}}\dfrac{\partial t_i}{\partial w_0} \partial_{t_i} = \dfrac{\partial_{t_{l_\RR+1}}}{d_{\bar k}} +\sum_{i=1}^{l_{\mathcal{R}}} \dfrac{d_i t_i}{d_{\bar k}} \partial_{t_i},
     \label{Eq:LGE}
 \end{gather}
thereby matching the expression of the unit and the Euler vector fields in \cref{Thm:DZ}.
\item[DZ-III] 
Let us consider now the Gram matrix of the intersection pairing on $\cM^{\rm LG}_\omega$ in the $x$-chart. Consider first the argument of the residues in  \eqref{eq:todag},
\beq 
\Upsilon_{ij}(p) \coloneqq
\frac{\delta^{(\mu)}_{\de_{x_i}} \lambda ~ \delta^{(\mu)}_{\de_{x_j}} \lambda}{\lambda \mu^2  \de_\mu \lambda}\rd\mu(p),
\label{eq:ups}
\eeq
so that
\beq 
\gamma(\de_{x_i}, \de_{x_j}) = \sum_{l}\underset{p=p^{\rm cr}_l}{\Res}\Upsilon_{ij}(p).
\label{eq:gammaresups}
\eeq 
From \eqref{eq:ups}, we deduce that the pole structure of $\Upsilon_{ij}(p)$ is as follows:
\ben[(i)]
\item it has simple poles at the critical points $\{p_l^{\rm cr}\}$, for which $\rd\lambda(p_i^{\rm cr})=0$;
\item it has at most simple poles at $\lambda(p)=0$, and only when both $i,j \neq l_{\RR}+1$: indeed, from \eqref{eq:w0tox} we have that $\delta^{(\mu)}_{\de_{x_{l_\RR+1}}} \lambda = c_\omega \lambda$, thereby cancelling the zeroes at the denominator when either $i$ or $j=l_\RR+1$;
\item it has at most simple poles at $\mu(p)=0$ (for which $\lambda(p)=\infty$: see \cref{rmk:muadm}) and only when $i=j=l_{\RR}+1$. To see why, notice that locally at a point $p'$ near $\mu(p)=0$ we have
\beq
\lambda(p') = \re^{c_\omega x_{l_\RR+1}} \mu(p')^{-q_p} (r_p+ \cO(\mu)),
\eeq 
where $q(p) \in \bbZ_{>0}$ and $r(p) \in \bbC$. Then, the denominator in \eqref{eq:ups} has a leading Puiseux asymptotics in $\mu$ of the form  $\mu^{1-2 q_p}$, resulting from the combination of the order $q_p$ divergence of $\lambda(\mu)$, the order $q_p+1$ divergence of $\de_\mu \lambda(\mu)$, and the double zero of $\mu^2$. For the numerator, we have  $\delta^{(\mu)}_{\de_{x_i}} \lambda \sim \mu^{\delta_{i, l_{\RR}+1}-q_p}$, since $r(p)$ is $w$-independent by \cref{prop:mu}, so that
\beq 
\ord_{\mu(p)=0} \Upsilon_{ij} = 1-\delta_{i,l_\RR+1}-\delta_{j,l_\RR+1};
\eeq 
\item superficially, there may be further poles to be expected at the critical points $\{q^{\rm cr}_m\}$ of the $\mu$-projection, $\rd\mu(q^{\rm cr}_m)=0$, that is, where the Ehresmann connection \eqref{eq:horder} induced by the $\mu$-foliation is singular and $\delta^{(\mu)}_{\de_{x_i}} \lambda$ possibly develops a pole. However these singularities are offset by a vanishing of the same order of $\rd \mu/{\de_\mu \lambda}$, so that 
\beq 
\ord_{q_m^{\rm cr}} \Upsilon_{ij} = 0.
\label{eq:upsqcr}
\eeq 
\\
\een  

Based on the above, we turn the contour around and equate the sum of residues at the critical points in \eqref{eq:todag} to a much more manageable sum of residues at poles and zeros of $\mu$ and $\lambda$. When $i=j=l_{\RR}+1$, we only have poles at $\mu=0$, and
\bea
\gamma(\de_{x_{l_\RR+1}},\de_{x_{l_\RR+1}})  &=& \cN_\omega
\sum_{\mu(p)=0}\underset{p'=p}{\Res}~ \Upsilon_{l_\RR+1,l_\RR+1}(p') \nn \\
&=& \cN_\omega c_\omega^2
\sum_{\mu(p)=0}\underset{p'=p}{\Res}~ \frac{\lambda}{\mu^2 \de_\mu \lambda} \rd \mu(p') \nn \\
&=& \cN_\omega c_\omega^2
\sum_{\mu(p)=0} \frac{1}{\ord_p \lambda} =  -\sum_{i: \mu(\infty_i)=0} \frac{\cN_\omega c_\omega^2}{(\mathsf{n}_\omega)_i}.
\eea
Setting $c_\omega= \sqrt{d_{\bar k}} (- \cN_\omega \sum_{i: \mu(\infty_i)=0} 1/(\mathsf{n}_\omega)_i)^{-1/2}$ gives $\gamma(\de_{x_{l_\RR+1}},\de_{x_{l_\RR+1}})=d_{\bar k}$. \\

Suppose now that $j=l_\RR+1$, $i<l_\RR+1$. Our analysis above shows that $\Upsilon_{ij}$ is regular outside the critical locus of $\lambda$, and therefore
\beq 
\gamma(\de_{x_i},\de_{x_{l_\RR+1}}) = \gamma(\de_{x_{l_\RR+1}},\de_{x_i})= \delta_{i,l_\RR+1} d_{\bar k}.
\eeq 

Finally, let's look at $i,j < l_{\RR}+1$. In this case, outside $\{p_l^{\rm cr}\}$, $\Upsilon_{ij}$ has only simple poles at the zeroes of $\lambda$, i.e. when
\bea
 0= & \cP_{\omega}(0,\mu) & = \cQ^{\rm red}_{\omega}(\mu)= \prod_{0 \neq \omega'\in\Gamma(\rho_\omega)}(\mu-\re^{\omega'(x)}) \nn \\ & \Longleftrightarrow & \mu = \re^{\omega'(x)}\,, \quad \omega' \in \Gamma(\rho_\omega),
\label{eq:xil0}
\eea 
where we used the shorthand notation $\omega'(x) \coloneqq \sum_{n=1}^{l_\RR} \omega'_n x_n$. The evaluation of the residue at $(\lambda,\mu)=(0,\re^{\omega'(x)})$ gives
\bea
\underset{\mu=\re^{\omega'(x)}}
\Res \Upsilon_{ij}
&=& 
\frac{\delta^{(\mu)}_{\de_{x_i}} \lambda~ \delta^{(\mu)}_{\de_{x_i}} \lambda}{\mu^2 (\de_\mu\lambda)^2}\Bigg|_{\mu=\re^{\omega'(x)}}
=
\frac{\de_{x_i} \cQ_{\omega}^{\rm red}\de_{x_j} \cQ_{\omega}^{\rm red}}{\mu^2 (\de_\mu\cQ_{\omega}^{\rm red})^2}\Bigg|_{\mu=\re^{\omega'(x)}} \nn \\
&=& \frac{\re^{2\omega'(x)} \omega'_i \omega'_j \prod_{\omega'',\omega''' \neq \omega'}\l(\re^{\omega'(x)}-\re^{\omega''(x)}\r)
\l(\re^{\omega'(x)}-\re^{\omega'''(x)}\r)
}{\re^{2\omega'(x)}
  \prod_{\omega'',\omega''' \neq \omega'}\l(\re^{\omega'(x)}-\re^{\omega''(x)}\r)
\l(\re^{\omega'(x)}-\re^{\omega'''(x)}\r)} \nn \\
&=& \omega'_i \omega'_j,
\label{eq:reseta}
\eea
where the second equality makes use of the implicit function
theorem. Summing over $\omega'$ gives
\bea
\gamma(\de_{x_i},\de_{x_{j}}) &=& \cN_\omega \sum_{\omega' \in \Gamma(\rho_\omega)} \omega_i' \omega_j' = \cN_\omega \Tr (\rho_\omega(h_i) \rho_\omega(h_j)) \nn \\
&=& \cN_\omega C_2(\rho_\omega)  (\mathcal{K}_{\RR})_{ij} \nn \\
&=& \cN_\omega \frac{(\omega, \omega+2 \mathsf{w}) \dim_\bbC \rho_\omega}{\dim_\bbC \mathfrak{g}}(\mathcal{K}_{\RR})_{ij}.
\label{eq:etacart}
\eea
In \eqref{eq:etacart} the first equality is immediate; the second reflects the fact that the resulting trace gives a non-degenerate Ad-invariant bilinear form on $\mathfrak{h}_\RR^\star$, hence proportional to the Killing pairing since $\mathfrak{g}$ is simple; and the final step identifies the proportionality factor with the corresponding quadratic Casimir eigenvalue for the representation $\rho_\omega$ in the appropriate normalisation (see e.g. \cite[Lect.~25.1]{MR1153249}). Picking $\cN_\omega\coloneqq -\frac{\dim_\bbC \mathfrak{g}}{(\omega, \omega+2 \mathsf{w}) \dim_\bbC \rho_\omega}$ concludes the identification of \eqref{eq:todag} with the intersection pairing in \cref{Thm:DZ}.
\item[DZ-IV]
The final missing piece to invoke the reconstruction theorem, \cref{Thm:DZ}, is to prove the quasi-polynomiality of the prepotential of $\cM_\omega^{\rm LG}$. Consider the change-of-variables\footnote{This might look like an unnecessary piece of trivial additional notation at this stage: it just amounts to replacing $w_0$-derivatives of the superpotential  with logarithmic $w_0$-derivatives. The reader will forgive us as this will help making the discussion of polynomiality completely manifest at the end of the argument: indeed, while components $c(\de_{w_i},\de_{w_j},\de_{w_k})$ of the $c$-tensor in the $w$-chart are Laurent monomials (with coefficients in $\bbQ[w_1, \dots, w_{l_\RR}]$) in $w_0$ 
with exponent ${2-\delta_{i0}-\delta_{j0}-\delta_{k0}}$, and in particular have negative power for $i=j=k=0$, the components $c(\de_{v_i},\de_{v_j},\de_{v_k})$ of the $c$-tensor in the $v$-chart will always be elements of $\re^{2 v_0} \bbQ[v_1, \dots, v_{l_\RR}]$ (see \eqref{eq:c0w}).}
\beq
v_i \coloneqq  \begin{cases}
                                  \log w_0 & \text{if } i=0,\\
                                   w_i & \text{otherwise,}
  \end{cases}
  \label{Eq:wtildeDef}
\eeq
and define
\beq 
\Xi_{ijk}(p) \coloneqq -\cN_\omega \frac{\delta^{(\mu)}_{\de_{v_i}}\lambda~\delta^{(\mu)}_{\de_{v_j}}\lambda~\delta^{(\mu)}_{\de_{v_k}}\lambda}{\mu^2 \de_\mu \lambda} \rd\mu(p).
\eeq 
Then the symmetric (0,3) tensor \eqref{eq:todac} in the $v$-chart reads
\beq 
c(\de_{v_i}, \de_{v_j}, \de_{v_k}) = \sum_{l}\underset{p=p^{\rm cr}_l}{\Res}\Xi_{ijk}(p).
\label{eq:etaresxi}
\eeq 
The same analysis as the one we carried out to verify [{\bf DZ-III]} reveals that $\Xi_{ijk}$ has poles at the critical points (i) and the poles (iii) of $\lambda$. Additionally, it has poles at the critical points (iv)~
of the $\mu$-projection,
due to the singularities of the derivation $\delta^{(\mu)}_{X}$ in \eqref{eq:horder}. As noted in \eqref{eq:upsqcr}, in the case of the intersection pairing, the poles of $\delta^{(\mu)}_{\de_{x_i}}\lambda$ and $\de_\mu \lambda$ in \eqref{eq:ups} were cancelling out between the numerator and denominator in the Puiseux expansion of $\Upsilon_{ij}(p)$ near $q^{\rm cr}_m$. On the other hand, the additional factor containing $\delta^{(\mu)}_{\de_{w_k}}\lambda$ in the numerator of \eqref{eq:todac} may give rise to a pole with non-vanishing residue for $\Xi_{ijk}(p)$ at $q_m^{\rm cr}$. \\

To compute \eqref{eq:etaresxi} we determine individually the residues of $\Xi_{ijk}(p)$ using the known expression of $\cP_\omega(\lambda,\mu)$ from \eqref{eq:shiftfun}.  By \cref{prop:mu}, the $\mu$-coordinate of the poles of $\lambda$ have simple $v$-independent expressions, which are either $\mu=0$, $\infty$, or a root of unity (see \eqref{eq:mupoles}). The calculation of the corresponding residues is straightforward, and we find
\beq 
\underset{p=\infty_r}{\Res}\Xi_{ijk}(p) \in \re^{2v_0} \bbQ[v_1, \dots, v_{l_\RR}]\,, \quad r=1, \dots, \ell(\mathsf{n}_\omega).
\label{eq:xirespol}
\eeq 
%
On the other hand, the $\mu$-coordinates of the critical points $\{p_l^{\rm cr}\}$ of $\lambda$ (resp. $\{q_m^{\rm cr}\}$ of $\mu$) are given by the roots of  $\mathrm{Discr}_\mu(\cP_\omega)(\mu)$ (resp. $\mathrm{Discr}_\lambda(\cP_\omega)(\mu)$). When $\deg_\lambda \cP_\omega>1$, these are both high degree polynomials in $\mu$ with $v$-dependent coefficients, and therefore $\mu(p_l^{\rm cr})$ and $\mu(q_m^{\rm cr})$ are given by complicated (algebraic) hypergeometric functions of $(\re^{v_0}, v_1, \dots, v_{l_\RR+1})$. Since $\Xi_{ijk}$ has in general non-vanishing residues at $p=q_m^{\rm cr}$, turning the contour around in the sum in \eqref{eq:etaresxi} will pick up some intricate hypergeometric contributions from these points, 
making it difficult to provide a manifest proof of the polynomiality of their sum. One exception however is when $i=0$: since $\delta^{(\mu)}_{\de_{v_0}} \lambda = \lambda$, the same count of the order of divergence as in \eqref{eq:upsqcr} shows that
\beq 
\ord_{q_m^{\rm cr}} \Xi_{0jk} = 0 ~\Longrightarrow~ \underset{p=q_m^{\rm cr}}{\Res} \Xi_{0jk}=0.
\eeq 
In this case, the residue theorem implies that, closing the contour around the complement of $\{p_l^{\rm cr}\}$, \eqref{eq:etaresxi} equates to a sum of residues coming only from the poles of $\lambda$, 
\beq 
c(\de_{v_0}, \de_{v_j}, \de_{v_k}) = - \sum_{r}\underset{p=\infty_r}{\Res}\Xi_{0jk}(p) \in \re^{2v_0} \bbQ[v_1, \dots, v_{l_\RR}],
\label{eq:c0w}
\eeq 
where we used \eqref{eq:xirespol}. In the same vein, we also obtain
\beq 
\eta(\de_{v_j}, \de_{v_k}) = - \sum_{r}\underset{p=\infty_r}{\Res}\frac{1}{\lambda}\Xi_{0jk}(p) \in \re^{v_0}\bbQ[v_1, \dots, v_{l_\RR}].
\label{eq:etaw}
\eeq 
To compute the remaining components $\widetilde{c}_{ijk}(v)\coloneqq c(\de_{v_i}, \de_{v_j}, \de_{v_k})$, we use the associativity of the Frobenius product in the $v$-chart: 
 \begin{equation}
    \sum_{k,l} \bigg(\widetilde{c}_{ijk}\widetilde{\eta}^{kl}\widetilde{c}_{0lm} - \widetilde{c}_{imk}\widetilde{\eta}^{kl}\widetilde{c}_{0lj}\bigg)=0,
    \label{Eq:WDVVu}
\end{equation}
where $\widetilde{\eta}^{kl}\in \re^{-v_0} \bbQ(w_1, \dots, w_{l_\RR})$ is the inverse of the Gram matrix of \eqref{eq:etaw}. The set of equations \eqref{Eq:WDVVu} give an a priori overconstrained inhomogeneous linear system for the unknowns $\widetilde{c}_{ijk}(v)$ with all $i,j,k\neq 0$. It is not at all obvious, in principle, that a unique solution of \eqref{Eq:WDVVu} exists; and even so, that such a solution is a polynomial (as opposed to rational) function of $(\re^{v_0},v_1, \dots, v_{l_\RR})$. That this is the case, however, can be shown from a direct calculation from \eqref{eq:c0w}--\eqref{eq:etaw}, which gives
\beq 
c(\de_{v_i}, \de_{v_j}, \de_{v_k}) \in \re^{2 v_0} \bbQ[v_1, \dots, v_{l_\RR}].
\eeq 
Using \cref{lem:flatpol} together with \eqref{Eq:wtildeDef}, we finally deduce that
\beq 
c(\de_{t_i}, \de_{t_j}, \de_{t_k}) = \frac{\de F_\RR}{\de t_i \de t_j \de t_k}\in \bbQ[t_1, \dots, t_{l_\RR+1},\re^{t_{l_\RR+1}}].
\label{eq:c0t}
\eeq 
\end{description}
The claim now follows upon invoking \cref{Thm:DZ}.

\end{proof} 

\begin{rmk}
The explicit form of \eqref{eq:shiftfun} was crucially used in the computation of \eqref{eq:c0w}--\eqref{eq:etaw}, as well as in determining the flat coordinates in \cref{lem:flatpol}. Despite the disadvantage of having to carry out a separate analysis of each of the seven Dynkin series, an immediate spin-off of this method is that closed-form polynomial expressions for the flat coordinates and, from \eqref{eq:c0t}, the prepotential are rather powerfully produced from straightforward residue computations. We will illustrate this in detail in the next Section. 
\end{rmk}

\subsection{Examples}
\label{sec:examples}
The construction of closed-form flat frames for $\eta$ and prepotentials in all Dynkin types does not follow directly from the original Dubrovin--Zhang construction of the flat pencil\footnote{A priori one could try to construct a quasi-homogeneous polynomial ansatz for the prepotential, impose the associativity of the Frobenius product, and solve for the coefficients. This typically results in a large system of non-linear equations, owing to the non-linearity of the WDVV equations, whose solution is already unviable e.g. for $\RR=E_6$. The proof of \cref{thm:mirror} shows that the Landau--Ginzburg formulas reduce this to an explicit calculation of residues and a relatively small-rank, {\it linear} problem.}. One of the advantages of the mirror formulation in \cref{thm:mirror} is that these can now be easily computed using the Landau--Ginzburg formalism.

\subsubsection{Classical root systems}

\begin{example}[$\RR=B_3$]

From \eqref{Eq:BlSuppot} we have the superpotential
\bea
 \lambda_{B_3} &=&    \dfrac{w_0}{\mu^2 (\mu+1)^2} \bigg(\mu^6+1 -(\mu^5 + \mu) (w_1-1)+(\mu^4 + \mu^2) (-w_1+w_2+1)\nn \\ &+& \mu^3 (-w_3^2+2 w_2+2)\bigg), 
\eea 

which gives flat coordinates
\begin{equation}
    t_4  = \dfrac{\log w_0}{2}, \quad  t_1 = w_0^{\frac{1}{2}} (w_1 + 1), \quad t_3 =  w_0^{\frac{1}{2}} w_3, \quad t_2=  w_0(w_1 + w_2 + 2).
\end{equation}

We obtain the prepotential for $B_3$ to be
\begin{equation}
    F_{B_3} = t_4 t_2^2 +\dfrac{1}{2} t_1^2 t_2 +t_3^2 t_2 -\dfrac{1}{48}t_1^4 - \dfrac{1}{24}t_3^4 +2 t_1 t_3^2  \re^{t_4}+ t_1^2 e^{2 t_4} +2 t_3^2 \re^{2 t_4}  +\dfrac{1}{2} \re^{4 t_4},
\end{equation}
\end{example}
which is seen to be equivalent to the free energy in \cite[Example 2.7]{MR1606165} by $F \mapsto \frac{F}{2}$.

\begin{example}[$\RR=B_4$]

For $B_4$, (\ref{Eq:BlSuppot}) is
\begin{equation}
\begin{aligned}
   \lambda_{B_4} =  \dfrac{w_0}{\mu^3 (\mu+1)^2} \Big(& \mu^8+ 1 -(\mu^7+ \mu) (w_1-1)+(\mu^6 + \mu^2) (-w_1+w_2+1) \\ & -(\mu^5+\mu^3) (w_1-w_2+w_3-1)  +\mu^4 (w_4^2-2 w_1-2 w_3)\Big), 
   \end{aligned}
\end{equation}

which gives flat coordinates
\begin{subequations}
\begin{equation*}
    t_5 = \dfrac{\text{log}\,w_0}{3}, \quad t_1 =  w_0^{\frac{1}{3}} (w_1 + 1), \quad t_2 =  w_0^{\frac{1}{2}} w_4, \quad t_4 =  w_0^{\frac{2}{3}} ( 4 w_1-w_1^2 + 6 w_2 + 11),
\end{equation*}
\begin{equation*}
    t_3 =    w_0 (2 w_1 + w_2 + w_3 + 2).
\end{equation*}
\end{subequations}

In this case, the resulting prepotential is given by
\begin{equation}
\begin{aligned}
    F_{B_4} = \, & \dfrac{1}{1944}t_1^4 t_4 -\dfrac{1}{9720}t_1^6-\dfrac{1}{648}  t_1^2 t_4^2+ +\dfrac{1}{9}  t_1t_4 t_3  -\dfrac{1}{24}t_2^4 + t_2^2 t_3+\dfrac{1}{1944}t_4^3  \\ & +t_5 t_3^2  + \dfrac{1}{3} t_1^2 t_2^2  \re^{t_5}   + \dfrac{1}{3}t_2^2 t_4 \re^{t_5} +\dfrac{1}{36}t_1^4 \re^{2 t_5} +\dfrac{1}{18} t_1^2 t_4  \re^{2 t_5} + 2t_1t_2^2  \re^{2 t_5} \\ & +\dfrac{1}{36}t_4^2 \re^{2 t_5} +2t_2^2 \re^{3 t_5} +  \dfrac{1}{2} t_1^2 \re^{4 t_5}+\dfrac{1}{3} \re^{6 t_5}.
    \end{aligned}
\end{equation}
\end{example}

%
%
%
%
%


\begin{example}[$\RR=C_3$]

Here, the superpotential is given by
\begin{equation}
    \lambda_{C_3} =  \dfrac{ w_0 (\mu^6 + 1- (\mu^5 + \mu) w_1+ (\mu^4+\mu^2) (w_2+1)-\mu^3 (w_1+w_3)}{ \mu^3}, 
\end{equation}

which leads to the following flat coordinates:
\begin{equation}
   t_4 =  \dfrac{\text{log}\, w_0}{3}, \quad t_1 = w_0^{\frac{1}{3}} w_1, \quad t_2 =  w_0^{\frac{2}{3}} (w_1^2 - 6 (w_2 + 1)) , \quad t_3 =  w_0 (w_1 + w_3).
\end{equation}

In this case, the prepotential is 
\begin{equation}
\begin{aligned}
  F_{C_3} = & \,  t_4t_3^2 - \dfrac{1}{9} t_1 t_2 t_3 -\dfrac{1}{9720}t_1^6 -\dfrac{1}{1944}t_1^4 t_2  - \dfrac{1}{648} t_1^2 t_2^2   - \dfrac{1}{1944}t_2^3 \\ & +\dfrac{1}{36} t_1^4 \re^{2 t_4}- \dfrac{1}{18}t_1^2 t_2 \re^{2 t_4} + \dfrac{1}{36} t_2^2 \re^{2 t_4}  + \dfrac{1}{2} t_1^2 \re^{4 t_4} + \dfrac{1}{3} \re^{6 t_4},
  \end{aligned}
\end{equation}
which is the same as the one found in Example 2.8 in \cite{MR1606165} after letting $F \mapsto \frac{F}{2}$, and $t_2 \mapsto -6t_2$.

\end{example}

\begin{example}[$\RR=D_4$]

For $l=4$, \eqref{Eq:DlSuppot} becomes 
\begin{equation}
\lambda_{D_4}     = \dfrac{w_0(\mu^8 - w_1(\mu^7+\mu) + w_2(\mu^6+\mu^2) + (w_1-w_3w_4)(\mu^5+\mu^3)  + 1)}{\mu^2 (\mu^2-1)^2}, 
\end{equation}
which has poles at $\mu = 0, \infty, 1, -1$, all of order 2. The resulting flat coordinates are
\begin{equation}
    t_5 = \dfrac{\text{log}\,w_0}{2}, \quad t_1 = w_0^{\frac{1}{2}}w_1, \quad t_3 = w_0^{\frac{1}{2}}(w_3+w_4), \quad t_4 = w_0^{\frac{1}{2}}(w_3 - w_4), \quad t_2 = w_0(w_2+2), 
\end{equation}

which leads to the prepotential 
\begin{equation}
\begin{aligned}
    F_{D_4} = & \, t_5t_2^2 + \dfrac{1}{4}t_3^2t_2 + \dfrac{1}{4}t_4^2t_2 -\dfrac{1}{48}t_1^4 +\dfrac{1}{2}t_1^2t_2 - \dfrac{1}{384}t_3^4 -\dfrac{1}{64}t_3^2t_4^2- \dfrac{1}{384}t_4^4  \\ & + \dfrac{1}{2}t_1t_3^2\re^{t_5} - \dfrac{1}{2}t_1t_4^2 \re^{t_5} + t_1^2\re^{2t_5}+ \dfrac{1}{2}t_3^2 \re^{2t_5} + \dfrac{1}{2}t_4^2 \re^{2t_5} + \dfrac{1}{2}\re^{4t_5}.
    \end{aligned}
\end{equation}
\end{example}

\begin{example}[$\RR=D_5$]

For $l=5$, (\ref{Eq:DlSuppot}) is given by
\begin{equation}
\begin{aligned}
    \lambda_{D_5} =  \dfrac{w_0}{\mu^3(\mu^2-1)^2} \Big(& \, \mu^{10} -w_1(\mu^9 + \mu) + w_2(\mu^8 + \mu^2) - w_3(\mu^7 + \mu^3) \\ & +(w_2-w_4w_5 + 1)(\mu^6 + \mu^4) - (w_4^2+w_5^2 - 2w_1 - 2w_3) \mu^5 + 1\Big),
    \end{aligned}
\end{equation}
which has poles at $\mu = 0$, $\infty$, $1$, and $-1$, of orders $3$, $3$, $2$, and $2$, respectively. We can compute the flat coordinates
\begin{gather}
    t_6 = \dfrac{\text{log}\,w_0}{3}, \quad t_1 = w_0^{\frac{1}{3}}w_1, \quad t_4 = w_0^{\frac{1}{2}}(w_4-w_5), \quad t_5 = w_0^{\frac{1}{2}}(w_4+w_5), \nonumber \\ t_2 = w_0^{\frac{2}{3}}(w_1^2-6(w_2+2)),  \quad t_3 = w_0(2w_1 + w_3), 
\end{gather}
which give the prepotential 
\bea
F_{D_5} &=& -\frac{t_1^6}{9720}+
\frac{1}{36} \re^{2 t_6} t_1^4-\frac{t_2 t_1^4}{1944}+\frac{1}{2} \re^{4 t_6} t_1^2-\frac{1}{648} t_2^2 t_1^2-\frac{1}{12} \re^{t_6} t_4^2 t_1^2\nn \\ &+& \frac{1}{2} \re^{2 t_6}
   t_4^2 t_1+\frac{1}{2} \re^{2 t_6} t_5^2 t_1-\frac{1}{9} t_2 t_3 t_1+\frac{e^{6 t_6}}{3}-\frac{t_4^4}{384}-\frac{t_5^4}{384}-\frac{t_2^3}{1944}
   \nn \\ 
   &-& \frac{1}{2} \re^{3 t_6} t_4^2+\frac{1}{12} \re^{t_6} t_2 t_4^2+\frac{1}{4} t_3
   t_4^2+\frac{1}{2} \re^{3 t_6} t_5^2-\frac{1}{64} t_4^2 t_5^2-\frac{1}{12} \re^{t_6} t_2 t_5^2 \nn \\
   &+& \frac{1}{12} \re^{t_6} t_5^2 t_1^2-\frac{1}{18} \re^{2 t_6} t_2 t_1^2 +\frac{1}{4} t_3 t_5^2 +\frac{1}{36} \re^{2 t_6} t_2^2+t_3^2 t_6
\eea

\end{example}

\subsubsection{Exceptional root systems}

\begin{example}[$\RR=E_6$]

For the exceptional cases, we find expressions for $\lambda$ near any ramification point by Puiseux expansions. By doing so, we obtain the following flat coordinates for $\RR=E_6$,
\begin{gather}
     t_1 = w_0^{\frac{1}{3}}w_1,  t_2 = w_0^{\frac{2}{3}}(w_1^2 - 6w_2  - 12w_5), ~ t_3 = w_0(2w_1w_5 + w_3 + 3w_6 + 3), \nn \\
    t_4 = w_0^{\frac{2}{3}}(-w_5^2  + 12w_1 + 6w_4), ~ t_5 = w_0^{\frac{1}{3}}w_5, ~ t_6 = w_0^{\frac{1}{2}}(w_6 + 2),~t_7 = \dfrac{\text{log}(w_0)}{6},  
\end{gather}
and the corresponding prepotential is given by
\begin{equation}
\begin{aligned}
    F_{E_6} =  & -\dfrac{1}{19440}t_1^6 +\dfrac{1}{72} \re^{2 t_7} t_5 t_1^4-\dfrac{1}{3888}t_2 t_1^4 +\dfrac{1}{6} \re^{6 t_7} t_1^3+\dfrac{1}{6} \re^{3
t_7} t_6 t_1^3+\dfrac{5}{18} \re^{4 t_7} t_5^2 t_1^2  \\ & +\dfrac{1}{36} \re^{t_7} t_5^2 t_6 t_1^2  +
\dfrac{1}{36} \re^{4 t_7} t_4 t_1^2+\dfrac{1}{36} \re^{t_7} t_6 t_4 t_1^2-\dfrac{1}{36} \re^{2 t_7} t_5 t_2 t_1^2+\dfrac{1}{72} \re^{2 t_7}
t_5^4 t_1    \\ & + \dfrac{1}{72} \re^{2 t_7} t_4^2 t_1  +
\dfrac{1}{2} \re^{8 t_7} t_5 t_1+\re^{5 t_7} t_5 t_6 t_1+\dfrac{1}{36} \re^{2 t_7} t_5^2 t_4 t_1-\dfrac{1}{6} \re^{3 t_7} t_6 t_2 t_1\\ &  - \dfrac{1}{19440}t_5^6 -\dfrac{1}{96}t_6^4 +
\dfrac{1}{6} \re^{6 t_7} t_5^3+\dfrac{1}{3} \re^{3 t_7} t_6^3+\dfrac{1}{3888}t_4^3-\dfrac{1}{3888}t_2^3 +\dfrac{1}{2} \re^{6 t_7} t_6^2\\ &  + \dfrac{1}{72} \re^{2 t_7} t_5 t_2^2  + \dfrac{1}{2} t_7 t_3^2 +\dfrac{1}{6} \re^{3 t_7} t_5^3 t_6  +
\dfrac{1}{3888}t_5^4 t_4 +\dfrac{1}{6} \re^{3 t_7} t_5 t_6 t_4-\dfrac{1}{36} \re^{4 t_7} t_5^2 t_2 \\ & -\dfrac{1}{36} \re^{t_7} t_5^2 t_6 t_2 - \dfrac{1}{36}
\re^{4 t_7} t_4 t_2  - \dfrac{1}{36} \re^{t_7} t_6 t_4 t_2  +\dfrac{1}{4} t_6^2 t_3 +
\dfrac{1}{18} t_5 t_4 t_3 \nn \\ & +\dfrac{1}{2} \re^{2 t_7} t_5 t_6^2 t_1
+\dfrac{1}{12} \re^{12 t_7}-\dfrac{1}{1296} t_2^2 t_1^2-\dfrac{1}{18}
t_2 t_3 t_1 -\dfrac{1}{1296}t_5^2
t_4^2 .
\end{aligned}
\end{equation}

\end{example}

\begin{example}[$\RR=E_7$]

Similarly to the $E_6$-case, by taking Puiseux expansions of the spectral curve, we obtain the following flat coordinates
\begin{gather}
    t_8 = \dfrac{\text{log}(w_0)}{12}, \quad t_6 = w_0^{\frac{1}{4}}w_6, \quad t_1 = w_0^{\frac{1}{3}}(w_1 + 2), \quad t_7 = w_0^{\frac{1}{2}}(2w_6 + w_7),  \nonumber \\ 
     t_5 = w_0^{\frac{1}{2}}(-w_6^2 + 8w_1 + 4w_5 + 12), \quad  t_2 = w_0^{\frac{2}{3}}(-w_1^2 + 26w_1 + 6w_2 + 12w_5 + 26), \nonumber \\ 
    t_4 = w_0^{\frac{3}{4}}(5w_6^3 + 24(6w_1 - w_5 + 21)w_6 + 96w_4 + 288w_7), \nonumber \\
    t_3 = w_0(3w_1^2 + 2w_1(w_5 + 8)+ 3w_6^2 + 3w_6w_7 + 2w_2 + w_3 + 7w_5 + 14). 
\end{gather}

The prepotential for $E_7$ takes the form
\begin{equation}
\begin{aligned}
    F_{E_7} = & -\dfrac{1}{4128768}t_6^8 + \dfrac{1}{18432}49 \re^{6 t_8} t_6^6 +\dfrac{1}{18432}\re^{2 t_8} t_1 t_6^6+\dfrac{1}{294912}t_5 t_6^6 +\dfrac{1}{384}
\re^{3 t_8} t_7 t_6^5 \\ & +\dfrac{19}{192} \re^{12 t_8} t_6^4  +\dfrac{5}{288} \re^{4 t_8} t_1^2 t_6^4-
\dfrac{1}{49152}t_5^2 t_6^4 +\dfrac{1}{8} \re^{8 t_8} t_1 t_6^4+\dfrac{1}{1536}13 \re^{6 t_8} t_5 t_6^4  \\ & +\dfrac{1}{576}
\re^{4 t_8} t_2 t_6^4+\dfrac{1}{4} \re^{9 t_8} t_7 t_6^3  +\dfrac{1}{576} \re^{t_8} t_1^2 t_7 t_6^3+
\dfrac{25}{96} \re^{5 t_8} t_1 t_7 t_6^3+\dfrac{7}{384} \re^{3 t_8} t_7 t_5 t_6^3\\ & +\dfrac{1}{9216}\re^{6 t_8}
t_4 t_6^3 +\dfrac{1}{9216}\re^{2 t_8} t_1 t_4 t_6^3  +\dfrac{1}{147456}t_5 t_4 t_6^3 +
\dfrac{1}{6} \re^{18 t_8} t_6^2+\dfrac{1}{288} \re^{2 t_8} t_1^4 t_6^2 \\ & +\dfrac{1}{24576}t_5^3 t_6^2+\dfrac{5}{8}
\re^{10 t_8} t_1^2 t_6^2+\dfrac{5}{8} \re^{6 t_8} t_7^2 t_6^2 +\dfrac{1}{8} \re^{2 t_8} t_1 t_7^2 t_6^2+
\dfrac{5}{512} \re^{6 t_8} t_5^2 t_6^2 \\ & +\dfrac{1}{288} \re^{2 t_8} t_2^2 t_6^2-\dfrac{1}{589824}t_4^2 t_6^2 +\dfrac{1}{2}
\re^{14 t_8} t_1 t_6^2  +\dfrac{1}{32} \re^{12 t_8} t_5 t_6^2+
\dfrac{1}{24} \re^{4 t_8} t_1^2 t_5 t_6^2\\ & +\dfrac{1}{144} \re^{2 t_8} t_1^2 t_2 t_6^2+\dfrac{1}{24}
\re^{6 t_8} t_1 t_2 t_6^2+\dfrac{1}{96} \re^{4 t_8} t_5 t_2 t_6^2  +
\dfrac{1}{384} \re^{3 t_8} t_7 t_4 t_6^2+\dfrac{1}{3} \re^{3 t_8} t_7^3 t_6 \\ & +\dfrac{1}{6} \re^{3 t_8}
t_1^3 t_7 t_6+\dfrac{7}{6} \re^{7 t_8} t_1^2 t_7 t_6+\re^{11 t_8} t_1 t_7 t_6 +
\dfrac{1}{4} \re^{9 t_8} t_7 t_5 t_6+\dfrac{1}{96} \re^{t_8} t_1^2 t_7 t_5 t_6\\ & +\dfrac{1}{6} \re^{7
t_8} t_7 t_2 t_6+\dfrac{1}{6} \re^{3 t_8} t_1 t_7 t_2 t_6 +
\dfrac{1}{96} \re^{t_8} t_7 t_5 t_2 t_6+\dfrac{1}{576} \re^{4 t_8} t_1^2 t_4 t_6-\dfrac{1}{49152}t_5^2 t_4 t_6 \\ & +\dfrac{1}{1536}\re^{2 t_8} t_1 t_5 t_4 t_6  +\dfrac{1}{576} \re^{4 t_8} t_2 t_4 t_6  +
\dfrac{1}{384} t_4 t_3 t_6+\dfrac{1}{24} \re^{24 t_8}-\dfrac{1}{19440}t_1^6  \\ &  -\dfrac{1}{96}t_7^4 -\dfrac{1}{49152}t_5^4 +\dfrac{1}{6}
\re^{12 t_8} t_1^3+\dfrac{1}{384} \re^{6 t_8} t_5^3+\dfrac{1}{3888}t_2^3+
\dfrac{1}{4} \re^{16 t_8} t_1^2+\dfrac{1}{2} \re^{12 t_8} t_7^2 \\ &  +\re^{8 t_8} t_1 t_7^2+\dfrac{1}{64} \re^{12
t_8} t_5^2  +\dfrac{1}{64} \re^{4 t_8} t_1^2 t_5^2   +\dfrac{1}{72} \re^{8 t_8} t_2^2-\dfrac{1}{1296}t_1^2 t_2^2 +
\dfrac{1}{72} \re^{4 t_8} t_1 t_2^2  \\ &  +\dfrac{1}{18432}\re^{6 t_8} t_4^2   +\dfrac{1}{18432}\re^{2 t_8} t_1 t_4^2 +\dfrac{1}{294912}t_5
t_4^2+\dfrac{1}{2} t_8 t_3^2+\dfrac{1}{72} \re^{4 t_8} t_1^5\\ & +
\dfrac{1}{8} \re^{10 t_8} t_1^2 t_5   +\dfrac{1}{8} \re^{6 t_8} t_7^2 t_5+\dfrac{1}{8} \re^{2 t_8} t_1 t_7^2 t_5 +\dfrac{1}{3888}t_1^4 t_2+\dfrac{1}{36}
\re^{4 t_8} t_1^3 t_2+\dfrac{1}{36} \re^{8 t_8} t_1^2 t_2 \\ &  +
\dfrac{1}{144} \re^{2 t_8} t_1^2 t_5 t_2+\dfrac{1}{24} \re^{6 t_8} t_1 t_5 t_2 +\dfrac{1}{576} \re^{t_8} t_1^2 t_7 t_4+\dfrac{1}{96} \re^{5 t_8}
t_1 t_7 t_4+\dfrac{1}{384} \re^{3 t_8} t_7 t_5 t_4  \\ & +\dfrac{1}{4} t_7^2 t_3+\dfrac{1}{128} t_5^2 t_3  +\dfrac{1}{18} t_1 t_2 t_3 -\dfrac{1}{2949120}t_4 t_6^5+\dfrac{1}{1536}\re^{2 t_8} t_1 t_5 t_6^4+\dfrac{1}{576} \re^{t_8} t_7 t_2 t_6^3 \\
& +\dfrac{5}{24} \re^{6 t_8} t_1^3 t_6^2+
\dfrac{1}{576} \re^{t_8} t_7 t_2 t_4+\dfrac{1}{6} \re^{4 t_8} t_7^2 t_2 +\dfrac{1}{24} \re^{6 t_8} t_1^3 t_5 +\dfrac{1}{288} \re^{2 t_8} t_1^4 t_5 \\
& +\dfrac{1}{288} \re^{2 t_8} t_5 t_2^2+\dfrac{2}{3} \re^{4 t_8} t_1^2 t_7^2 +\dfrac{5}{36} \re^{8 t_8} t_1^4  +\dfrac{1}{1536}\re^{6 t_8} t_5 t_4
t_6 +\dfrac{5}{16} \re^{5 t_8} t_1 t_7 t_5 t_6 \\ &
+\dfrac{1}{64} \re^{3 t_8} t_7 t_5^2 t_6+\dfrac{1}{512} \re^{2 t_8} t_1 t_5^2 t_6^2+\dfrac{1}{8} \re^{8 t_8} t_1 t_5 t_6^2 
\end{aligned}
\end{equation}
\end{example}

\begin{example}[$\RR=F_4$]

In this case we obtain flat coordinates 
\begin{gather}
    t_5 = \dfrac{1}{6}\text{log}(w_0),
  \quad t_4 = w_0^{\frac{1}{3}}(w_4 + 1), \quad   t_3 = w_0^{\frac{2}{3}}(-w_4^2 + 16w_4 +  6w_3 + 6w_1 + 11), \nonumber \\
     t_2 = w_0(2w_4^2 + 6w_4 + w_4w_1 + w_3 + w_2 + 4w_1 + 5), \quad t_1 = w_0^{\frac{1}{2}}(w_4 + w_1 + 2).
\end{gather}

The resulting prepotential is 
\begin{equation}
\begin{aligned}
    F_{F_4} = & -\dfrac{1}{9720}t_4^6 +\dfrac{1}{36} \re^{2 t_5} t_4^5+\dfrac{5}{18} \re^{4 t_5} t_4^4+\dfrac{1}{36} \re^{t_5} t_1 t_4^4+\dfrac{1}{432}
t_3 t_4^4  + \dfrac{1}{3} \re^{6 t_5}  t_4^3+\dfrac{1}{3} \re^{3 t_5} t_1 t_4^3 \\ &  +\dfrac{1}{4} \re^{2 t_5} t_3 t_4^3 +
\dfrac{1}{2} \re^{8 t_5} t_4^2+\dfrac{1}{2} \re^{2 t_5} t_1^2 t_4^2-\dfrac{1}{32} t_3^2 t_4^2+\re^{5 t_5} t_1 t_4^2+\dfrac{1}{4} \re^{4 t_5}
t_3 t_4^2 \\ &  +\dfrac{9}{16} \re^{2 t_5} t_3^2 t_4 +
\dfrac{3}{2} \re^{3 t_5} t_1 t_3 t_4+\dfrac{1}{2} t_3 t_2 t_4+\dfrac{1}{12} \re^{12 t_5}-\dfrac{1}{96} t_1^4 +\dfrac{1}{3} \re^{3 t_5} t_1^3+\dfrac{3
t_3^3}{64} \\ & +\dfrac{9}{16} \re^{4 t_5} t_3^2 +\dfrac{9}{16} \re^{t_5} t_1 t_3^2   +
\dfrac{1}{2} t_5 t_2^2+\dfrac{1}{4} t_1^2 t_2+\dfrac{1}{2} \re^{6 t_5} t_1^2 +\dfrac{1}{4} \re^{t_5} t_1 t_3 t_4^2 . 
\end{aligned}
\end{equation}
\end{example}

\begin{example}[$\RR=G_2$]

For $G_2$, we obtain flat coordinates
\begin{equation}
t_3 = \dfrac{\text{log}(w_0)}{6}, \quad t_1 = w_0^{\frac{1}{2}}(w_1 + 1), \quad t_2 = w_0(2w_1+w_2+2). 
\label{Eq:G2flats}
\end{equation}

In this case, the prepotential takes the form
\begin{equation}
    F_{G_2} =  \dfrac{1}{2}t_2^2 t_3 + \dfrac{1}{4} t_2 t_1^2 - \dfrac{1}{96}t_1^4 + \dfrac{1}{3} t_1^3 \re^{3t_3} + \dfrac{1}{2} t_1^2\re^{6t_3} + \dfrac{1}{12} \re^{12t_3} ,
    \label{Eq:G2F}
\end{equation}
which matches exactly the expression found in \cite[Example 2.4]{MR1606165}. 
\end{example}

\subsection{Non-minimal irreducible representations}
It is argued in \cite{Brini:2017gfi}, based on the isomorphism of Toda flows on Prym--Tyurin varieties associated to different representations, \cite{MR1182413,MR1401779}, that the Frobenius manifold obtained from the construction of \cref{sec:supconstr} is independent of the choice of highest weight $\omega$. \\

Let us verify this explicitly for  $\RR=G_2$. In this case picking $\omega=\omega_2$ gives the second smallest-dimensional non-trivial irreducible representation of $G_2$, which is the 14-dimensional adjoint representation $\rho_{\omega_2}=\mathfrak{g}_2$.  By the same method of the previous section we obtain
\bea
\mathfrak{p}^{[01]_{G_2}}_0 &=& 1, \nn \\
 \mathfrak{p}^{[01]_{G_2}}_1 &=& \chi_2, \nn \\
 \mathfrak{p}^{[01]_{G_2}}_2 &=& \chi_1^3 - \chi_1^2 - \chi_1(2 \chi_2 + 1), \nn \\
 \mathfrak{p}^{[01]_{G_2}}_3 &=& \chi_1^4 - \chi_1^3 - \chi_1^2 (3\chi_2 + 1)  + \chi_1 + 2\chi_2^2 + \chi_2, \nn \\
  \mathfrak{p}^{[01]_{G_2}}_4 &=&  \chi_1^3(\chi_2 - 1) - \chi_1^2(\chi_2 - 1) - \chi_1(2\chi_2^2  - \chi_2 - 1) - \chi_2^2 + \chi_2, \nn \\
\mathfrak{p}^{[01]_{G_2}}_5 &=&  \chi_1^5 - 2\chi_1^4 - 5\chi_1^3\chi_2 + \chi_1^2(3\chi_2 + 2) + \chi_1(6\chi_2^2 + 5\chi_2 - 1) + \chi_2^3 + 2\chi_2^2,  \nn \\
\mathfrak{p}^{[01]_{G_2}}_6 &=&  \chi_1^4 - 3\chi_1^3\chi_2 + \chi_1^2(\chi_2^2 - \chi_2 - 2) + \chi_1\chi_2(4\chi_2 + 3) + 2\chi_2^2 + \chi_2, \nn \\
\mathfrak{p}^{[01]_{G_2}}_7 &=&  + 4\chi_1^4 + 2\chi_1^3(3\chi_2+1) + 2\chi_1^2(\chi_2^2 -2\chi_2 - 3) \nn \\
 &-&  2\chi_1\chi_2(4\chi_2 + 3) - 2\chi_2^3 - 6\chi_2^2 + 2 -2\chi_1^5,
\eea
and $\mathfrak{p}^{[01]_{G_2}}_i = \mathfrak{p}^{[01]_{G_2}}_{14-i}$ by reality of the adjoint representation. The characteristic polynomial \eqref{eq:charpol} then factorises as
\beq
\cP_{[01]_{G_2}}^{\rm red} = (\mu-1)^2 \cP_{G_2,\rm short} \cP_{G_2,\rm long},
\eeq
with the three factors corresponding to the three irreducible Weyl orbits of the adjoint representation associated to the zero, short, and long roots of $G_2$:
\bea
\cP_{G_2,\rm short} &=& \cP^{\rm red}_{[10]_{G_2}}, \nn \\
 \cP_{G_2,\rm long} &=& \mu ^6+\mu ^5 \left(w_1-w_2+1\right)+\mu ^4 \left(w_1^3-3 w_2 w_1-w_1-2 w_2+1\right) \nn \\ &+& \mu ^3 \left(2 w_1^3-w_1^2-4 w_2 w_1-2 w_1-w_2^2-4 w_2+1\right) \nn \\ &+& \mu ^2
   \left(w_1^3-3 w_2 w_1-w_1-2 w_2+1\right)+\mu  \left(w_1-w_2+1\right)+1.
\eea
For any $w$, the curve $\overline{C}_w^{(1)}=\overline{\{\cP_{G_2,\rm long}=0\}}$ is a copy of $\bbP^1$, and the $\lambda$-projection has ramification profile 
\begin{equation}
 \big( \, \overbrace{1,2}^{\mu = 0}, \, \overbrace{1,2}^{\mu = \infty}  \, \big). 
 \label{Eq:G2adjram}
\end{equation}
The embedding $\iota_{[01]_{G_2}}: M^{\rm LG}_{[01]_{G_2}} \hookrightarrow H_{0,(0,1,0,1)}$ gives the same flat coordinates (\ref{Eq:G2flats}) as for the case $\omega=[10]_{G_2}$, and up to scaling the prepotential coincides with the prepotential (\ref{Eq:G2F}), as expected. 

\subsection{Non-canonical Dynkin marking}

In \cite{Brini:2017gfi}, it was proposed that the family of Frobenius algebras obtained in \eqref{eq:shiftfun} through the shift of $w_j \to w_j +\delta_{ij} \frac{\lambda}{w_0}$ for any $i$ should give the Frobenius structure corresponding to Dynkin node $\alpha_i$. \cref{thm:mirror} shows that this is the case for $i=\bar k$, and this proposal is consistent with the analysis of the generalised type-A mirrors of \cite{MR1606165}. However we now show that the conjecture is false in this form at the stated level of generality. Considering the case $\RR=G_2$, we see that shifting $w_1$ instead of $w_2$ in \eqref{eq:g2char} yields 

\begin{equation}
\begin{aligned}
   \mathcal{P}_{[10]_{G_2}, i=1} = & \,  \left(\dfrac{\lambda}{w_0}\right)^2\mu^3 + \dfrac{\lambda}{w_0} (-\mu^5 - 2w_1\mu^3 - \mu - 2) + \mu^6 + (\mu^5 + \mu)(1-w_1)  \\ & + (\mu^4+\mu^2)(w_2+1)  - \mu^3(w_1^2 - 2w_2 - 1)  +1.
\end{aligned}  
   \label{Eq:G2u1curve}
\end{equation}


By computing the metric $   \eta = \sum_{ij}\tilde \eta_{ij}\rd w^i \rd w^j$ we get
\begin{equation}
\tilde\eta_{ij} =  \begin{pmatrix} 
 \dfrac{8w_1+1}{4w_0} & 1  & 0 \\
1 & 0 & 0\\
 0 & 0 & 0
\end{pmatrix},
\end{equation}
which is clearly a singular matrix. Hence, (\ref{Eq:G2u1curve}) cannot define a Frobenius manifold structure, and the conjecture fails in this generality.

\begin{rmk} 
Opening for allowed changes in the primary differential, i.e. scalings and translations of $\rd\log\mu$ which are type~I (Legendre) transformations \cite{Dubrovin:1994hc}, under which the intersection form is invariant, does not resolve the nondegeneracy, and neither does attempting to change the pole structure manually akin to that of (\ref{Def:DSZZSuppot}). Thus, the problem of finding Frobenius manifolds associated to non-canonical nodes for exceptional groups is still open, and indeed even the existence of non-canonical Frobenius manifold structures for DZ-manifolds of exceptional types is, at present, unknown. 
\end{rmk}

\section{Applications}
\label{Section:LL}
\subsection{Topological degrees of Lyashko--Looijenga maps}

The semi-simple locus of a generically semi-simple, $n$-dimensional Frobenius manifold $\cM$ is topologically a covering of finite multiplicity over a quotient by $S_{n}$ of the complement of their discriminant, with covering map 
\bea
\mathrm{LL} : \cM & \longrightarrow & (\bbC^{n} \setminus \mathrm{Discr}_\cM)/S_{n} \nn \\
t & \longrightarrow & \{e_1(u(t)), \dots, e_{n}(u(t))\}
\label{eq:LL}
\eea
assigning to $t \in \cM$ the unorderet set of its canonical coordinates in the form of their elementary symmetric polynomials $e_i(u_1, \dots, u_n)$. When $\cM$ has a Landau--Ginzburg description as a holomorphic family of meromorphic functions, the application $\mathrm{LL}$ is the classical Lyashko--Looijenga (LL) mapping, sending a meromorphic function to the unordered set of its critical values. 

As anticipated in \cref{sec:ll_0}, a direct corollary of \cref{thm:mirror} is the computation of the degree of the LL map of $\cM_{\omega}^{\rm LG} \simeq \cM^{\rm DZ}_{\RR}$. The calculation of the LL-degree can be tackled combinatorially through the enumeration of certain embedded graphs \cite[Sec.~1.3.2]{MR1918855}, which unfortunately proves to be intractable for a general stratum of a Hurwitz space of arbitrary genus and ramification profile. In the case of $\deg \mathrm{LL}(\cM_{\omega}^{\rm LG})$, however, its realisation as a conformal Frobenius manifold allows to bypass the problem altogether by the use of the quasi-homogeneous B\'ezout theorem. To this aim, note that
\beq
 \det(z-(E(t) \cdot))  = \prod_{i=1}^{l_\RR+1}(z-u_i(t)) 
=\sum_{j=0}^{l_\RR+1} (-1)^j e_j(u_1(t), \dots, u_{l_\RR+1}(t)) z^{l_\RR+1-j}.
\label{eq:LLpol}
\eeq
Setting $T \coloneqq \re^{t_{l_\RR+1}}$, it follows from \cref{Thm:DZ,thm:mirror} that the LL-map \eqref{eq:LL} is polynomial in $(t_1, \dots, t_{l_{\RR}+1}, T)$ since both the product and the Euler vector field are in \eqref{eq:LLpol}. Moreover, it follows from the definition \eqref{eq:prodss} of the quantum product in canonical coordinates that the canonical idempotents $\de_{u_i}$ have weight $-1$ under the Euler scaling, meaning that $e_i(u)$ is quasi-homogeneous of degree $i$. We can then avail ourselves of the graded generalisation of B\'ezout's theorem (see e.g. \cite[Thm~3.3]{MR1918855}).
\begin{thm}
Let $F:\bbA^n \to \bbA^n$ be a finite morphism induced by a quasi-homogeneous map $F_*:\bbC[y_1, \dots y_n] \to \bbC[x_1, \dots, x_n]$ with degrees $p_i$ (resp. $q_i$) for $y_i$ (resp. $x_i$) for $i=1,\dots,n$. Then
\beq
\deg F = \prod_{i=1}^n \frac{p_i}{q_i}.
\eeq
\label{thm:bezout}
\end{thm}
An immediate consequence of \cref{Thm:DZ,thm:mirror,thm:bezout} is the following
\begin{cor}
The degree of the LL-map of the Hurwitz stratum $M_\omega^{\rm LG}$ is 
\begin{equation}
    \dfrac{(l_\RR+1)! (\omega_{\bar k}, \omega_{\bar k})^{l_\RR+1}}{\prod_{j=1}^{l_\RR} (\omega_j, \omega_{\bar k})}.
\label{eq:LLdeg}
\end{equation}
\label{cor:LLdeg}
\end{cor}

We collect in \cref{Tab:topdataclass} the calculation of the degrees for the minimal choices of weight $\omega$. Our expectation that $\cM_{\omega}^{\rm LG} \simeq \cM_\RR^{\rm DZ}$ for any of the infinitely many choices of dominant weight $\omega \in \Lambda^+_w(\RR)$ implies that the same formula \eqref{eq:LLdeg} would hold for the Hurwitz strata associated to those non-minimal choices by the construction of \cref{sec:supconstr}. For type $A_l$, \cref{cor:LLdeg} recovers Arnold's formula for the LL-degree of the space of complex trigonometric polynomials, as was already shown in \cite{MR1606165}. The fifth column indicates how $\iota_\omega(M_\RR^{\rm DZ})$ sits as a stratum inside the parent Hurwitz space $H_{g_\omega,\mathsf{n}_\omega}$, either through explicit character relations in $\mathrm{Rep}(\cG_\RR)$ or as a fixed locus of an automorphism of the Hurwitz space induced by the folding of the Dynkin diagram.

\begin{table}[t]
\begin{center}

{
\begin{tabular}{ |c|c|c|c|c|c|c|  }

\hline
$\mathbf{\RR}$   &  $g_\omega$  & $\mathsf{n}_\omega$ & $d_{g_\omega; \mathsf{n}_\omega}$ & $\iota_\omega\big(\cM_\RR^{\rm DZ}\big)$ & $\deg({\rm LL})$ \\
\hline
\hline
$A_l$ &  0  & $(\bar k-1, l- \bar k)$ & $l+1$ & $\HH^{[\mu]}_{g_\omega,\mathsf{n}_\omega}$ & $ \frac{(l+1-\bar k)^{l+1-\bar k} (\bar k)^{\bar k} l!}{(\bar k-1)! (l-\bar k)!}$\\
\hline
$B_l$ &  0  &  $(l-2, l-2, 1)$ & $2l+1$ & $\big(\HH^{[\mu]}_{g_\omega,\mathsf{n}_\omega}\big)^{\mu_2}$ & $ 2(l+1)l(l-1)^{l}$\\
\hline
$C_l$ &  0 & $(l-1, l-1)$  & $2l$ & $\big(\HH^{[\mu]}_{g_\omega,\mathsf{n}_\omega}\big)^{\mu_2}$ & $(l+1)l^{l+1}$ \\
\hline
$D_l$ &  0 & $(l-3, l-3, 1, 1)$  & $2l+2$ & $\big(\HH^{[\mu]}_{g_\omega,\mathsf{n}_\omega}\big)^{\mu_2}$   & $4l(l^2-1)(l-2)^{l-1}$ \\
\hline
$E_6$ & 5 & $(5, 5, 2, 2, 2,2,2)$  & 42 & \eqref{eq:e6char} & $2^4 \cdot 3^7 \cdot 5 \cdot 7$ \\  
\hline
$E_7$ &  33  & \makecell{$(11, 5,3,11, 5,$ \\ $3, 1, 1, 3, 3)$}  &  130 & \eqref{eq:e7char} & $2^{14} \cdot 3^4 \cdot 5 \cdot 7$ \\
\hline
$E_8$ &  128  &
\makecell{
 $(29, 29, 14, 14,14,$ \\ $14, 14,
14, 9,9,$ \\ $9,9, 5,5, 4, $ \\ $4,4, 4,4,4,$ \\ $2,2,0,0)$}  & 518 & \makecell{\cite{Brini:2017gfi,Brini:2019agj}} &  $2^5\cdot 3^6 \cdot 5^6 \cdot 7$\\
\hline
$F_4$ &  4 & $(5, 5, 2, 2,2,2)$  & 36 & \makecell{\eqref{eq:f4char};\\ $\big(M_{[100000]_{E_6}}^{\rm LG}\big)^{\mu_2}$} &  $2^4 \cdot 3^4 \cdot 5$\\
\hline
$G_2$ & 0 & $(1, 1, 1)$  & 7 & \makecell{\eqref{eq:g2char};\\ $\big(M_{[1000]_{D_4}}^{\rm LG}\big)^{S_3}$} &  $2^3 \cdot 3^2$ \\
\hline
\end{tabular}
}
\end{center}
\caption{Lyashko--Looijenga degrees for all Dynkin types.}
\label{Tab:topdataclass}
\end{table}

\subsection{The type-$\RR$ extended Toda hierarchy}

Another notable consequence of the determination of the prepotential and the superpotential of $\cM_\RR^{\rm DZ}$ in \cref{thm:mirror} is the construction of a dispersionless bihamiltonian hierarchy of integrable PDEs on the loop space $\LL\cM_\RR^{\rm DZ}$. This is amenable to a presentation both in normal and dispersionless Lax forms, and generalises the dispersionless limit of the extended bi-graded Toda hierarchy of \cite{MR2246697,MR1606165} (corresponding to $\RR=A_l$) to general Dynkin types.\\

We first recall the general theory of principal hierarchies associated to Frobenius manifolds as formulated by Dubrovin \cite{Dubrovin:1992dz}. Let $\cM$ be an $n$-dimensional semi-simple complex Frobenius manifold, $\{t_\a : \cM \to \bbC\}_{\a=1}^n$ a coordinate chart for $\cM$ which is flat for the metric $\eta$, and $\LL\cM=\{\mathfrak{u}:S^1 \to \cM\}$ the formal loop space of $\cM$ -- an element $\mathfrak{u} \in \LL\cM$ being an $n$-tuple $\mathfrak{u}= (\mathfrak{u}^1, \dots, \mathfrak{u}^n)$ with $\mathfrak{u}_i \in \bbC[[X,X^{-1}]]$ a formal Laurent series in a periodic coordinate $X \in S^1$. Let $(\eta^{\a\b})$ and $(\gamma^{\a\b})$, $\a,\b=1,\dots, n$, denote the Gram matrices of the cotangent metrics $\eta^*$ and $\gamma^*$ on $\cM$. The loop space $\LL\cM$ can be endowed with a bihamiltonian pencil of hydrodynamic Poisson brackets
\beq
\{\mathfrak{u}^\a(X), \mathfrak{u}^\b(Y) \}_{[\lambda]} = (\gamma^{\a\b}+\lambda
\eta^{\a\b}) \delta'(X-Y) + \sum_{\d} \Gamma^{\a\b}_{\delta}(\mathfrak{u}) \partial_X \mathfrak{u}^\delta~ \delta(X-Y),
\label{eq:PBlambda}
\eeq
where $\Gamma^{\a\b}_{\delta}$ denotes the Christoffel symbol of the Levi-Civita connection of $\gamma^*$ in the flat coordinate chart $t$ for $\eta^*$. Let $\nabla$ denote the affine, torsion free connection on $T(\cM \times \bbC^\star)$ defined by
\bea
\nabla_V W &=&  i_V \rd_\cM W + z V \cdot W, \nn \\
\nabla_{\de_z} W &=&  i_{\de_z} \rd_{\bbC^\star} W -  E \cdot W-z^{-1} \cV W,
\label{eq:defconn}
\eea
where $z \in \bbC^\star$, $V(z)$, $W(z) \in \Gamma(T\cM)$, and $\cV\in \Gamma(\mathrm{End}(TM))$ is the grading operator defined in flat coordinates as $\cV^\a_\b=(1-n/2)\delta^\a_\b +\de_{\b}E^\a$. The Frobenius manifold axioms imply that $\nabla$ is flat \cite[Lecture~6]{Dubrovin:1994hc}, and a basis of horizontal sections $\sum_\a W_\b^\a \de_\a$ can be taken to have the form $W_\b^\a =\sum_\nu \eta^{\a\nu} \de_\b h_\nu(\mathfrak{u}, z) z^\cV z^R$ for some constant matrix $R$ (determined by the monodromy data of the Frobenius manifold; see \cite[Lecture~2]{Dubrovin:1998fe}) and $h(\mathfrak{u}, z) \in \Gamma(\cO_\cM)[[z]]$. Furthermore, the Taylor coefficients $h_{\a,m}(\mathfrak{u}) :=[z^{m+1}] h_{\a}$ define Hamiltonian densities for which the corresponding local Hamiltonians,
\beq
H_{\a,m}[\mathfrak{u}] \coloneqq \int_{S^1} h_{\a,m}((\mathfrak{u}(X)) \rd X,
\eeq
are in involution with respect to \eqref{eq:PBlambda} for all $\lambda$,
\beq
\{H_{\a,m},H_{\b,n}\}_{[\lambda]} = 0.
\eeq
The corresponding involutive Hamiltonian flows 
\beq
\de_{T_{\a,m}} \mathfrak{u}^\b:= \{\mathfrak{u}^\b, H_{\a,m}\}_{[\lambda]}= \sum_{\delta\epsilon} \bigg[(\gamma^{\b\delta}+\lambda \eta^{\b\delta}) \de^2_{\mathfrak{u}^\delta \mathfrak{u}^\epsilon} h_{\a,m}(\mathfrak{u}) \mathfrak{u}^\epsilon_X +  \Gamma^{\b\delta}_\epsilon \de_{\mathfrak{u}^\delta} h_{\a,m}(\mathfrak{u}) \partial_X \mathfrak{u}^\epsilon\bigg]
\label{eq:normcoord}
\eeq
for $\a=1, \dots, n$ and $m=0, \dots, \infty $ define an integrable hierarchy of quasi-linear PDEs on $\LL\cM$, called the \emph{principal hierarchy} of $\cM$; the dependent variables  $ \mathfrak{u}^\a= \mathfrak{u}^\a(X,T)$ are called the normal coordinates of the hierarchy. This hierarchy moreover satisfies the $\tau$-symmetry condition
\beq
\de_{T_{\mu,m}}  h_{\nu,n}((\mathfrak{u}(X, T)) = \de_{T_{\nu,n}}  h_{\mu,m}(\mathfrak{u}(X,T)) = \frac{\de^3 \log \tau(X,T)}{\de X \de {T_{\mu,m}} \de {T_{\nu,n}} },
\eeq
for some function $\tau(X; (T_{\mu,m})_{\mu,m})$. In particular, $\mathfrak{u}^\a(X,T) = \de^2_{X,T_{\a,0}}\log \tau$.\\

\subsubsection{The dispersionless extended $\RR$-type Toda hierarchy: Hamiltonian and Lax--Sato form}
In the case of the Dubrovin--Zhang Frobenius manifolds of type~$\RR$, \cref{thm:mirror} can be used effectively to solve the problem of finding the fundamental solutions of \eqref{eq:defconn} and obtain the presentation of the principal hierarchy in normal coordinates \eqref{eq:normcoord}. An immediate adaptation of \cite[Proposition~6.3]{Dubrovin:1994hc} gives the following
\begin{prop}
With conventions as in \cref{Thm:5.1}, let $\widetilde  h_{i,\a}$, $\widetilde  h^{\rm ext}_j$, $\widetilde  h^{\rm res}_k$ be the flat coordinates for the deformed connection \eqref{eq:defconn} on $\cM_\RR^{\rm DZ} \times \bbC^\star$ normalised such that $\widetilde h_{i,\a}=\tau_{i,\a}+\cO(z)$, $\widetilde h^{\rm ext}_{j}=\tau^{\rm ext}_j+\cO(z)$,  $\widetilde h^{\rm res}_{k}=\tau^{\rm res}_k+\cO(z)$. Then,
\bea
\widetilde h_{i,\a}(\tau,z) &=& -\frac{n_i+1}{\a} \underset{\infty_i}{\Res}~\kappa_i^\alpha {}_1F_1\big(1,1+\frac{\a}{n_i+1}; z \lambda(\mu)\big) \frac{\rd\mu}{\mu}, \label{eq:RToda1} \\
\widetilde h^{\rm ext}_{j}(\tau,z) &=& \mathrm{p.v.} \int_{\infty_0}^{\infty_j} \re^{z \lambda} \frac{\rd \mu}{\mu}, \label{eq:RToda2} \\
\widetilde h^{\rm res}_{k}(\tau,z) &=& \underset{\infty_i}{\Res} \frac{\re^{z \lambda} -1}{z} \frac{\rd \mu}{\mu}, 
\label{eq:RToda3}
\eea
where $ {}_1F_1(a,b;x)=\sum_{n=0}^\infty \frac{(a)_n x^n}{(b)_n n!}$ is Kummer's confluent hypergeometric function, and $(a)_n \coloneqq \Gamma(a+n)/\Gamma(a)$.
\label{prop:mirrorham}
\end{prop}

We call the bihamiltonian integrable hierarchy defined by \eqref{eq:normcoord} and \eqref{eq:RToda1}--\eqref{eq:RToda3} the \emph{dispersionless extended $\RR$-type Toda hierarchy}. For $\RR=A_n$, this coincides with the dispersionless limit of the bi-graded Toda hierarchy of \cite{MR2246697}. The adjective ``extended'' refers to the Hamiltonian flows generated by $H^{\rm ext}_{j}[\mathfrak{u}]$, which are higher order versions of the space translation: when $\RR$ is simply-laced these encode the non-stationary Gromov--Witten invariants of the type-$\RR$ $\bbP^1$ orbifold given by insertions of descendents of the identity. We refer to the Hamiltonian flows generated by $H_{i,\a}[\mathfrak{u}]$ and $H^{\rm res}_{k}[\mathfrak{u}]$ as the {\it stationary flows} of the hierarchy.\\

\cref{thm:mirror} also provides a dispersionless Lax--Sato description of  \eqref{eq:normcoord} and \eqref{eq:RToda1}--\eqref{eq:RToda3} as a specific reduction of the universal genus-$g_\omega$ Whitham hierarchy with $\ell(\mathsf{n}_\omega)$ punctures \cite{Krichever:1992qe,Dubrovin:1992dz}. Let $\widetilde C^{(\omega)}_w$ be the universal covering of $\overline{C_w^{\omega}} \setminus \{\infty_i\}_i$,  the fibre at $w$ of the Landau--Ginzburg family of \cref{thm:mirror}, viewed as an analytic variety.  
Following \cite{Dubrovin:1994hc}, we consider second and third kind differentials  $\Omega_{i,\a}$, $\Omega^{\rm ext}_j$, $\Omega^{\rm res}_k$ defined on $\widetilde C^{(\omega)}_w$ such that
\bea
\Omega_{i,\a;m} &=& -\frac{1}{n_i+1}\left[\left(\frac{\a}{n_i+1} \right)_{m+1} \right]^{-1} \rd \lambda^{\a/(n_i+1)+m} + \mathrm{regular},\nn \\
\Omega^{\rm ext}_{j;m} &=& 
\left\{
\bary{cc}
-\frac{\rd\psi_m(\lambda)}{n_j+1} + \mathrm{regular}~{\rm near}~\infty_i, \\
\frac{\rd\psi_m(\lambda)}{n_0+1} + \mathrm{regular}~{\rm near}~\infty_0,
\eary
\right.
\nn \\
\Omega^{\rm res}_{i;m} &=& -\rd \left(\frac{\lambda^{m+1}}{(m+1)!} \right) + \mathrm{regular},
\eea
where $\psi_m(\lambda)\coloneqq \lambda^m/m!(\log \lambda - H_m)$, and $H_m$ is the $m^{\rm th}$ harmonic number. Then, \eqref{eq:normcoord} and \eqref{eq:RToda1}--\eqref{eq:RToda3} are equivalent to the dispersionless Lax system
\beq
\de_{T_{(i,\a);m}} \lambda = \{\lambda, q_{i,\a;m}\}_{\rm LS}
, \quad \de_{T^{\rm ext}_{i;m}} \lambda = \{\lambda, q^{\rm ext}_{j;m}\}_{\rm LS}, \quad 
\de_{T^{\rm res}_{i;m}} \lambda = \{\lambda, q^{\rm res}_{i;m}\}_{\rm LS},
\eeq
where $q_{i,\a;m}(\mu) \coloneqq \int^\mu \Omega_{i,\a;m}$,  $q^{\rm ext}_{j;m}(\mu) \coloneqq \int^\mu \Omega^{\rm ext}_{j;m}$ and  $q^{\rm res}_{j;m}(\mu) \coloneqq \int^\mu \Omega^{\rm res}_{j;m}$, and 
\beq
\{f(\mu,X), g(\mu, X)\}_{\rm LS} \coloneqq \mu \l( \de_\mu f \de_X g- \de_X f \de_\mu g\r).
\eeq
Having a closed-form superpotential for $\cM_\RR^{\rm DZ}$ from \cref{thm:mirror} in particular provides explicit expressions for the Lax--Sato and Hamiltonian densities.

\begin{rmk}
In \cite{MR2440696}, a deformation scheme for the genus zero universal Whitham hierarchy is introduced in terms of a Moyal-type quantisation of the dispersionless Lax--Sato formalism. It would be intriguing to apply these ideas to the cases when $M^{\rm LG}_{\omega}$ embeds into a genus $g_\omega=0$ Hurwitz space, and verify that the resulting dispersionful deformation of the Principal Hierarchy is compatible with the hierarchy obtained by the quantisation of the underlying semi-simple CohFT.
\end{rmk}

\begin{example}[$\RR=G_2$]
Let us consider for example $\RR=G_2$. We construct explicitly the whole tower of Hamiltonian densities for the stationary flows of the type-$G_2$ dispersionless Toda hierarchy. In principle, these can be computed (up to a triangular linear transformation in the flow variables $T^{\a,m}$) by imposing the recursion relation, coming from the first line of \eqref{eq:defconn},
\beq
\de^2_{t_\a t_\b} h_{\gamma, m} = \sum_\delta c_{\a \b}^\delta \de_{t_\delta} h_{\gamma, m-1}
\label{eq:recdefconn}
\eeq
with $\gamma=1,2,3$, $m \geq 0$,  $h_{\gamma,0}=\sum_\delta \eta_{\gamma,\delta} t_\delta$, and $c_{\a\b}^\gamma$ are the structure constants of the quantum product determined by the prepotential \eqref{Eq:G2F}. While the recursion \eqref{eq:recdefconn} is ostensibly very hard to solve directly, the combination of \cref{thm:mirror,prop:mirrorham} allows to give closed forms for the stationary Hamiltonians $H_{\gamma,m}$, $\gamma=1,3$, parametrically in $m$. This is most easily achieved in flat coordinates $(x_0, x_1,x_2)$ for the second metric $\gamma$, and in Hamiltonian form for the corresponding Poisson bracket. From \eqref{eq:g2spot} we have
\beq
\lambda(\mu)=\frac{w_0}{\mu^2(\mu+1)^2} \prod_{i=1}^6 (\mu-a_i(x)),
\eeq
with
\beq
a_1= \re^{x_1}, ~ a_2=\re^{-x_1+x_2}, ~ a_3=\re^{2 x_1-x_2},  ~ a_4= \re^{-x_1},  ~ a_5=\re^{x_1-x_2}, ~ a_6=\re^{x_2-2x_1}.
\eeq
Labelling the punctures at $\mu=0$, $-1$ and $\infty$ as $\infty_0$, $\infty_1$ and $\infty_2$ respectively we have:
\bea
\widetilde h^{\rm res}_{0; m} = -\widetilde h^{\rm res}_{2; m}  = h_{3,m}, & \qquad & \widetilde h^{\rm res}_{1;m}=0,  \nn \\
\widetilde h_{0,1/2; m} = \widetilde h_{2,1/2; m} = h_{1,m}, & \qquad & \widetilde h_{1,1/2;m}=-2 \widetilde h_{0,1/2; m}.  \nn \\
\eea
From \eqref{eq:RToda1} and \eqref{eq:RToda3} we then get that 
\bea
h_{1,m} &=&  
\sum_{j=0}^{2m+1} \sum_{\substack{k_1, \dots, k_6=0,\dots, 2m+1 \\  \sum_i k_i=2m+1-j}}  \frac{(2 m+1)_j \re^{(m+1/2) x_0}}{j! (\frac{3}{2})_m} \prod_{i=1}^6 \frac{ \left(m-k_i+3/2\right)_{k_i}}{a_i(x)^{k_i-m-1/2}k_i!}, \nn \\
h_{3,m} &=&  
\sum_{j=0}^{2m} \sum_{\substack{k_1, \dots, k_6=0,\dots, m \\  \sum_i k_i=2m-j}}  \frac{(2 m)_j \re^{m x_0}}{j! m!} \prod_{i=1}^6 \frac{a_i(x)^{m-k_i} \left(m-k_i+1\right)_{k_i}}{k_i!}, \nn \\
\eea
where $(a)_m = \Gamma(a+m)/\Gamma(m)$ is the usual notation for the Pochhammer symbol.  In these coordinates, the Gram matrix of the second metric  $\gamma$ and its inverse read
\beq
(\gamma) = 
\left(
\begin{array}{ccc}
 -1 & 0 & 0 \\
 0 & 12 & -6 \\
 0 & -6 & 4 \\
\end{array}
\right), \qquad
(\gamma^{-1})=
\left(
\begin{array}{ccc}
-1 & 0 & 0 \\
 0 & \frac{1}{3} & \frac{1}{2} \\
 0 & \frac{1}{2} & 1 \\
\end{array}
\right).
\eeq
Let $x_i=f_i(t_1,t_2, t_3)$, $i=0,1,2$ be the change-of-variables expressing the $x$-coordinates in the flat coordinate chart $(t_1, t_2, t_3)$ for the first metric $\eta$, and define accordingly $\mathfrak{w}^i = f_i(\mathfrak{u}^1,\mathfrak{u}^2,\mathfrak{u}^3)$ for the corresponding dependent variables for the principal hierarchy. In these coordinates, the second Poisson bracket takes the form (recall that $\gamma^{ij}:=(\gamma^{-1})_{ij}$)
\beq
\{\mathfrak{w}^i(X),\mathfrak{w}^j(Y)\}_{0}= \gamma^{ij} \delta'(X-Y)
\eeq
and the stationary flows are given by
\beq
\frac{\partial \mathfrak{w}^i}{\partial T_{j,m}} = \{\mathfrak{w}^i, H_{j,m}\}_{0} = \sum_{k=0,1,2}\gamma^{ik} \partial_{\mathfrak{w}^k} h_{j,m} (\mathfrak{w}) \de_X\mathfrak{w}^k, \quad j=1,3.
\eeq

\end{example}

\begin{appendix}

\section{Character relations for $E_6$ and $E_7$}
\label{Section:AppenA}

\subsection{$\RR=E_6$} The character relations for $\rho=\rho_{\omega_1}$ and $k=1, \dots, 13$ are given below. The remaining character relations for $k>13$ are obtained from the complex conjugation relation
\beq 
\mathfrak{p}^{[100000]_{E_6}}_{k}(\chi_1, \chi_2, \chi_3, \chi_4, \chi_5,\chi_6)=\mathfrak{p}^{[000010]_{E_6}}_{27-k}(\chi_5, \chi_4, \chi_3, \chi_2, \chi_1,\chi_6)\,.
\eeq 

From \cref{lem:vanish,cor:nomega} we find:
\bea
\mathfrak{p}^{[100000]_{E_6}}_0 &=& 1, \nn \\
\mathfrak{p}^{[100000]_{E_6}}_1 &=& \chi_1, \nn \\
\mathfrak{p}^{[100000]_{E_6}}_2 &=&\chi_2, \nn \\
\mathfrak{p}^{[100000]_{E_6}}_3 &=& \chi_3, \nn \\ 
\mathfrak{p}^{[100000]_{E_6}}_4 &=& \chi_1-\chi_5^2-\chi_2 \chi_5+\chi_4+\chi_4 \chi_6, \nn \\
\mathfrak{p}^{[100000]_{E_6}}_5 &=& \chi_1^2-2 \chi_5^2 \chi_1+2 \chi_4 \chi_1+\chi_4^2+\chi_5 \chi_6^2+\chi_2-2
\chi_3 \chi_5+\chi_5-\chi_2 \chi_6-\chi_5 \chi_6, \nn \\
\mathfrak{p}^{[100000]_{E_6}}_6 &= &   \chi_1 \chi_5-\chi_5^3-\chi_2 \chi_5^2+2 \chi_4 \chi_5-2 \chi_1 \chi_6 \chi_5+\chi_4
\chi_6 \chi_5+\chi_6^3\nn \\ &+& 2 \chi_1 \chi_2-2 \chi_3+\chi_2 \chi_4 -3 \chi_3 \chi_6, \nn \\
\mathfrak{p}^{[100000]_{E_6}}_7 &= &  2 \chi_2^2+\chi_5 \chi_2-2 \chi_5 \chi_6 \chi_2+\chi_3 \chi_5^2+\chi_1 \chi_6^2+\chi_4
\chi_6^2-3 \chi_1 \chi_3  \nn \\ & -& 2 \chi_3 \chi_4   +\chi_4-\chi_5^2 \chi_6-\chi_1 \chi_6+\chi_4
\chi_6, \nn \\
\mathfrak{p}^{[100000]_{E_6}}_8 &= &  \chi_2 \chi_5^3-\chi_1 \chi_6 \chi_5^2+\chi_6^2 \chi_5+\chi_1 \chi_2 \chi_5-2
\chi_3 \chi_5-3 \chi_2 \chi_4 \chi_5+\chi_3 \chi_6 \chi_5  \nn \\ & -& 2 \chi_6 \chi_5   +\chi_5-\chi_1^2-\chi_2
\chi_6^2+\chi_2 \chi_3+\chi_1 \chi_4+\chi_1^2 \chi_6-\chi_2 \chi_6+2 \chi_1 \chi_4
\chi_6, \nn \\
\mathfrak{p}^{[100000]_{E_6}}_9 &= &  \chi_1 \chi_5^4-\chi_6 \chi_5^3+\chi_2 \chi_5^2-4 \chi_1 \chi_4 \chi_5^2+\chi_2
\chi_6 \chi_5^2-\chi_2^2 \chi_5-\chi_1 \chi_6^2 \chi_5 \nn \\ & +& \chi_4 \chi_5+3 \chi_4 \chi_6 \chi_5+\chi_6^3+\chi_3^2+2 \chi_1 \chi_4^2+\chi_1
\chi_2-6 \chi_3+4 \chi_1^2 \chi_4-4 \chi_2 \chi_4 \nn \\ &+& 4 \chi_1 \chi_3
\chi_5  -2 \chi_2 \chi_4 \chi_6  -3 \chi_3 \chi_6-4 \chi_1 \chi_5+3\nn \\
\mathfrak{p}^{[100000]_{E_6}}_{10} &= &   \chi_5^5-5 \chi_4 \chi_5^3+\chi_1 \chi_6 \chi_5^3-\chi_6^2 \chi_5^2-\chi_1
\chi_2 \chi_5^2+5 \chi_3 \chi_5^2-\chi_5^2-2 \chi_1^2 \chi_5  \nn \\ & +& \chi_2 \chi_3 \chi_5+4 \chi_1 \chi_4 \chi_5+\chi_1^2 \chi_6 \chi_5-2 \chi_2
\chi_6 \chi_5-3 \chi_1 \chi_4 \chi_6 \chi_5+\chi_1 \chi_6^2  \nn \\ & +& \chi_1^2
\chi_2-5 \chi_1 \chi_3+2 \chi_1 \chi_2 \chi_4-5 \chi_3 \chi_4-\chi_2^2 \chi_6-2 \chi_1
\chi_6  +\chi_1 \chi_3 \chi_6 \nn \\
&+& \chi_2 \chi_5+5 \chi_4^2 \chi_5+2 \chi_4 \chi_6^2+2 \chi_1-\chi_4 \chi_6, \nn \\
\mathfrak{p}^{[100000]_{E_6}}_{11} &= &  \chi_6 \chi_5^4-\chi_2 \chi_5^3+\chi_1 \chi_3 \chi_5^2-\chi_4 \chi_5^2+2 \chi_1
\chi_6 \chi_5^2-4 \chi_4 \chi_6 \chi_5^2-2 \chi_6^2 \chi_5 \nn \\ & +& 3 \chi_3
\chi_5+3 \chi_2 \chi_4 \chi_5-2 \chi_1 \chi_2 \chi_6 \chi_5+3 \chi_3 \chi_6 \chi_5-\chi_6
\chi_5+\chi_1 \chi_2^2+2 \chi_4^2  \nn \\ & +& 2 \chi_1^2 \chi_6^2-2 \chi_2 \chi_6^2+2 \chi_2-3 \chi_1^2
\chi_3+\chi_2 \chi_3+\chi_2^2 \chi_4+\chi_1 \chi_4   -2 \chi_1 \chi_3 \chi_4  \nn \\
&+&\chi_2 \chi_6+2 \chi_4^2 \chi_6-2 \chi_1^2
\chi_6-\chi_1 \chi_2 \chi_5, \nn \\
\mathfrak{p}^{[100000]_{E_6}}_{12} &= &  2 \chi_6 \chi_1^3+\chi_5^2 \chi_1^2-\chi_4 \chi_1^2-2 \chi_2 \chi_5
\chi_1^2-\chi_5^2 \chi_6 \chi_1^2+\chi_4 \chi_6 \chi_1^2+3 \chi_5 \chi_6^2 \chi_1 \nn \\ & -& 3 \chi_2 \chi_3 \chi_1+\chi_2 \chi_4 \chi_5 \chi_1+\chi_5 \chi_1-5 \chi_2
\chi_6 \chi_1-2 \chi_5 \chi_6 \chi_1+\chi_2^3+\chi_3 \chi_5^3 \nn \\& -& \chi_5^3-2 \chi_6^3+3 \chi_3^2-\chi_3+\chi_2
\chi_4+3 \chi_2^2 \chi_5-3 \chi_3 \chi_4 \chi_5   +2 \chi_4 \chi_5  +\chi_5^3 \chi_6, \nn \\
&+& 2 \chi_2
\chi_1+6 \chi_3 \chi_6 -3 \chi_4 \chi_5 \chi_6-\chi_1^3 -\chi_2
\chi_5^2 \chi_6\nn \\
 \mathfrak{p}^{[100000]_{E_6}}_{13} &= &  \chi_1^4-2 \chi_5^2 \chi_1^3+2 \chi_4 \chi_1^3+\chi_4^2 \chi_1^2-3 \chi_2 \chi_1^2-2
\chi_3 \chi_5 \chi_1^2-\chi_5 \chi_1^2-\chi_2 \chi_6 \chi_1^2 \nn \\ & +& 2
\chi_5^3 \chi_1+2 \chi_2 \chi_5^2 \chi_1-2 \chi_6^2 \chi_1+\chi_3 \chi_1-4 \chi_2 \chi_4
\chi_1+\chi_2^2 \chi_5 \chi_1-4 \chi_4 \chi_5 \chi_1  \nn \\ & +& 3 \chi_3
\chi_6 \chi_1+\chi_4 \chi_5 \chi_6 \chi_1-\chi_6 \chi_1+2 \chi_1+\chi_2^2-2 \chi_2
\chi_4^2+\chi_3 \chi_5^2+\chi_2 \chi_4 \chi_5^2  \nn \\ & +& \chi_5^2 \chi_6^2-3 \chi_3 \chi_4+2 \chi_4+\chi_2
\chi_5-\chi_2 \chi_3 \chi_5+\chi_2^2 \chi_6-2 \chi_5^2 \chi_6-2 \chi_2 \chi_5 \chi_6 \nn \\ 
&+& 4 \chi_5 \chi_6 \chi_1^2 -\chi_5^3 \chi_6 \chi_1
. 
\label{eq:e6char}
\eea

\subsection{$\RR=E_7$}

We include here the character relations for $\rho=\rho_{\omega_6}$ and $k=1, \dots, 11$; note that by reality, we have $\mathfrak{p}^{\omega_6}_{56-k}=\mathfrak{p}^{\omega_6}_{k}$. The full set of character relations for $k$ up to 28 is available upon request.\\

\bea
\mathfrak{p}^{[0000010]_{E_7}}_0 &=& 1, \nn \\
\mathfrak{p}^{[0000010]_{E_7}}_1 &=& \chi_6,\nn \\
\mathfrak{p}^{[0000010]_{E_7}}_2 &=& \chi_5+1,\nn \\
\mathfrak{p}^{[0000010]_{E_7}}_3 &=& \chi_4+\chi_6,\nn \\
\mathfrak{p}^{[0000010]_{E_7}}_4 &=&  \chi_3 + \chi_5+1,\nn \\
\mathfrak{p}^{[0000010]_{E_7}}_5 &=& (1-\chi_1) \chi_4+\left(-\chi_1^2+\chi_1+\chi_2+\chi_5+1\right) \chi_6+\chi_2 \chi_7,\nn \\
 \mathfrak{p}^{[0000010]_{E_7}}_6 &= &  (1-2 \chi_5) \chi_1^2 -2 \chi_1^3+\left(\chi_6^2-\chi_7 \chi_6+\chi_7^2+4 \chi_2-2
\chi_3+2 \chi_5+2\right) \chi_1 \nn \\ & -& \chi_3+2 \chi_5+2 \chi_2 (\chi_5+1)+\chi_4
\chi_6-\chi_4 \chi_7-\chi_6 \chi_7+\chi_5^2+\chi_2^2+1,  \nn \\
 \mathfrak{p}^{[0000010]_{E_7}}_7 &= &  \chi_4 \left(2-\chi_1^2+\chi_1+\chi_2+2 \chi_5\right)+(1-\chi_1^3+(2 \chi_2+\chi_5+3)\chi_6
\chi_1+2 \chi_2 \nn \\ & - & 2 \chi_3+\chi_5 +  \chi_7 \left(-2 \chi_1^2+(\chi_2-2 \chi_5+1)
\chi_1+\chi_7^2+3 \chi_2-3 \chi_3\right), \nn \\
 \mathfrak{p}^{[0000010]_{E_7}}_8 &= & (2 \chi_6^2+\chi_4 \chi_6-2
\chi_7 \chi_6+\chi_7^2+4 \chi_2  -2 \chi_1^3-4 \chi_3+2 \chi_5-2 \chi_4 \chi_7) \chi_1 \nn \\ &+& \chi_2^2+2
\chi_4^2+\chi_6^2+\chi_5 \chi_7^2+\chi_7^2-2 \chi_3  -3 \chi_3 \chi_5+3 \chi_4 \chi_6+\chi_4
\chi_7 \nn \\ & +& \chi_2 (\chi_6^2+\chi_7 \chi_6+\chi_7^2  -2
\chi_3  +2 \chi_5)+\left(\chi_3-2 \chi_5-\chi_6 \chi_7+2\right) \chi_1^2 \nn \\ &-& \chi_5 \chi_6 \chi_7+\chi_6 \chi_7, \nn \\
 \mathfrak{p}^{[0000010]_{E_7}}_9 &= &  (\chi_1+2) \chi_6^3-2 \chi_1 \chi_7 \chi_6^2+(-\chi_1^3+\chi_1^2+\left(\chi_7^2+2
\chi_2-2 \chi_3-2 \chi_5\right) \chi_1 \nn \\ & -& 3 \chi_3-2 \chi_5+\chi_2 (\chi_5+1)-\chi_5^2-1)
\chi_6+((\chi_2+\chi_3+\chi_5 + 2) \chi_1 \nn \\ &-& (\chi_5+2) \chi_1^2+\chi_5^2 -\chi_3+3 \chi_5  +2
\chi_2 (\chi_5+1)+2) \chi_7 \nn \\ &+& \chi_4 (\chi_1^3-\chi_1^2+(-3 \chi_2+\chi_5 + 2) \chi_1 \nn \\ &-&\chi_7^2-2
\chi_2+\chi_3+3 \chi_5-\chi_6 \chi_7+4), \nn \\
 \mathfrak{p}^{[0000010]_{E_7}}_{10} &= &   +\left(-\chi_6^2-2 \chi_7 \chi_6+\chi_7^2-(4
\chi_2+1) \chi_5+\chi_4 (\chi_6+\chi_7)\right) \chi_1^2\nn \\ &-& (\chi_4^2+3 \chi_7 \chi_4-4 \chi_5^2-\chi_2
\chi_6^2-2 \chi_6^2+\chi_7^2-3 \chi_2 \chi_6 \chi_7-3 \chi_6 \chi_7)
\chi_1 \nn \\ &  +& \chi_6 \chi_7^3  + \chi_3^2  +4 \chi_2 \chi_5^2+2 \chi_5^2+\chi_2 \chi_6^2-5
\chi_5 \chi_6^2-5 \chi_6^2-\chi_2 \chi_7^2-\chi_5 \chi_7^2  \nn \\ & +& 5 \chi_5-4 \chi_2 \chi_4 \chi_6+3 \chi_4 \chi_6+\chi_4 \chi_5 \chi_6-2
\chi_2 \chi_4 \chi_7+\chi_4 \chi_7+2 \chi_2 \chi_6 \chi_7 \nn \\ 
&+& 2
\chi_2 \chi_5+\chi_5 \chi_1^4-\chi_6 \chi_7 \chi_1^3+ \left(4 \chi_6^2+\chi_7^2-2\right)\chi_5\chi_1
\chi_1+2 \chi_5 \chi_6 \chi_7+3 \chi_6^4 \
\nn \\ 
& +&  \chi_6
\chi_7+\chi_3 \left(-6 \chi_6^2-3 \chi_7 \chi_6+(4 \chi_1+5) \chi_5+4\right)-\chi_7^2+2 \chi_2^2 \chi_5+3, \nn \\
 \mathfrak{p}^{[0000010]_{E_7}}_{11} &= & \chi _6 \chi _1^5-\chi _7 \chi _1^4+2 \chi _4 \chi _1^3-5 \chi _2 \chi _6 \chi _1^3-5 \chi _6 \chi _1^3+\chi _5 \chi _7 \chi _1^3-\chi _6^3 \chi _1^2\nn \\ 
 &-& \chi _6 \chi _7^2 \chi _1^2+\chi _4 \chi _1^2-\chi _4 \chi
   _5 \chi _1^2+\chi _2 \chi _6 \chi _1^2+5 \chi _3 \chi _6 \chi _1^2-4 \chi _5 \chi _6 \chi _1^2+3 \chi _6 \chi _1^2 \nn \\ &+& 4 \chi _2 \chi _7 \chi _1^2-2 \chi _5 \chi _7 \chi _1^2+3 \chi _7 \chi _1^2+4 \chi _6^3
   \chi _1+\chi _7^3 \chi _1+\chi _4 \chi _6^2 \chi _1-6 \chi _2 \chi _4 \chi _1 \nn \\ 
   &+& \chi _3 \chi _4 \chi _1-3 \chi _4 \chi _1+4 \chi _4 \chi _5 \chi _1+5 \chi _2^2 \chi _6 \chi _1-2 \chi _5^2 \chi _6 \chi _1+8
   \chi _2 \chi _6 \chi _1 \nn \\ 
   &-&8 \chi _3 \chi _6 \chi _1+4 \chi _2 \chi _5 \chi _6 \chi _1-\chi _6 \chi _1+\chi _5^2 \chi _7 \chi _1+\chi _6^2 \chi _7 \chi _1-\chi _2 \chi _7 \chi _1\nn \\ &-& 4 \chi _3 \chi _7 \chi _1-3
   \chi _2 \chi _5 \chi _7 \chi _1+\chi _5 \chi _7 \chi _1-2 \chi _4 \chi _6 \chi _7 \chi _1-3 \chi _7 \chi _1\nn \\ &+& 2 \chi _5 \chi _6^3+\chi _4 \chi _5^2-\chi _4 \chi _6^2-\chi _4 \chi _7^2+2 \chi _2 \chi _6 \chi
   _7^2+\chi _5 \chi _6 \chi _7^2+\chi _6 \chi _7^2-7 \chi _2 \chi _4 \nn \\ &+& 5 \chi _3 \chi _4+3 \chi _4+2 \chi _2 \chi _4 \chi _5+3 \chi _4 \chi _5+3 \chi _2^2 \chi _6-\chi _5^2 \chi _6\nn \\ &+& 2 \chi _2 \chi _6-5 \chi _2
   \chi _3 \chi _6-2 \chi _3 \chi _6+7 \chi _2 \chi _5 \chi _6-5 \chi _3 \chi _5 \chi _6-\chi _5 \chi _6\nn \\ &+&\chi _6-2 \chi _2^2 \chi _7-\chi _4^2 \chi _7+\chi _5^2 \chi _7-\chi _2 \chi _6^2 \chi _7-2 \chi _5
   \chi _6^2 \chi _7-2 \chi _6^2 \chi _7-2 \chi _2 \chi _7\nn \\ &+&3 \chi _3 \chi _7-\chi _2 \chi _5 \chi _7+\chi _3 \chi _5 \chi _7+3 \chi _5 \chi _7+2 \chi _4 \chi _6 \chi _7+2 \chi _7.  \nn \\
\label{eq:e7char}
\eea
\end{appendix}


\bibliography{miabiblio}

\def\cprime{$'$} \def\cprime{$'$}
\begin{bibdiv}
\begin{biblist}

\bib{antoniou2020saito}{article}{
      author={Antoniou, G.},
      author={Feigin, M.},
      author={Strachan, I. A.~B.},
       title={The {S}aito determinant for {C}oxeter discriminant strata},
        date={2020},
      eprint={2008.10133},
}

\bib{MR1387484}{article}{
      author={Arnold, V.~I.},
       title={Topological classification of complex trigonometric polynomials
  and the combinatorics of graphs with an identical number of vertices and
  edges},
        date={1996},
        ISSN={0374-1990},
     journal={Funktsional. Anal. i Prilozhen.},
      volume={30},
      number={1},
       pages={1\ndash 17, 96},
         url={https://doi.org/10.1007/BF02383392},
      review={\MR{1387484}},
}

\bib{MR777682}{book}{
      author={Arnold, V.~I.},
      author={Gusein-Zade, S.~M.},
      author={Varchenko, A.~N.},
       title={Singularities of differentiable maps. {V}ol. {I}},
      series={Monographs in Mathematics},
   publisher={Birkh\"{a}user Boston, Inc., Boston, MA},
        date={1985},
      volume={82},
        ISBN={0-8176-3187-9},
         url={https://doi.org/10.1007/978-1-4612-5154-5},
        note={The classification of critical points, caustics and wave fronts,
  Translated from the Russian by Ian Porteous and Mark Reynolds},
      review={\MR{777682}},
}

\bib{Borot:2015fxa}{article}{
      author={Borot, G.},
      author={Brini, A.},
       title={{Chern--Simons theory on spherical Seifert manifolds, topological
  strings and integrable systems}},
        date={2018},
     journal={{Adv.~Theor.~Math.~Phys.}},
      volume={{22}},
      number={{4}},
       pages={{305\ndash 394}},
      eprint={1506.06887},
}

\bib{Bouchard:2007hi}{article}{
      author={Bouchard, V.},
      author={Marino, M.},
       title={{Hurwitz numbers, matrix models and enumerative geometry}},
        date={2008},
     journal={Proc. Symp. Pure Math.},
      volume={78},
       pages={263\ndash 283},
      eprint={0709.1458},
}

\bib{Brini:2017gfi}{article}{
      author={Brini, A.},
       title={{$E_8$ spectral curves}},
        date={2020},
        ISSN={0024-6115},
     journal={Proc. Lond. Math. Soc. (3)},
      volume={121},
      number={4},
       pages={954\ndash 1032},
      eprint={1711.05958},
      review={\MR{4105791}},
}

\bib{Brini:2019agj}{article}{
      author={Brini, A.},
       title={{Exterior powers of the adjoint representation and the Weyl ring
  of $E_8$}},
        date={2020},
     journal={J. Algebra},
      volume={551},
       pages={301\ndash 341},
      eprint={1902.04380},
}

\bib{Brini:2013zsa}{article}{
      author={Brini, A.},
      author={Cavalieri, R.},
      author={Ross, D.},
       title={{Crepant resolutions and open strings}},
        date={2019},
     journal={J. Reine Angew. Math.},
      volume={755},
       pages={191\ndash 245},
      eprint={1309.4438},
}

\bib{Brini:2021wrm}{article}{
      author={Brini, A.},
      author={Osuga, K.},
       title={{Five-dimensional gauge theories and the local B-model}},
        date={2022},
     journal={Lett. Math. Phys.},
      volume={112},
       pages={44},
}

\bib{Caporaso:2006gk}{article}{
      author={Caporaso, N.},
      author={Griguolo, L.},
      author={Marino, M.},
      author={Pasquetti, S.},
      author={Seminara, D.},
       title={{Phase transitions, double-scaling limit, and topological
  strings}},
        date={2007},
     journal={Phys. Rev.},
      volume={D75},
       pages={046004},
      eprint={hep-th/0606120},
}

\bib{MR2246697}{article}{
      author={Carlet, G.},
       title={The extended bigraded {T}oda hierarchy},
        date={2006},
        ISSN={0305-4470},
     journal={J. Phys. A},
      volume={39},
      number={30},
       pages={9411\ndash 9435},
         url={http://dx.doi.org/10.1088/0305-4470/39/30/003},
      review={\MR{2246697 (2007g:37066)}},
}

\bib{Cheng:2019qjo}{article}{
      author={Cheng, J.},
      author={Milanov, T.},
       title={{The extended D-Toda hierarchy}},
        date={2021},
        ISSN={1022-1824},
     journal={Selecta Math. (N.S.)},
      volume={27},
      number={2},
       pages={Paper No. 24, 85},
      eprint={1909.12735},
         url={https://doi.org/10.1007/s00029-021-00646-1},
      review={\MR{4247924}},
}

\bib{Crescimanno:1994eg}{article}{
      author={Crescimanno, M.~J.},
      author={Taylor, W.},
       title={{Large $N$ phases of chiral QCD in two-dimensions}},
        date={1995},
     journal={Nucl. Phys. B},
      volume={437},
       pages={3\ndash 24},
      eprint={hep-th/9408115},
}

\bib{Dubrovin:1992eu}{article}{
      author={Dubrovin, B.},
       title={{Hamiltonian formalism of Whitham type hierarchies and
  topological Landau-Ginsburg models}},
        date={1992},
     journal={Commun.Math.Phys.},
      volume={145},
       pages={195\ndash 207},
}

\bib{Dubrovin:1992dz}{article}{
      author={Dubrovin, B.},
       title={{Integrable systems in topological field theory}},
        date={1992},
     journal={Nucl. Phys.},
      volume={B379},
       pages={627\ndash 689},
}

\bib{Dubrovin:1994hc}{article}{
      author={Dubrovin, B.},
       title={Geometry of {$2$}{D} topological field theories},
        date={1994},
     journal={{\rm in} ``Integrable systems and quantum groups'' ({M}ontecatini
  {T}erme, 1993), Lecture Notes in Math.},
      volume={1620},
       pages={120\ndash 348},
      eprint={hep-th/9407018},
      review={\MR{MR1397274 (97d:58038)}},
}

\bib{Dubrovin:1993nt}{article}{
      author={Dubrovin, B.},
       title={{Differential geometry of the space of orbits of a Coxeter
  group}},
        date={1998},
     journal={Surveys in Differential Geometry},
      volume={4},
      number={1},
       pages={181\ndash 211},
      eprint={hep-th/9303152},
}

\bib{Dubrovin:1998fe}{article}{
      author={Dubrovin, B.},
       title={{Painleve transcendents and two-dimensional topological field
  theory}},
        date={1999},
     journal={CRM Ser. Math. Phys.},
       pages={287\ndash 412},
      eprint={math/9803107},
      review={\MR{1713580 (2001h:53131)}},
}

\bib{MR2070050}{article}{
      author={Dubrovin, B.},
       title={On almost duality for {F}robenius manifolds, in ``{G}eometry,
  topology, and mathematical physics"},
        date={2004},
     journal={Amer. Math. Soc. Transl. Ser. 2},
      volume={212},
       pages={75\ndash 132},
      review={\MR{2070050 (2005e:53142)}},
}

\bib{Dubrovin:2015wdx}{article}{
      author={Dubrovin, B.},
      author={Strachan, I.},
      author={Zhang, Y.},
      author={Zuo, D.},
       title={{Extended affine Weyl groups of BCD-type: Their Frobenius
  manifolds and Landau\textendash{}Ginzburg superpotentials}},
        date={2019},
     journal={Adv. Math.},
      volume={351},
       pages={897\ndash 946},
      eprint={1510.08690},
}

\bib{MR1606165}{article}{
      author={Dubrovin, B.},
      author={Zhang, Y.},
       title={Extended affine {W}eyl groups and {F}robenius manifolds},
        date={1998},
        ISSN={0010-437X},
     journal={Compositio Math.},
      volume={111},
      number={2},
       pages={167\ndash 219},
         url={http://dx.doi.org/10.1023/A:1000258122329},
      review={\MR{1606165}},
}

\bib{DuninBarkowski:2012bw}{article}{
      author={Dunin-Barkowski, P.},
      author={Orantin, N.},
      author={Shadrin, S.},
      author={Spitz, L.},
       title={{Identification of the Givental formula with the spectral curve
  topological recursion procedure}},
        date={2014},
     journal={Commun. Math. Phys.},
      volume={328},
       pages={669\ndash 700},
      eprint={1211.4021},
}

\bib{Ekedahl:LL}{article}{
      author={Ekedahl, T.},
      author={Lando, S.},
      author={Shapiro, M.},
      author={Vainshtein, A.},
       title={{Hurwitz numbers and intersections on moduli spaces of curves}},
        date={2001},
     journal={Invent. math.},
      volume={146},
       pages={297\ndash 327},
}

\bib{Fock:2014ifa}{article}{
      author={Fock, V.~V.},
      author={Marshakov, A.},
       title={{Loop groups, Clusters, Dimers and Integrable systems}},
        date={2014},
      eprint={1401.1606},
}

\bib{MR260752}{article}{
      author={Fulton, W.},
       title={Hurwitz schemes and irreducibility of moduli of algebraic
  curves},
        date={1969},
        ISSN={0003-486X},
     journal={Ann. of Math. (2)},
      volume={90},
       pages={542\ndash 575},
         url={https://doi.org/10.2307/1970748},
      review={\MR{260752}},
}

\bib{MR1153249}{book}{
      author={Fulton, W.},
      author={Harris, J.},
       title={Representation theory},
      series={Graduate Texts in Mathematics},
   publisher={Springer-Verlag, New York},
        date={1991},
      volume={129},
        ISBN={0-387-97527-6; 0-387-97495-4},
         url={https://doi.org/10.1007/978-1-4612-0979-9},
        note={A first course, Readings in Mathematics},
      review={\MR{1153249}},
}

\bib{MR1181077}{article}{
      author={Goulden, I.~P.},
      author={Jackson, D.~M.},
       title={The combinatorial relationship between trees, cacti and certain
  connection coefficients for the symmetric group},
        date={1992},
        ISSN={0195-6698},
     journal={European J. Combin.},
      volume={13},
      number={5},
       pages={357\ndash 365},
         url={https://doi.org/10.1016/S0195-6698(05)80015-0},
      review={\MR{1181077}},
}

\bib{MR1763948}{article}{
      author={Goulden, I.~P.},
      author={Jackson, D.~M.},
      author={Vainshtein, A.},
       title={The number of ramified coverings of the sphere by the torus and
  surfaces of higher genera},
        date={2000},
        ISSN={0218-0006},
     journal={Ann. Comb.},
      volume={4},
      number={1},
       pages={27\ndash 46},
         url={https://doi.org/10.1007/PL00001274},
      review={\MR{1763948}},
}

\bib{MR1649966}{article}{
      author={Goupil, A.},
      author={Schaeffer, G.},
       title={Factoring {$n$}-cycles and counting maps of given genus},
        date={1998},
        ISSN={0195-6698},
     journal={European J. Combin.},
      volume={19},
      number={7},
       pages={819\ndash 834},
         url={https://doi.org/10.1006/eujc.1998.0215},
      review={\MR{1649966}},
}

\bib{Ito:1997ur}{article}{
      author={Ito, K.},
      author={Yang, S-K.},
       title={{Flat coordinates, topological Landau-Ginzburg models and the
  Seiberg-Witten period integrals}},
        date={1997},
     journal={Phys. Lett. B},
      volume={415},
       pages={45\ndash 53},
      eprint={hep-th/9708017},
}

\bib{Krichever:1992qe}{article}{
      author={Krichever, I.~M.},
       title={{The tau function of the universal Whitham hierarchy, matrix
  models and topological field theories}},
        date={1994},
     journal={Commun. Pure Appl. Math.},
      volume={47},
       pages={437},
      eprint={hep-th/9205110},
}

\bib{MR1918855}{article}{
      author={Lando, S.~K.},
       title={Ramified coverings of the two-dimensional sphere and intersection
  theory in spaces of meromorphic functions on algebraic curves},
        date={2002},
        ISSN={0042-1316},
     journal={Uspekhi Mat. Nauk},
      volume={57},
      number={3(345)},
       pages={29\ndash 98},
         url={https://doi.org/10.1070/RM2002v057n03ABEH000511},
      review={\MR{1918855}},
}

\bib{MR1724267}{article}{
      author={Lando, S.~K.},
      author={Zvonkine, D.},
       title={On multiplicities of the {L}yashko--{L}ooijenga mapping on strata
  of the discriminant},
        date={1999},
        ISSN={0016-2663},
     journal={Funct. Anal. Appl.},
      volume={33},
      number={3},
       pages={178\ndash 188 (2000)},
      review={\MR{1724267}},
}

\bib{MR2324558}{article}{
      author={Lando, S.~K.},
      author={Zvonkine, D.},
       title={Counting ramified converings and intersection theory on spaces of
  rational functions. {I}. {C}ohomology of {H}urwitz spaces},
        date={2007},
        ISSN={1609-3321},
     journal={Mosc. Math. J.},
      volume={7},
      number={1},
       pages={85\ndash 107, 167},
         url={https://doi.org/10.17323/1609-4514-2007-7-1-85-107},
      review={\MR{2324558}},
}

\bib{Lerche:1996an}{article}{
      author={Lerche, W.},
      author={Warner, N.~P.},
       title={{Exceptional SW geometry from ALE fibrations}},
        date={1998},
     journal={Phys. Lett.},
      volume={B423},
       pages={79\ndash 86},
      eprint={hep-th/9608183},
}

\bib{MR0422675}{article}{
      author={Looijenga, E.},
       title={The complement of the bifurcation variety of a simple
  singularity},
        date={1974},
     journal={Inv. Math.},
      volume={23},
       pages={105\ndash 116},
}

\bib{MR542251}{article}{
      author={Lyashko, O.~V.},
       title={The geometry of bifurcation diagrams},
        date={1979},
        ISSN={0042-1316},
     journal={Uspekhi Mat. Nauk},
      volume={34},
      number={3(207)},
       pages={205\ndash 206},
      review={\MR{542251}},
}

\bib{MR1182413}{article}{
      author={McDaniel, A.},
      author={Smolinsky, L.},
       title={A {L}ie-theoretic {G}alois theory for the spectral curves of an
  integrable system. {I}},
        date={1992},
        ISSN={0010-3616},
     journal={Comm. Math. Phys.},
      volume={149},
      number={1},
       pages={127\ndash 148},
         url={http://projecteuclid.org/euclid.cmp/1104251141},
      review={\MR{1182413}},
}

\bib{MR1401779}{article}{
      author={McDaniel, A.},
      author={Smolinsky, L.},
       title={A {L}ie-theoretic {G}alois theory for the spectral curves of an
  integrable system. {II}},
        date={1997},
        ISSN={0002-9947},
     journal={Trans. Amer. Math. Soc.},
      volume={349},
      number={2},
       pages={713\ndash 746},
         url={http://dx.doi.org/10.1090/S0002-9947-97-01853-9},
      review={\MR{1401779}},
}

\bib{Milanov:2014pma}{article}{
      author={Milanov, T.},
      author={Shen, Y.},
      author={Tseng, H-H.},
       title={{Gromov--Witten theory of Fano orbifold curves and ADE-Toda
  Hierarchies}},
        date={2016},
        ISSN={1465-3060},
     journal={Geom. Topol.},
      volume={20},
      number={4},
       pages={2135\ndash 2218},
      eprint={1401.5778},
         url={https://doi.org/10.2140/gt.2016.20.2135},
      review={\MR{3548465}},
}

\bib{Nekrasov:1996cz}{article}{
      author={Nekrasov, N.},
       title={{Five dimensional gauge theories and relativistic integrable
  systems}},
        date={1998},
     journal={Nucl. Phys.},
      volume={B531},
       pages={323\ndash 344},
      eprint={hep-th/9609219},
}

\bib{MR3268770}{article}{
      author={Norbury, P.},
      author={Scott, N.},
       title={Gromov--{W}itten invariants of {$\mathbb{P}^1$} and
  {E}ynard--{O}rantin invariants},
        date={2014},
        ISSN={1465-3060},
     journal={Geom. Topol.},
      volume={18},
      number={4},
       pages={1865\ndash 1910},
         url={http://dx.doi.org/10.2140/gt.2014.18.1865},
      review={\MR{3268770}},
}

\bib{otani2023number}{article}{
      author={Otani, T.},
      author={Shiraishi, Y.},
      author={Takahashi, A.},
       title={The number of full exceptional collections modulo spherical
  twists for extended {D}ynkin quivers},
        date={2023},
      eprint={2308.04031},
}

\bib{MR1944086}{article}{
      author={Panov, D.},
      author={Zvonkine, D.},
       title={Counting meromorphic functions with critical points of large
  multiplicities},
        date={2002},
        ISSN={0373-2703},
     journal={Zap. Nauchn. Sem. S.-Peterburg. Otdel. Mat. Inst. Steklov.
  (POMI)},
      volume={292},
      number={Teor. Predst. Din. Sist. Komb. i Algoritm. Metody. 7},
       pages={92\ndash 119, 179},
         url={https://doi.org/10.1007/s10958-005-0095-1},
      review={\MR{1944086}},
}

\bib{MR2672302}{article}{
      author={Rossi, P.},
       title={Gromov--{W}itten theory of orbicurves, the space of
  tri-polynomials and symplectic field theory of {S}eifert fibrations},
        date={2010},
        ISSN={0025-5831},
     journal={Math. Ann.},
      volume={348},
      number={2},
       pages={265\ndash 287},
      eprint={0808.2626},
         url={https://doi.org/10.1007/s00208-009-0471-0},
      review={\MR{2672302}},
}

\bib{MR2440696}{article}{
      author={Szablikowski, B.~M.},
      author={Blaszak, M.},
       title={Dispersionful analog of the {W}hitham hierarchy},
        date={2008},
        ISSN={0022-2488},
     journal={J. Math. Phys.},
      volume={49},
      number={8},
       pages={082701, 20},
         url={https://doi.org/10.1063/1.2970774},
      review={\MR{2440696}},
}

\bib{Williams:2012fz}{article}{
      author={Williams, H.},
       title={{Double Bruhat Cells in Kac--Moody Groups and Integrable
  Systems}},
        date={2013},
     journal={Lett. Math. Phys.},
      volume={103},
       pages={389\ndash 419},
      eprint={1204.0601},
}

\end{biblist}
\end{bibdiv}

\end{document}